\documentclass[a4paper,reqno,12pt]{amsart}
\usepackage[left=1 in, right=1 in,top=1 in, bottom=1 in]{geometry}
\usepackage{amsfonts}
\usepackage{amssymb}
\usepackage{amsthm}
\usepackage{amsmath}
\usepackage{mathrsfs}
\usepackage{color}
\usepackage{enumerate}
\usepackage{hyperref}
\usepackage[numbers,sort&compress]{natbib}
\usepackage{pdfsync}
\usepackage{esint}
\usepackage{graphicx}
\usepackage{float}
\usepackage{caption}
\usepackage{subfigure}
\allowdisplaybreaks

\makeatother

\numberwithin{equation}{section}

\newtheorem{theorem}{Theorem}[section]
\newtheorem{definition}[theorem]{Definition}

\newtheorem{lemma}[theorem]{Lemma}
\newtheorem{remark}[theorem]{Remark}

\newtheorem{corollary}[theorem]{Corollary}
\numberwithin{equation}{section}

\newcommand*{\Id}{\ensuremath{\mathrm{Id}}}
\newcommand*{\supp}{\ensuremath{\mathrm{supp\,}}}

\newcommand{\aint}{{\fint}}

\newcommand{\T}{{\mathbb{T}}}

\newcommand{\rr}{\mathring{R}}
\newcommand{\ru}{\mathring{R}_q^u}
\newcommand{\rb}{\mathring{R}_q^B}
\newcommand{\p}{\partial}
\renewcommand{\P}{\mathbb{P}}
\newcommand{\mh}{\mathcal{H}}

\renewcommand{\div}{{\mathrm{div}}}
\newcommand{\curl}{{\mathrm{curl}}}

\renewcommand{\u}{{u_q}}
\newcommand{\h}{{B_q}}

\renewcommand{\d}{{\rm d}}

\newcommand{\g}{g_{(\tau)}}

\newcommand{\norm}[1]{\lVert#1\rVert}
\newcommand{\abs}[1]{|#1|}

\newcommand{\la}{\lambda_q}
\newcommand{\laq}{\lambda_{q+1}}

\newcommand{\rs}{r_{\perp}}
\newcommand{\rp}{r_{\parallel}}
\newcommand{\va}{\varepsilon}

\newcommand{\wo}{w_{q+1}^{(o)}}
\newcommand{\dqo}{d_{q+1}^{(o)}}
\newcommand{\wdc}{\wt D_{(k)}^c}

\newcommand{\R}{{\mathbb R}}
\newcommand{\lbb}{\overline{\lambda}}

\def\a{{\alpha}}
\def\vf{{\varphi}}
\def\lbb{\lambda}
\def\wt{\widetilde}
\def\9{{\infty}}
\def\ve{{\varepsilon}}
\def\na{{\nabla}}
\def\bbr{{\mathbb{R}}}

\def\({\left(}
\def\){\right)}

\begin{document}
	
\title[] {Non-uniqueness of weak solutions to 3D magnetohydrodynamic equations}

\author{Yachun Li}
\address{School of Mathematical Sciences, CMA-Shanghai, MOE-LSC, and SHL-MAC,  Shanghai Jiao Tong University, China.}
\email[Yachun Li]{ycli@sjtu.edu.cn}
\thanks{}

\author{Zirong Zeng}
\address{School of Mathematical Sciences, Shanghai Jiao Tong University, China.}
\email[Zirong Zeng]{beckzzr@sjtu.edu.cn}
\thanks{}

\author{Deng Zhang}
\address{School of Mathematical Sciences, CMA-Shanghai, Shanghai Jiao Tong University, China.}
\email[Deng Zhang]{dzhang@sjtu.edu.cn}
\thanks{}



\begin{abstract}
We prove the non-uniqueness of weak solutions to 3D magnetohydrodynamic (MHD for short) equations.
The constructed weak solutions do not conserve the magnetic helicity
and can be close to any given smooth, divergence-free and mean-free velocity
and magnetic fields.
Furthermore, we prove that the weak solutions constructed by Beekie-Buckmaster-Vicol \cite{bbv20}
for the ideal MHD
can be obtained as a strong vanishing viscosity and resistivity limit of a sequence of weak solutions
to  MHD  equations.
This shows that,
in contrast to the weak ideal limits,
Taylor's conjecture does not hold along the vanishing viscosity and resistivity limits.
Unlike in the context of the NSE \cite{bv19b} and
the ideal MHD \cite{bbv20},
new types of velocity and magnetic flows,
featuring both the refined spatial and temporal intermittency,  are constructed
to respect the geometry of MHD and to control the strong viscosity and resistivity.
Compatible algebraic structure is derived in the convex integration scheme.
More interestingly,
the new intermittent flows indeed enable us to
prove the aforementioned results for
the hyper-viscous and hyper-resistive MHD equations
up to the sharp exponent $5/4$,
which coincides exactly with the Lions exponent for 3D hyper-viscous
NSE.
\end{abstract}

\maketitle

\section{Introduction and main results}

\subsection{Introduction} \label{Subsec-intro}

The problem of non-uniqueness of weak solutions to hydrodynamic models
attract a lot of attention in the last decade.
Significant progresses have been made towards the Euler equations
and Navier-Stokes equations
based on the convex integration method,
which originates from the work by Nash \cite{nash54} concerning the construction of $C^1$ isometric immersions of $S^n$ into $\R^m$ $(m\geq n+2)$ and
was introduced by De Lellis-Sz\'{e}kelyhidi
in the pioneering paper \cite{dls09} to 3D Euler equations.

Among various hydrodynamic models,
the viscous and resistive MHD system is a canonical macroscopic model to describe the motion of conductive fluid,
such as plasma or liquid metals,
under a complicated interaction between the electromagnetic phenomenon and fluid dynamical phenomenon
(see \cite{LQ14}).
We refer the reader to  \cite{b1993,b2003,d2001,ST83}
for more physical interpretations of the MHD system.

The applicability of convex integration scheme to MHD equations
is at an early stage (see \cite{bv21} and the literature review below),
and it is the objective of the present article.
As a matter of fact,
because of the strong coupling of the velocity and magnetic fields,
the convex integration method requires the construction of both the
velocity and magnetic fields
in an appropriate way, in order to respect the geometry of MHD.
This particular structural requirement,
however,
limits the oscillation directions and thus the spatial intermittency
of flows,
which results in a substantial difficulty to control the strong viscosity and resistivity.

More precisely,
we are concerned with the
three-dimensional MHD system on the torus $\T^3:=[-\pi,\pi]^3$,
\begin{equation}\label{equa-MHD}
	\left\{\aligned
	&\p_tu -\nu_1\Delta u+(u\cdot \nabla )u -(B\cdot \nabla )B + \nabla P=0,  \\
	&\p_tB-\nu_2\Delta B+(u\cdot \nabla )B- (B\cdot \nabla )u  =0, \\
	& \div u = 0, \quad \div B = 0,
	\endaligned
	\right.
\end{equation}
where  $u=(u_1,u_2,u_3)^\top(t,x)\in \R^3 $, $B=(B_1,B_2,B_3)^\top(t,x)\in \R^3$
and $P=P(t,x)\in \R$
correspond to the  velocity field,
magnetic field and pressure of the fluid, respectively,
and $\nu_1, \nu_2\geq 0$ are the viscous and resistive coefficients, respectively.
We note that,
in the case without magnetic fields,
\eqref{equa-MHD} reduces to the  incompressible Navier-Stokes equation
(NSE for short).

Another important model is the ideal MHD system,
where the viscous and resistive coefficients vanish, namely,
\begin{equation}\label{equa-IMHD}
	\left\{\aligned
	&\p_tu+(u\cdot \nabla )u -(B\cdot \nabla )B + \nabla P=0,  \\
	&\p_tB+(u\cdot \nabla )B- (B\cdot \nabla )u  =0, \\
	& \div u = 0, \quad \div B = 0.
	\endaligned
	\right.
\end{equation}
The  incompressible ideal MHD system is the classical macroscopic model
coupling the Maxwell equations to the evolution of an electrically conducting incompressible fluid
(cf. \cite{b1993, d2001}).
In the  case $B\equiv 0$,
\eqref{equa-IMHD} then reduces to the Euler equations.

The well-posedness problem of NSE and  MHD
has been extensively studied in the literature. For the initial data with finite energy,
the existence of global weak solution $u$ to  NSE was first proved by Leray \cite{leray1934} in 1934 and later
by Hopf \cite{hopf1951} in 1951 in bounded domains,
which satisfies $u\in C_{weak}\big([0,+\infty);L^2(\Omega)\big)\cap L^2\big([0,+\infty);\dot{H}^1(\Omega)\big)$,
$\Omega = \bbr^3$ or $\mathbb{T}^3$,
and obeys the following energy inequality
\begin{align}\label{nsenergy}
	\|u(t)\|_{L^2}^2 + 2 \nu \int_{t_0}^t \|\na u(s)\|_{L^2}^2 \d s
	\leq \|u(t_0)\|_{L^2}^2
\end{align}
for any $t>0$ and a.e. $t_0\geq 0$.
The existence of Leray-Hopf solutions to generalized MHD equations
was proved by Wu \cite{wu03}.

The uniqueness  of Leray-Hopf solutions is
considered as one of the most important issues in the mathematical fluid mechanics.
The well known Lady\v{z}henskaya-Prodi-Serrin condition states that,
the Leray-Hopf solution to NSE is regular and thus unique in the space $L^\gamma_tL^p_x$
for ${2}/{\gamma}+{3}/{p}\leq 1$ with $p>3$ (see \cite{kl57,prodi59,serrin62}).
The endpoint case $L^\9_tL^3_x$ was solved by Escauriaza-Seregin-\v{S}ver\'{a}k \cite{iss03}.
The well-posedness of NSE for small initial data in the space ${\rm BMO^{-1}}$ was proved by Koch-Tataru \cite{kt01}.
Moreover, Lions  \cite{lions69} obtained the well-posedness for the hyper-viscous NSE
where the exponent of viscosity is bigger than $ 5/4$.
Concerning the MHD type equations,
the uniqueness of global classical solutions
was proved by Wu \cite{wu03},
when the exponents of viscosity and resistivity are bigger than ${1}/{2}+{n}/{4}$,
which is exactly the Lions exponent for the hyper-viscous NSE in dimension three $n=3$.
Moreover, the uniqueness in the case analogous to the Lady\v{z}henskaya-Prodi-Serrin condition
was proved by
Zhou \cite{zhou07}.
See also \cite{Lin98,ST83,tao09,Vasseur07,wu08} and references therein.

Concerning the non-uniqueness problem of weak solutions,
in the pioneering paper
\cite{dls09} De Lellis-Sz\'{e}kelyhidi introduced
the convex integration scheme to the 3D Euler equations  and  proved the existence of bounded solutions
with compact support in space-time (see also \cite{dls10}).
Afterwards, significant progress towards the famous Onsager conjecture for the 3D Euler equations
was made by De Lellis-Sz\'{e}kelyhidi \cite{dls13,dls14}.
The resolution of the Onsager conjecture was finally solved
in $C_{x,t}^{\beta}\,(0<\beta<1/3)$ by Isett \cite{I18}, and
by Buckmaster-De Lellis-Sz\'{e}kelyhidi-Vicol \cite{bdsv19} for dissipative solutions.
We also refer the reader to the papers \cite{cvy21,ckms21,lxx16} for Euler equations,
\cite{bcd21,cl21,ms18,ms20} for transport equations,
and \cite{bsv19} for SQG equations.
See also the surveys \cite{dls17,bv19r,bv21} on convex integration methods.

A recent break-through was obtained by Buckmaster-Vicol \cite{bv19b}
for the 3D incompressible NSE.
The non-uniqueness of weak solutions to NSE was proved in  \cite{bv19b}
by using an $L_x^2$-based intermittent convex integration scheme,
which is efficient to handle the viscosity of NSE.
Later, Luo-Titi \cite{lt20} proved the non-uniqueness results for the 3D hyper-viscous NSE,
whenever the exponent of viscosity is less than the Lions exponent $5/4$.
Moreover, Luo-Qu \cite{lq20} proved the non-uniqueness for the 2D hypoviscous NSE.
The non-uniqueness results
for the stationary NSE were proved by Cheskidov-Luo \cite{cl20,luo19}.
We also refer to the works by Jia-\v{S}ver\'{a}k \cite{js14,js15} for
another method for the non-uniqueness issue of Leray-Hopf
solutions in $L^\9_tL^3_x$ under a certain assumption for the linearized Navier-Stokes operator,
and Colombo-De Lellis-De Rosa \cite{CDR18} and De Rosa \cite{DR19}
for the non-uniqueness of Leray weak solutions to hypoviscous NSE with small viscosity.

Regarding the ideal MHD system \eqref{equa-IMHD},
it posses a number of global invariants:
	\begin{enumerate}
		\item[$\bullet$] The total  energy:\ \
		$\displaystyle \mathcal{E}(t) =\frac12\int_{\T^3}|u(t,x)|^2+|B(t,x)|^2 \d x$;
		\item [$\bullet$] The cross helicity: \ \
		$\displaystyle \mathcal{H}_{\omega,B}(t) =\int_{\T^3}u(t,x)\cdot B(t,x) \d x$;
		\item [$\bullet$] The magnetic helicity: \ \
		$\displaystyle \mathcal{H}_{B,B}(t):=\int_{\T^3}A(t,x)\cdot B(t,x) \d x$.
		\end{enumerate}	
Here $A$ is  a  mean-free periodic vector field satisfying $\curl A=B$.

In analogy with the Onsager conjecture for Euler equations,
it was conjectured in \cite{bv21} that,
$L^3_{t,x}$ (resp. $L^3_t B^0_{3,c_0,x}$) is the critical space for the conservation of the magnetic helicity.
More precisely,
any weak solutions belonging to $L^3_{t,x}$ (resp. $L^3_t B^0_{3,c_0,x}$) conserve the magnetic helicity
({\it the rigidity part}),
while below the threshold there exist weak solutions violating the magnetic helicity conservation
({\it the flexible part}).
Moreover, the space $C^{1/3}_{t,x}$ (resp. $L^3_t B^{1/3}_{3,\9,x}$)
was conjectured to be the threshold for the
conservation of the total energy and the cross helicity
(see Remark \ref{Rem-viscosity-limit} (ii) below).

These conjectures indeed lie in the general scope of the important issue
to identify the {\it critical/threshold regularity} for the conservation laws.
We  refer to Klainerman \cite{K17} for the general formulations of  different thresholds
for supercritical Hamiltonian evolution equations, including the 3D MHD equations.

On the rigidity side,
the magnetic helicity conservation for the 3D ideal MHD was proved
in \cite{KL07,A09,FL18} and \cite{CKS97},
respectively,
for the critical spaces $L^{3}_{t,x}$ and $L^3_tB^\a_{3,\9,x}$ with $\a>0$.
The corresponding rigidity results for the total energy conservation
were proved in \cite{CKS97}.
On the flexible side,
weak solutions with non-trivial energy
and vanishing magnetic helicity were first given in \cite{BFL15},
by embedding the ideal MHD into the $2\frac 12$D Euler flow via a symmetry assumption.
Wild solutions with compact support in space-time were first constructed by
Faraco-Lindberg-Sz\'{e}kelyhidi \cite{fls21},
via the $L^\9_{t,x}$ convex integration scheme.
The constructed solutions in \cite{fls21}
violate the energy conservation, while the magnetic helicity is still conserved
and thus vanishes due to the compact support in time of solutions.

The delicate point,
as observed by Beekie-Buckmaster-Vicol \cite{bbv20},
is that the magnetic helicity is conserved under much milder regularity conditions.
As a matter of fact,
the magnetic helicity is commonly expected in the plasma physics to be conserved in the infinite conductivity limit,
known as {\it Taylor's conjecture} (\cite{taylor74,taylor86}).
This conjecture was proved by Faraco and Lindberg \cite{fl20}
under the {\it weak ideal limits},
namely, the weak limits of Leray solutions to MHD \eqref{equa-MHD}.

For general weak solutions,
Beekie-Buckmaster-Vicol \cite{bbv20} first constructed weak solutions
in $C_tL^2_x$ violating the magnetic helicity conservation.
Unlike in \cite{fls21}, the proof of \cite{bbv20} relies on the intermittent convex integration scheme,
in which the intermittent shear flows are constructed to respect the geometry of MHD.
The flexible part of the Onsager-type conjecture
for the magnetic helicity
was recently solved by Faraco-Lindberg-Sz\'{e}kelyhidi \cite{fls21.2},
based on the convex integration via staircase laminates.

See also Dai \cite{dai18} for the Hall MHD system
where the Hall nonlinearity takes the dominant effect,
and Feireisl-Li \cite{FL20} for the ill-posedness of the MHD system
where the fluid is compressible, inviscid and magnetically resistive.
\vspace{1ex}

As mentioned above,
the strong coupling of the velocity and magnetic fields in MHD equations
limits the intermittency of flows in the convex integration scheme.
More precisely, for the NSE, the flows constructed in \cite{bv19b}
have 3D spatial intermittency,
which is essential to control the viscosity in NSE.
Regarding the ideal MHD,
the intermittent shear flows constructed in \cite{bbv20} have 1D intermittency
which only permit to control the hypo-viscosity and hypo-resistivity $(-\Delta)^\alpha$
with $\alpha\in(0,1/2)$.
Hence, both the above intermittent flows are not applicable to the
viscous and resistive MHD \eqref{equa-MHD} under consideration here.

The main purpose of this paper is to
address the non-uniqueness problem of weak solutions for the canonical MHD system \eqref{equa-MHD}.
By constructing a new class of intermittent velocity and magnetic flows
adapted to the geometry of MHD,
we are able to control the strong viscosity and resistivity.
More interestingly,
the new class of intermittent flows even enables us to
achieve the non-uniqueness of weak solutions to a more general class of MHD type equations,
for all the exponents of viscosity and resistivity less than the sharp Lions exponent,
namely,
\begin{equation}\label{equa-gMHD}
	\left\{\aligned
	&\p_tu+ \nu_1 (-\Delta)^{\alpha_1} u+(u\cdot \nabla )u -(B\cdot \nabla )B + \nabla P=0,  \\
	&\p_tB+ \nu_2 (-\Delta)^{\alpha_2} B+(u\cdot \nabla )B- (B\cdot \nabla )u  =0, \\
	& \div u = 0, \quad \div B = 0,
	\endaligned
	\right.
\end{equation}
where $\alpha_1,\alpha_2\in (0, 5/4)$,  $\nu_1, \nu_2\geq 0$.
We note that,
the value $5/4$ coincides with the Lions exponent for the well-posedness of the hyper-viscous NSE \cite{lions69},
and with the exponent $1/2 + n/4$ for the well-posedness of generalized MHD in dimension three $n=3$ \cite{wu03}.

The constructed weak solutions
can be close to any given  divergence-free, mean-free velocity and magnetic fields.
In particular,
the weak solutions
live in the space  $L^\gamma_tW^{s,p}_x$
with $0\leq s< {2}/{p}+ {1}/{\gamma}- 3/2$, $1\leq p,\gamma<+\9$,
and do not conserve the magnetic helicity.
Thus, this provides more examples
in the spaces $L^\gamma_tW^{s,p}_x$ for the flexible part of
the Onsager-type conjecture for the ideal MHD \eqref{equa-IMHD}.

Another aim of the present work is to  make some progress towards the understanding of
the relationship between the non-uniqueness of solutions to the
ideal MHD \eqref{equa-IMHD} and MHD equations \eqref{equa-MHD} or \eqref{equa-gMHD},
which in turn relates to the physical Taylor conjecture.

As pointed out in  \cite{bbv20,bv21},
the weak solutions to the ideal MHD \eqref{equa-IMHD} with non-trivial magnetic helicity
cannot be obtained as the  weak ideal limits of MHD Leray-Hopf solutions.
Instead, we prove that the weak solutions constructed
by Beekie-Buckmaster-Vicol \cite{bbv20}
can be obtained as a strong {\it vanishing viscosity and resistivity  limit}
of a sequence of weak solutions to the MHD equations  \eqref{equa-gMHD},
where the exponents $\a_i$ can be even larger than one.
An interesting outcome is that,
together with the non-conservative results in the work \cite{bbv20},
this shows that Taylor's conjecture does not hold along the  vanishing viscosity and resistivity limits
of general weak solutions,
in contrast to the weak ideal limits of Leray weak solutions.

\vspace{1ex}
One of the main novelties of our proof
is the construction of the {\it velocity and magnetic flows
with the refined spatial intermittency},
which are designed specifically to have two oscillation directions and
concentrate on small spatial cuboids.
This structure enables us to gain an additional
1D spatial intermittency than that in
the ideal MHD case \cite{bbv20},
and thus to control the viscosity and resistivity $(-\Delta)^{\a_i}$
with $\a_i \in [0,1)$, $i=1,2$.

Another nice feature is that,
the new intermittent flows have almost mutually disjoint supports.
Actually, the intermittent flows have much smaller intersections,
which provide almost 2D spatial intermittency
and thus contribute negligible errors in the convex integration scheme.
This in particular simplifies the control of the nonlinearities
in MHD equations.

In order to pass beyond the border line $\a_i =1$, $i=1,2$,
another key novelty of proof is the construction of the
{\it $L_t^2$-based  temporal intermittent flows} adapted to the structure of MHD.
We introduce the high temporal oscillations in the
velocity and magnetic flows.
Rather than the pointwise analysis in time,
we measure the solutions on average,
namely, in the space $L^2_tL^2_x$,
which is in spirit close to the works \cite{bv19b,cl21}.
In particular, this new temporal intermittency
gains an extra  1D intermittency,
which permits to
control the much stronger viscosity  and resistivity  $(-\Delta)^{\a_i}$,
$\a_i\geq 1$,
and, more interestingly,
even to achieve all the exponents $\a_i\in [1,5/4)$
up to the sharp Lions exponent.

Let us also mention that,
besides the incompressibility correctors in \cite{bbv20},
two new types of temporal correctors adapted to MHD
are also introduced for the velocity and magnetic perturbations,
in order to balance both the high spatial and temporal oscillations
arising from the corresponding concentration functions.
The new velocity flows
and magnetic flows indeed
exhibit a quite compatible algebraic structure in the convex integration scheme.
We expect the new method developed here would be also of interest
in the further understanding of MHD equations.  \\

{\bf Notations.}
For $p\in [1,+\infty]$ and $s\in \R$, we use the following short notations
\begin{align*}
	L^p_x:=L^p_x(\T^3),\quad H^s_x:=H^s_x(\T^3), \quad W^{s,p}_x:=W^{s,p}_x(\T^3).
\end{align*}
The mean of $u \in L^1(\T^n)$ is given by
$
\aint_{\T^n} u \d x =  {|\T^n|}^{-1} \int_{\T^n} u \d x ,
$
where $|\cdot|$ denotes the Lebesgue measure.
For any $p, \gamma\in [1,+\infty]$,
$L^\gamma_tL^p_x$ denotes the usual Banach space $L^\gamma_t(\T;L^p(\T^3))$.
Moreover, given any Banach space $X$,
$C(\T;X)$ denotes  the space of continuous functions from $\mathbb{T}$ to $X$,
equipped with the norm $\|u\|_{C_tX}:=\sup_{t\in \mathbb{T}}\|u(t)\|_X$.
In particular, we write $L^p_{t,x}:= L^p_tL^p_x$
and $C_{t,x}:=C_tC_x$ for simplicity.
Let
\begin{align*}
	\norm{u}_{W^{N,p}_{t,x}}:=\sum_{0\leq m+|\zeta|\leq N} \norm{\p_t^m \na^{\zeta} u}_{L^p_{t,x}}, \ \
     \norm{u}_{C_{t,x}^N}:=\sum_{0\leq m+|\zeta|\leq N}
	  \norm{\p_t^m \na^{\zeta} u}_{C_{t,x}},
\end{align*}
where $\zeta=(\zeta_1,\zeta_2,\zeta_3)$ is the multi-index
and $\na^\zeta:= \partial_{x_1}^{\zeta_1} \partial_{x_2}^{\zeta_2} \partial_{x_3}^{\zeta_3}$.
For any $A\subseteq \T$, set
\begin{align*}
	N_{\va_*}(A):=\{t\in \T:\ \exists s\in A,\ s.t.\ |t-s|\leq \va_*\}.
\end{align*}

We also adapt the notations from \cite{LQ14}.
Let $u,\,v$ be two vector fields,
the corresponding second order tensor product is defined by
\begin{align*}
u\otimes v:=(u_iv_j)_{1\leq i,j\leq 3}.
\end{align*}
For any second-order tensor $A=(a_{ij})_{1\leq i,j\leq 3}$,
set
\begin{align*}
\div A:= \(\sum_{j=1}^3 {\p_{x_j} a_{1j}} ,\sum_{j=1}^3 {\p_{x_j}  a_{2j}},\sum_{j=1}^3 {\p_{x_j}  a_{3j}} \)^\top.
\end{align*}
The right product of a vector field $v=(v_1,v_2,v_3)^\top$ to a second-order tensor $A=(a_{ij})_{1\leq i,j\leq 3}$ is defined by
\begin{align*}
Av:=\(\sum_{j=1}^{3} a_{1j}v_j,\sum_{j=1}^{3} a_{2j}v_j,\sum_{j=1}^{3} a_{3j}v_j \)^\top.
\end{align*}
In particular, for any scalar function $f$ and second-order tensor $A=(a_{ij})_{1\leq i,j\leq 3}$,
one has the Leibniz rule
\begin{align*}
	\div(fA)= f\div A+A\nabla f.
\end{align*}

We  use the notation $a\lesssim b$, which means that $a\leq Cb$ for some constant $C>0$.

\subsection{Formulation of main results}

To begin with, let us formulate precisely   the definition of weak solutions to \eqref{equa-gMHD}.
\begin{definition}[Weak solution]\label{weaksolu}
	We say that $(u, B) \in L^2(\T; L^2(\T^3))$ is a weak solution to
the  MHD equations \eqref{equa-gMHD} if
	\begin{itemize}
		\item For all $t\in\T$, $(u(t,\cdot), B(t,\cdot))$ are divergence free in the sense of distributions and have zero spatial mean.
		\item Equation \eqref{equa-gMHD} holds in the sense of distributions, i.e.,
		for any divergence-free test functions $\varphi  \in C_0^{\infty} (\T \times \mathbb{T}^3)$,
		\begin{align*}
			\int_{\mathbb{T}^3}  \partial_t \varphi \cdot u - \nu_1(-\Delta)^{\alpha_1}\varphi \cdot u  + \nabla\varphi :( u \otimes u - B \otimes B) \d x  &= 0,
			\\
			\int_{\mathbb{T}^3}  \partial_t \varphi \cdot B - \nu_2 (-\Delta)^{\alpha_2}\varphi \cdot B + \nabla\varphi  :(B \otimes u- u \otimes B) \d x  &= 0.
		\end{align*}
	\end{itemize}
	Here, for any $3\times 3$ matrices $A=({A_{ij}})$ and $S=({S_{ij}})$,
	$ A:S=\sum_{i,j=1}^{3}A_{ij}S_{ij}$.
\end{definition}
For simplicity, we consider the solutions on the temporal interval $\mathbb{T}$
to be consistent with the spatial torus $\mathbb{T}^3$.
The arguments of this paper can be also extended to any bounded temporal intervals.

Below we mainly formulate the main results for the more general MHD type equations \eqref{equa-gMHD},
which in particular includes the canonical MHD system \eqref{equa-MHD}
that is the main model of the present work.

\begin{theorem} \label{Thm-Non-gMHD}
	Let $\alpha_1,\alpha_2\in [0,5/4)$,
	$(\tilde{u},\tilde{B})$ be any smooth, divergence-free and mean-free vector fields on $\T\times \T^3$ ,
    and $(s,p,\gamma)\in \mathcal{A}$  where
    \begin{align} \label{A-regularity}
			\mathcal{A}:=\left\{ (s,p,\gamma)\in [0,\frac 32)\times [1, +\9)\times[1,+\9): 0\leq s< \frac{2}{p}+\frac{1}{\gamma}-\frac32 \right\}.
		\end{align}
    (See Figure~1. for the admissible set $\mathcal{A}$  of regularities  when $s=0$.)

    Then,
	there exists $\beta'\in(0,1)$,
	such that for any given $\va_*>0$, there exist  a velocity field $u$ and a magnetic field $B$ such that the following holds:
	\begin{enumerate}[(i)]
		\item Weak solution: $(u,B)$ is a weak solution
        to \eqref{equa-gMHD} in the sense of Definition~\ref{weaksolu} with zero spatial mean.
		\item Regularity: $u, B \in H^{\beta'}_{t,x} \cap L^\gamma_tW^{s,p}_x$.
		\item Temporal support: $\supp_t u \cup  \supp_t B  \subseteq N_{\va_*}(\supp_t \tilde{u}\cup \supp_t \tilde{B})$.
		\item Small deviations on average: $\|u-\tilde{u}\|_{L^1_tL^2_x}\leq \va_*$, $\|B-\tilde{B}\|_{L^1_tL^2_x}\leq \va_*$.
      \item Small deviations of magnetic helicity: $	\|\mh_{B,B}-\mh_{\wt B,\wt B }\|_{L_t^1}\leq \va_* $.
	\end{enumerate}
\end{theorem}

\begin{figure}[H]\label{domain}
\centering
\includegraphics[width=0.40\textwidth]{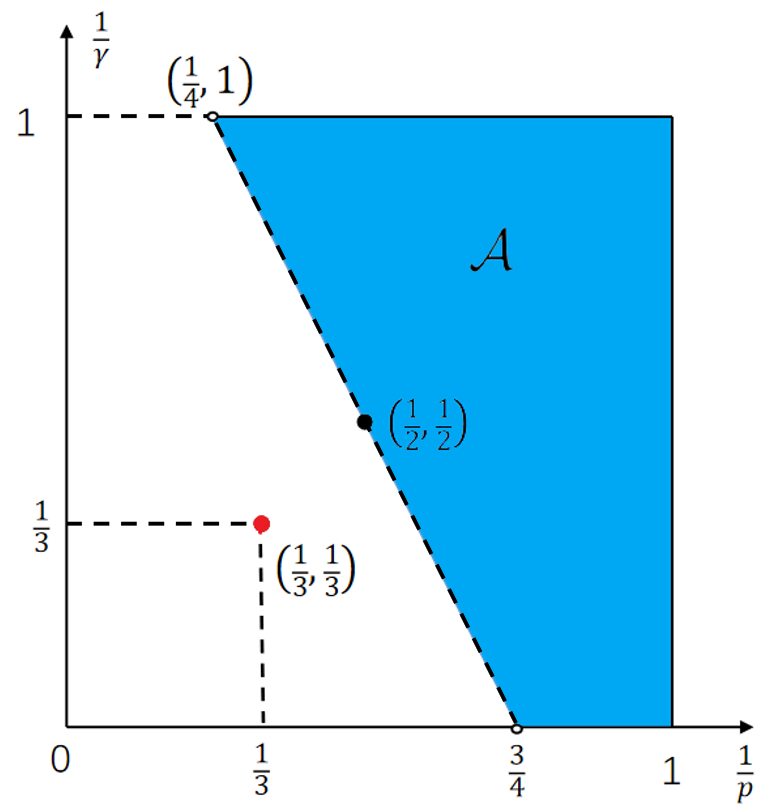}
\caption{{\protect\\}}
\centering{(The red point is the conjectured regularity  corresponding to $L^3_{t,x}$,

the blue region is the admissible set $\mathcal{A}$ of regularities  when $s=0$.)}
\end{figure}

As a consequence, Theorem \ref{Thm-Non-gMHD}
gives the following result concerning the non-uniqueness of weak solutions
and the non-conservation of magnetic helicity.

\begin{corollary} \label{Cor-Nonuniq-Nonconserv-gMHD}
For any $\a_1,\a_2\in [0,5/4)$ and
any $(s,p,\gamma)\in \mathcal{A}$,
where the regularity admissible set $\mathcal{A}$ is as in \eqref{A-regularity},
there exist infinitely many weak solutions to \eqref{equa-gMHD}
with the same data at time zero,
which live in the space  $H^{\beta'}_{t,x}  \cap L^\gamma_tW^{s,p}_x $
for some $\beta'>0$
and do not conserve the magnetic helicity.
\end{corollary}

\begin{remark} \rm\ \

	(i). The non-uniqueness results for the ideal MHD \eqref{equa-IMHD}
	were proved by Faraco-Lindberg-Sz\'{e}kelyhidi  \cite{fls21}
	for weak solutions with compact support in space-time,
	and by Beekie-Buckmaster-Vicol \cite{bbv20} for
	weak solutions with non-trivial magnetic helicity.

    To the best of our knowledge, Corollary \ref{Cor-Nonuniq-Nonconserv-gMHD} provides
    the first non-uniqueness result  for the weak solutions to
     the MHD equations \eqref{equa-gMHD} with $\alpha_i\in (0,5/4)$,
     i.e., less than the sharp Lions exponent.

	(ii). By the Onsager-type conjecture for the magnetic helicity in \cite{bv21},
	it was expected that $L^3_{t,x}$ is the threshold for the magnetic helicity conservation for the ideal MHD.
    This conjecture was solved recently by Faraco-Lindberg-Sz\'{e}kelyhidi \cite{fls21.2},
    by constructing weak solutions uniformly-in-time in the spatial Lorentz space
    $L^{3,\infty}_x$ with non-trivial magnetic helicity,
    based on the convex integration via staircase laminates.

    In view of Corollary \ref{Cor-Nonuniq-Nonconserv-gMHD},
    one has more  examples
    in the spaces $H^{\beta'}_{t,x} \cap L^\gamma_tW^{s,p}_x$
    for the flexible part of the conjecture for the magnetic helicity,
    where $\beta'$ is a positive small constant and $(s,p,\gamma) \in \mathcal{A}$ with $\mathcal{A}$ as in Theorem \ref{Thm-Non-gMHD}.
	
	(iii).
    In the higher dimensions $d\geq 4$,
	one may prove the non-uniqueness of weak solutions in $C_tL^2_x$
	for the viscous and resistive MHD \eqref{equa-MHD}.
	The solutions can be continuous in time.
	This is possible because, in high dimensions $d\geq 4$,
	there is more freedom to choose the oscillation directions in the
    velocity and magnetic flows,
	which permit to gain the  3D spatial intermittency
	to control the Laplacian $(-\Delta)$.
	
	(iv).
	We expect that the proof of this paper applies to the non-uniqueness in law for the stochastic generalized MHD equations
	driven by Wiener processes,
    which will be done in a forthcoming work.
	We would like to refer the reader to
    the papers \cite{BFH20,cff21,hzz19a,hzz19b} for the recent developments in the stochastic contexts.
\end{remark}

Our next result is concerned with the
vanishing viscosity and resistivity  limits of the weak solutions to \eqref{equa-gMHD}.

\begin{theorem} \label{Thm-iMHD-gMHD-limit}
	Let $\a_1,\a_2\in (0,5/4)$ and
   $(u,B)\in H^{\wt{\beta}}_{t,x} \times H^{\wt{\beta}}_{t,x}$ be any mean-free weak solution to the ideal MHD \eqref{equa-IMHD},
	where $\wt \beta >0$.
	Then, there exist $\beta' \in (0, \wt \beta)$ and a sequence of weak solutions
	$(u^{(\nu_{n})},B^{(\nu_n)})\in H^{\beta'}_{t,x} \times H^{\beta'}_{t,x}$
	to the MHD equations \eqref{equa-gMHD},
    where $\nu_n=(\nu_{1,n},\nu_{2,n})$
    and $\nu_{1,n}$, $\nu_{2,n}$
    are the viscosity and resistivity coefficients,  respectively,
	such that
	as $\nu_{n}\rightarrow 0$,
	\begin{align}\label{convergence}
		u^{(\nu_{n})}\rightarrow u,\quad B^{(\nu_{n})}\rightarrow B \quad\text{strongly in}\  H^{\beta'}_{t,x}.
	\end{align}

\end{theorem}

\begin{remark} \label{Rem-viscosity-limit} \rm\ \

	(i). By Theorem~\ref{Thm-iMHD-gMHD-limit}, the set of the
	accumulation points of weak solutions to the MHD equations \eqref{equa-gMHD},
	in the $H^{\beta'}_{t,x}$ topology,
	contains the weak solutions to the ideal MHD equations \eqref{equa-IMHD}.
    This extends the vanishing viscosity result in
    the context of the NSE \cite{bv19b}
    to the MHD equations
    up to the Lions exponent.

    Let us mention that, the Sobolev space $H^{\wt \beta}_{t,x}$ is less regular than  $C^{\wt \beta}_{t,x}$
    used in the vanishing viscosity limit result for the NSE in \cite{bv19b}.
    In order to get the temporal regularity of solutions,
    we use the Slobodetskii-type norm of $H^{\wt \beta}_{t,x}$ to gain the regularity on average,
    which turns out to be efficient to prove the vanishing viscosity limit in the Sobolev space $H^{\beta'}_{t,x}$
     with $\beta'<\wt \beta$. See the proof of \eqref{u-un-lbbn} and \eqref{uu-uun-wtM} below.

   (ii). Another interesting outcome of Theorem \ref{Thm-iMHD-gMHD-limit} is
	related to the weak solutions constructed in \cite{bbv20}
    and to the Taylor conjecture concerning the conservation of magnetic helicity.
	
	The Taylor conjecture  is mathematically formulated in the concept of
	weak ideal limits,
	i.e., the weak limits of Leray weak solutions,
    and has been recently proved by
	Faraco-Lindberg  \cite{fl20}.
	The crucial ingredient of the proof
	is the regularity of Leray approximating solutions.
	Hence,
	the weak solutions constructed in \cite{bbv20} with non-trivial magnetic helicity cannot be obtained
	as the weak ideal limits of Leray-Hopf solutions to the MHD \eqref{equa-MHD}
    (cf. \cite{bbv20,bv21}).

    Let us mention that, the weak solutions constructed in \cite{bbv20}
    has the regularity $H^{\wt{\beta}}_{t,x}$ for some $\wt \beta>0$
    (see Remark \ref{Rem-IMHD-Hbeta} below).
	Hence, by virtue of Theorem \ref{Thm-iMHD-gMHD-limit},
    we note that, in contrast to the weak ideal limits,
	the weak solutions in \cite{bbv20} actually can be obtained
	as the vanishing viscosity  and resistivity limits
	of those to the MHD equations \eqref{equa-gMHD},
	where the exponents $\a_i$ can be even  larger than one.
	In particular,
	this yields that Taloy's conjecture does not hold along the
	vanishing viscosity and resistivity limits of the weak solutions to \eqref{equa-gMHD}.
	
\end{remark}

\subsection{Outline of the proof}

Our strategy of proof is based on the intermittent  convex integration scheme,
inspired by the works \cite{bbv20,bv19b,cl21}.

More precisely, for each integer $q \in \mathbb{N}$, we consider the following relaxation system
\begin{equation}\label{mhd1}
	\left\{\aligned
	&\p_t \u+\nu_1(-\Delta)^{\alpha_1} \u+ \div(\u\otimes\u-\h\otimes\h)+\nabla P_q=\div \ru,  \\
	&\p_t \h+\nu_2 (-\Delta)^{\alpha_2} \h+ \div(\h\otimes\u-\u\otimes\h)=\div \rb , \\
	&\div \u = 0, \quad \div \h= 0,
	\endaligned
	\right.
\end{equation}
where the Reynolds stress $\ru$ is a symmetric traceless $3\times 3$ matrix,
and the magnetic stress $\rb$ is a skew-symmetric $3\times 3$ matrix.

The quantitative estimates of the relaxation solutions $(u_q, B_q, \mathring{R}^u_q, \mathring{R}^B_q)$
are specified in the main iteration result below.

More precisely, let $\alpha:=\max\{\alpha_1,\alpha_2\}$,
$\alpha_* :=\max\{2\alpha-1,0\}$.
Take $\varepsilon\in \mathbb{Q}_+$
sufficiently small such that
\begin{equation}\label{e3.1}
	\varepsilon\leq\min\left\{\frac{3-2\alpha_*}{20},\frac{1}{10}\right\},
\end{equation}
and
\begin{align}  \label{s-p-gamma-ve}
     s+\frac 32 - \frac 2p - \frac 1 \gamma
     + 8 \ve <0.
\end{align}
For $q\in \mathbb{N}$,
the frequency parameter $\lbb_q$ and
the amplitude parameter $\delta_{q+2}$ are defined by
\begin{equation}\label{la}
	\la=a^{(b^q)}, \ \
	\delta_{q+2}=\lambda_{q+2}^{-2\beta}.
\end{equation}
Here, $a\in 5\mathbb{N}$ is a large integer of multiple $5$ such that
$a^\varepsilon\in 5\mathbb{N}$,
$b\in 2\mathbb{N}$ is a large integer of multiple $2$
such that $\lambda_{q+1}^{2\varepsilon}=\lambda_q^{2\varepsilon b}\in \mathbb{N}$
and
\begin{align}
   b>\frac{1000}{\varepsilon}, \ \
   0<\beta<\frac{1}{100b^2}. \label{b-beta-ve}
\end{align}
Let $q \in \mathbb{N}$,
we assume the following inductive estimates of the relexation solutions to \eqref{mhd1} at level $q$:
\begin{align}
	& \|\u\|_{C_{t,x}^{1}} \leq   \lambda_{q}^{7},\quad \|\h\|_{C_{t,x}^{1}} \leq  \lambda_{q}^{7}, \label{ubc} \\
	& \|\ru\|_{C^{1}_{t,x}} \leq   \lambda_{q}^{14},\quad \|\rb\|_{C^{1}_{t,x}} \leq \lambda_{q}^{14},\label{rubc1} \\
	& \|\ru\|_{L^{1}_{t,x}} \leq  \delta_{q+1},\quad \|\rb\|_{L^{1}_{t,x}} \leq  \delta_{q+1}. \label{rub}
\end{align}

The main iteration result is contained in the following theorem.

\begin{theorem} [Main iteration]\label{Prop-Iterat}
	Let $\alpha_1,\alpha_2 \in [0,5/4)$.
	Then, there exist $\beta\in (0,1)$,
    $M>0$ large enough and $a_0=a_0(\beta, M)$,
    such that for any integer $a\geq a_0$,
	the following holds:
	
	Suppose that $(\u,\h,\ru,\rb)$ solves  \eqref{mhd1}
	and satisfies  \eqref{ubc}-\eqref{rub}.
	Then, there exists $(u_{q+1}, B_{q+1}, \mathring{R}^u_{q+1}, \mathring{R}^B_{q+1})$
	solving \eqref{mhd1} and satisfying \eqref{ubc}-\eqref{rub} with $q+1$ replacing $q$.
	In addition, we have
	\begin{equation}\label{u-B-L2tx-conv}
		\left\|u_{q+1}-u_{q}\right\|_{L^{2}_{t,x}} \leq  M\delta_{q+1}^{\frac{1}{2}}, \quad \left\|B_{q+1}-B_{q}\right\|_{L^{2}_{t,x}} \leq  M\delta_{q+1}^{\frac{1}{2}},
	\end{equation}
	\begin{equation}\label{u-B-L1L2-conv}
		\left\|u_{q+1}-u_{q}\right\|_{L^1_tL^{2}_{x}} \leq  \delta_{q+2}^{\frac{1}{2}}, \quad \left\|B_{q+1}-B_{q}\right\|_{L^1_tL^{2}_{x}} \leq  \delta_{q+2}^{\frac{1}{2}},
	\end{equation}
and
\begin{align}
&\supp_t u_{q+1} \cup\supp_t B_{q+1}
\subseteq N_{\delta_{q+2}^{\frac12}}\(\supp_t u_{q} \cup\supp_t B_{q} \cup \supp_t \rr_{q}^u\cup \supp_t \rr_{q}^B\) ,\label{suppub}\\
& \supp_t \rr_{q+1}^u \cup \supp_t \rr_{q+1}^B
\subseteq N_{\delta_{q+2}^{\frac12}}\(\supp_t u_{q} \cup\supp_t B_{q} \cup \supp_t \rr_{q}^u\cup \supp_t \rr_{q}^B\).  \label{supprub}
\end{align}
\end{theorem}

The heart of the proof of Theorem \ref{Prop-Iterat} is   to
construct suitable perturbations
\begin{align*}
   w_{q+1} \simeq u_{q+1} - u_q,\ \ d_{q+1} \simeq B_{q+1}- B_q,
\end{align*}
module small mollification errors,
such that
the corresponding nonlinear effects decrease the amplitudes of the Reynolds and magnetic stresses,
in order to fulfill the above iteration procedure.

\subsection{New ingredients of proof}

As mentioned in \S \ref{Subsec-intro},
besides the intermittent velocity flows $\{W_{(k)}\}$,
the intermittent magnetic flows $\{D_{(k)}\}$
are also required and shall be
constructed in an appropriate way,
in order to respect the geometry of MHD.
This forces to limit the oscillation directions,
and thus the intermittent velocity and magnetic flows
do not have the 3D spatial intermittency as in the NSE case \cite{bv19b}.

In the context of the ideal MHD \eqref{equa-IMHD},
the key intermittent shear velocity flows
and the intermittent shear magnetic flows
were introduced of the form
\begin{align} \label{Wk-Dk-IMHD}
   W_{(k)} = \phi_{(k)} k_1,\quad D_{(k)} =  \phi_{(k)} k_2,
\end{align}
where $\{\phi_{(k)}\}$ are the spatial concentration functions
with one oscillation direction $k$,  orthogonal to the other two directions $k_1,k_2$.
The shear flows have the almost 1D spatial intermittency, i.e.,
\begin{align}    \label{Wk-Dk-interm-IMHD}
	\|W_{(k)}\|_{C_tL^1_x} +  \|D_{(k)}\|_{C_tL^1_x}
	\lesssim \lbb^{-\frac 12 +\ve},
\end{align}
which permits to control the
fractional viscosity and resistivity
$(-\Delta)^{\a_i}$ with $\a_i\in [0,1/2)$, $i=1,2$.

One of the novelties of our proof
is the construction of a new class of intermittent velocity and magnetic flows
adapted to the geometry of MHD,
which feature the refined intermittency in both the space and time
and, in particular, enable us to control the stronger viscosity and resistivity
$(-\Delta)^{\a_i}$ with $\a_i\geq 1/2$, $i=1,2$.

\vspace{1ex}

(i) {\it Spatial intermittent building blocks.}
We construct the refined intermittent
velocity and magnetic flows
\begin{equation}    \label{Wk-Dk-gMHD}
	W_{(k)} = \psi_{(k_1)}\phi_{(k)} k_1,\quad D_{(k)} = \psi_{(k_1)}\phi_{(k)} k_2,
\end{equation}
which are $({\T}/({\lambda\rs}))^3$-periodic and
concentrate on  smaller cuboids with volume $\lambda^{-3}\rs^{-3}\rp$
in each of the cubes with side length ${2\pi}/{(\lambda\rs)}$.
Here,
$\psi_{(k_1)}$ and $\phi_{(k)} $ are suitable concentration functions
(see \eqref{snp} below),
$\rp$ and $\rs$ are the small constants
to parameterize the spatial concentration of the building blocks,
as in the  intermittent jets for the NSE \cite{bv19r}.

One nice feature of the refined  flows is that,
they have the almost 2D intermittency
\begin{align}     \label{Wk-Dk-interm-gMHD}
	\|W_{(k)}\|_{C_tL^1_x} +  \|D_{(k)}\|_{C_tL^1_x}
	\lesssim \lbb^{-1 + \ve},
\end{align}
and thus enable one to control the
viscosity and resistivity $(-\Delta)^{\a_i}$ with $\a_i\in [1/2, 1)$, $i=1,2$.
Another nice feature is that,
the intersections between distinct intermittent flows concentrate on much smaller cuboids,
which  provide the almost
2D intermittency, i.e.,
\begin{align}     \label{kk'-interm-gMHD}
	\|W_{(k)}\otimes W_{(k')}\|_{C_tL^1_x}	\lesssim \lbb^{-1 + \ve},\quad k\neq k',
\end{align}
and thus simplify the control of the nonlinearities of MHD.

Let us also mention that,
unlike the one oscillation direction $k$ in \eqref{Wk-Dk-IMHD}
in the ideal MHD case,
the building blocks $W_{(k)}$ and $D_{(k)}$ in \eqref{Wk-Dk-gMHD}
oscillate in two orthogonal directions $k_1$ and $k$.
This new oscillation direction $k_1$
causes extra high spatial oscillations,
which need to be balanced
by a new temporal parameter $\mu$
in the  flows
and the temporal correctors $w^{(t)}_{q+1}$ and $d^{(t)}_{q+1}$.
See \eqref{veltemcor}, \eqref{magtemcor}
and the important indentities \eqref{utem} and \eqref{btem} below.

It might be tempting to introduce one more oscillation direction $k_2$ in the building blocks,
i.e., with the oscillation directions $(k,k_1,k_2)$ as in the NSE context \cite{bv19b,bv19r}.
However,
this causes even more high spatial oscillations,  which seem not possible to be balanced by
the temporal parameter $\mu$.

\vspace{1ex}

(ii) {\it Temporal intermittent building blocks.}
The next major difficulty is to pass beyond the border line $\a_i =1$, $i=1,2$.

The key idea here is to exploit the new intermittency from
the temporal oscillations.
Thus, rather than performing the pointwise analysis in time,
we measure the Reynolds and magnetic stresses on average,
i.e., in the space $L^1_{t,x}$.

Let us mention that,
this is in spirit close to  the works \cite{bv19b,cl20}.
In \cite{cl20} the temporal intermittency is crucial
to achieve the non-uniqueness of weak solutions for transport equations
at the critical space regularity.
In the NSE context \cite{bv19b},
the key ingredient is the $L^2_x$-based spatial intermittency,
instead of the previous pointwise analysis in space as in the Euler settings.

We introduce the temporal concentration function $\g$
in the principal part of perturbations $(w_{q+1}^{(p)}$, $d_{q+1}^{(p)})$,
and control the oscillation errors in the space $L^1_t$.
Here, $\tau$ is the constant to parameterize the temporal concentration.
In particular,
this permits to gain an additional  1D intermittency, i.e.,
\begin{align}
	\|\g\|_{L^1_t}  \lesssim \lbb^{-\frac12 + \ve}.
\end{align}
More interestingly,
the gained temporal intermittency enables to control the
much stronger viscosity and resistivity  $(-\Delta)^{\a_i}$
for all the exponents $\a_i$ less than $5/4$, $i=1,2$,
that is, exactly the Lions exponent.

It should be mentioned that,
extra high temporal oscillations also arise due to
the presence of the temporal concentration function $\g$.
In order to balance these high oscillations,
another new type of temporal correctors $\wo,\dqo$
will be introduced, respectively, for the velocity and magnetic flows.
See \eqref{wo}, \eqref{do}
and the key identities \eqref{utemcom} and \eqref{btemcom} below.

\vspace{1ex}

The remainder of this paper is organized as follows.
We first regularize the relexation solutions to \eqref{mhd1}
by the standard mollification procedure in Section \ref{Sec-Moll}.
Section \ref{Sec-Interm-Flow} is devoted to the construction of the crucial intermittent
velocity and magnetic flows adapted to MHD,
which feature the intermittency in both the space and the time.
Then, in Section \ref{Sec-Pert} we construct the velocity and magnetic perturbations,
consisting of the principal parts to decrease the Reynolds and magnetic stresses,
the incompressibility correctors and two types of temporal correctors.
Several key algebraic identities and analytic estimates also will be given there,
which lead to the verification of the inductive estimates for the velocity and magnetic fields.
Then, the inductive estimates of Reynolds and magnetic stresses
are  verified in Section \ref{Sec-Rey-mag-stress}.
Consequently, the proof of the main results is presented in Section \ref{Sec-Proof-Main}.

\section{Mollification} \label{Sec-Moll}

In order to avoid the loss of derivatives, we mollify the velocity and magnetic fields.
Let $\phi_{\varepsilon}$ and $\varphi_{\varepsilon}$ be the  mollifiers on $\T^3$ and $\T$, respectively, $\ve>0$,
and $\supp \vf_\ve\subseteq (-\ve, \ve)$.

The mollifications of $(\u,\h,\ru,\rb)$ in space and time are defined by
\begin{equation}\label{mol}
	\begin{array}{ll}
		& u_{\ell}:=\left(u_{q} *_{x} \phi_{\ell}\right) *_{t} \varphi_{\ell}, \quad B_{\ell}:=\left(\h *_{x} \phi_{\ell}\right) *_{t} \varphi_{\ell}, \\
		& \mathring{R}_\ell^u:=\left(\ru *_{x} \phi_{\ell}\right) *_{t} \varphi_{\ell},\quad \mathring{R}_\ell^B:=\left(\rb *_{x} \phi_{\ell}\right) *_{t} \varphi_{\ell},
	\end{array}
\end{equation}
where the scale of mollification is parameterized by
\begin{align} \label{l-lbbq}
	\ell:=\la^{-20}.
\end{align}

Then, by equation \eqref{mhd1},
$(u_{\ell},B_{\ell},\mathring{R}_\ell^u,\mathring{R}_\ell^B)$ satisfies
\begin{equation}\label{me}
	\left\{\aligned
	&\partial_{t} u_{\ell}+\nu_1(-\Delta)^{\alpha_1} u_{\ell}+\div\left(u_{\ell}\otimes u_{\ell}-B_{\ell}\otimes B_{\ell}\right)+\nabla P_{\ell}=\div(\mathring{R}_{\ell}^{u}+\mathring{R}_{com}^{u}), \\
	&\partial_{t} B_{\ell}+\nu_{2}(-\Delta)^{\alpha_2} B_{\ell}+\div\left(B_{\ell}\otimes u_{\ell}-u_{\ell} \otimes B_{\ell}\right)=\div(\mathring{R}_{\ell}^{B}+\mathring{R}_{com}^{B}),\\
	&\div u_{\ell}=0, \quad\div B_{\ell}=0, \\
	\endaligned
	\right.
\end{equation}
where the traceless symmetric commutator stress $\mathring{R}_{c om}^{u}$ and
the skew-symmetric commutator stress $\mathring{R}_{c om}^{B}$ are of form
\begin{equation}\label{e3.7}
	\begin{aligned}
		&\mathring{R}_{com}^{u} :=u_{\ell}\mathring{\otimes} u_{\ell}-B_{\ell} \mathring{\otimes}B_{\ell}-\left(\u \mathring{\otimes}\u-\h \mathring{\otimes}\h   \right) *_{x} \phi_{\ell} *_{t} \varphi_{\ell},\\
		&\mathring{R}_{com}^{B} :=B_{\ell} \otimes u_{\ell}-u_{\ell}\otimes B_{\ell}-\left(\h \otimes\u-\u \otimes \h   \right) *_{x} \phi_{\ell} *_{t} \varphi_{\ell},
	\end{aligned}
\end{equation}
and the pressure $P_{\ell}$ is given by
\begin{equation}\label{e3.8}
	\begin{aligned}
		&P_{\ell} :=\left(P_{q} *_{x} \phi_{\ell}\right) *_{t} \varphi_{\ell}-\left|u_{\ell}\right|^{2}+\left|B_{\ell}\right|^{2}+\left(\left|\u\right|^{2}-\left|\h\right|^{2}\right) *_{x} \phi_{\ell} *_{t} \varphi_{\ell}.
	\end{aligned}
\end{equation}

By the standard mollification estimates and the inductive estimates \eqref{ubc}-\eqref{rub},
for any $N\in \mathbb{N}_+$,
\begin{align}
	& \left\|u_\ell\right\|_{C_{t,x}^N}+\left\|B_\ell\right\|_{C_{t,x}^N}
        \lesssim \ell^{-N+1}(\left\|u_q\right\|_{C_{t,x}^1}+\left\|B_q\right\|_{C_{t,x}^1}) \lesssim  \ell^{-N+1}\la^7\lesssim \ell^{-N},    \label{e3.9} \\
	&  \| \mathring{R}_\ell^u \|_{C_{t,x}^N} + \|\mathring{R}_\ell^B \|_{C_{t,x}^N} \lesssim \ell^{-N+1} ( \| \mathring{R}_q^u \|_{C_{t,x}^1}
         + \|\mathring{R}_q^B \|_{C_{t,x}^1} ) \lesssim  \ell^{-N+1} \lambda_{q}^{14}\lesssim \ell^{-N},  \label{e3.92} \\
	&  \| \mathring{R}_\ell^u \|_{L^{1}_{t,x}}\lesssim  \delta_{q+1},    \quad
       \| \mathring{R}_\ell^B \|_{L^{1}_{t,x}} \lesssim   \delta_{q+1}. \label{e3.93}
\end{align}
Moreover,
by the inductive estimate \eqref{ubc} and the double commutator estimate
(see, e.g., \cite[Proposition B.1]{luo19}, \cite[Lemma 1]{CD12}, \cite{cwt94}),
for $1<p<+\9$,
\begin{equation}\label{e3.10}
	 \|\mathring{R}_{com}^{u} \|_{L^{1}_{t}L^p_x} \lesssim \|\mathring{R}_{com}^{u} \|_{C_{t,x}}
	\lesssim \ell  (\|u_q\|_{C_{t,x}^1}^2+\|B_q\|_{C_{t,x}^1}^2 )
	\lesssim \ell  \lambda_{q}^{14},
\end{equation}
and, similarly,
\begin{equation}\label{e3.11}
	 \|\mathring{R}_{com}^{B} \|_{L^{1}_{t}L^p_x}
	\lesssim  \|\mathring{R}_{com}^{B} \|_{C_{t,x}}
	\lesssim \ell  \lambda_{q}^{14}.
\end{equation}

\section{Intermittent velocity and magnetic flows} \label{Sec-Interm-Flow}

This section contains the crucial intermittent velocity and magnetic flows,
which are indeed the fundamental building blocks in the convex integration scheme.

We set the parameters  $\rs$, $\rp$, $\lambda$, $\mu$, $\tau$ and $\sigma$ as follows
\begin{equation}\label{larsrp}
	\rs := \lambda_{q+1}^{-1+2\varepsilon},\ \rp := \lambda_{q+1}^{-1+6\varepsilon},\
	 \lambda := \lambda_{q+1},\ \mu:=\lambda_{q+1}^{\frac 3 2-6\varepsilon}, \
      \tau:=\lambda_{q+1}^{1-6\varepsilon}, \ \sigma:=\lambda_{q+1}^{2\varepsilon},
\end{equation}
where $\varepsilon$ is the small constant satisfying \eqref{e3.1}.

\subsection{Geometric lemmas.} To begin with, let us first recall two geometric lemmas in \cite{bbv20}.
\begin{lemma} ({\bf First Geometric Lemma}, \cite[Lemma 4.1]{bbv20})
	\label{geometric lem 1}
	There exists a set $\Lambda_B \subset S^2 \cap \mathbb{Q}^3$ that consists of vectors $k$ with associated orthonormal bases $(k, k_1, k_2)$,  $\varepsilon_B > 0$, and smooth positive functions $\gamma_{(k)}: B_{\varepsilon_B}(0) \to \mathbb{R}$, where $B_{\varepsilon_B}(0)$ is the ball of radius $\varepsilon_B$ centered at 0 in the space of $3 \times 3$ skew-symmetric matrices, such that for  $A \in B_{\varepsilon_B}(0)$ we have the following identity:
	\begin{equation}
		\label{antisym}
		A = \sum_{k \in \Lambda_B} \gamma_{(k)}^2(A) (k_2 \otimes k_1 - k_1 \otimes k_2) .
	\end{equation}
\end{lemma}

\begin{lemma} ({\bf Second Geometric Lemma}, \cite[Lemma 4.2]{bbv20})
	\label{geometric lem 2}
	There exists a set $\Lambda_u \subset S^2 \cap \mathbb{Q}^3$ that consists of vectors $k$ with associated orthonormal bases $(k, k_1, k_2)$,  $\varepsilon_u > 0$, and smooth positive functions $\gamma_{(k)}: B_{\varepsilon_u}(\Id) \to \mathbb{R}$, where $B_{\varepsilon_u}(\Id)$ is the ball of radius $\varepsilon_u$ centered at the identity in the space of $3 \times 3$ symmetric matrices,  such that for  $S \in B_{\varepsilon_u}(\Id)$ we have the following identity:
	\begin{equation}
		\label{sym}
		S = \sum_{k \in \Lambda_u} \gamma_{(k)}^2(S) k_1 \otimes k_1 .
	\end{equation}
	Furthermore, we may choose $\Lambda_u$ such that $\Lambda_B \cap \Lambda_u = \emptyset$.
\end{lemma}

As pointed out in \cite{bbv20},
there exists $N_{\Lambda} \in \mathbb{N}$ such that
\begin{equation} \label{NLambda}
	\{ N_{\Lambda} k,N_{\Lambda}k_1 , N_{\Lambda}k_2 \} \subset N_{\Lambda} \mathbb{S}^2 \cap \mathbb{Z}^3.
\end{equation}
For instance,  $N_{\Lambda} = 65$ suffices.
We denote by $M_*$ the geometric constant such that
\begin{align}
	\sum_{k \in \Lambda_{u}} \norm{\gamma_{(k)}}_{C^4(B_{\varepsilon_u}(\Id))}
     + \sum_{k \in \Lambda_{B}} \norm{\gamma_{(k)}}_{{C^4(B_{\varepsilon_B}}(0))} \leq M_*.
	\label{M bound}
\end{align}
This parameter  is universal and will be used later in the estimates of the size of perturbations.

According to the proof of \cite[Lemmas 4.1 and 4.2]{bbv20},
we can choose the wave vectors satisfying $k_2\not = k_2'$ if $k\not = k'$,
so that the intersections between distinct intermittent flows
concentrate on much smaller cuboids and have 2D spatial intermittency
(see Lemma \ref{product estimate lemma} below).

\subsection{Spatial building blocks.}  Let $\Phi : \mathbb{R} \to \mathbb{R}$ be a smooth cut-off function supported on
the interval $[-1,1]$.
We normalize $\Phi$ such that $\phi := - \frac{d^2}{dx^2}\Phi$ satisfies
\begin{equation}\label{e4.91}
	\frac{1}{2 \pi}\int_{\mathbb{R}} \phi^2(x)\d x = 1.
\end{equation}
Moreover, let $\psi: \mathbb{R} \rightarrow \mathbb{R}$ be a smooth and mean zero function,
supported on the interval $[-1,1]$, satisfying
\begin{equation}\label{e4.92}
	\frac{1}{2 \pi} \int_{\mathbb{R}} \psi^{2}\left(x\right) \d x=1.
\end{equation}
The corresponding rescaled cut-off functions are defined by
\begin{equation*}
	\phi_{\rs}(x) := {\rs^{-\frac{1}{2}}}\phi\left(\frac{x}{\rs}\right), \quad
	\Phi_{\rs}(x):=   {\rs^{-\frac{1}{2}}} \Phi\left(\frac{x}{\rs}\right),\quad
	\psi_{\rp}\left(x\right) := {r_{\|}^{- \frac 1  2}} \psi\left(\frac{x}{r_{\|}}\right).
\end{equation*}
Note that, $\phi_{\rs}=-\rs^2 \frac{d^2}{dx^2} \Phi_{\rs}$,
and $\phi_{\rs}, \psi_{\rp}$ are supported in the ball of radius $\rs$ and $\rp$, respectively, in $\bbr$.
By an abuse of notation,
we periodize $\phi_{\rs}$, $\Phi_{\rs}$ and $\psi_{\rp}$ so that
they are treated as periodic functions defined on $\mathbb{T}$.

\vspace{1ex}

The \textit{intermittent velocity flows} are defined by
\begin{equation*}
	W_{(k)} :=  \psi_{\rp}(\lambda \rs N_{\Lambda}(k_1\cdot x+\mu t))\phi_{\rs}( \lambda \rs N_{\Lambda}k\cdot x)k_1,\ \  k \in \Lambda_u \cup \Lambda_B  ,
\end{equation*}
and the \textit{intermittent magnetic flows} are defined by
\begin{equation*}
	D_{(k)} :=  \psi_{\rp}(\lambda \rs N_{\Lambda}(k_1\cdot x+\mu t))\phi_{\rs}( \lambda \rs N_{\Lambda}k\cdot x)k_2, \ \ k \in \Lambda_B .
\end{equation*}

Here, $N_{\Lambda}$ is given by \eqref{NLambda} above,
$(k,k_1,k_2)$ are the orthonormal bases in $\R^3$ in  Lemmas~\ref{geometric lem 1} and \ref{geometric lem 2}.
The parameters $\rp$ and $\rs$ parameterize the concentration of the flows,
and $\mu$ is the temporal oscillation parameter.

By definition, $\{W_{(k)}, D_{(k)}\}$ are $(\mathbb{T}/(\lbb \rs))^3$-periodic,
supported on thin cuboids with length $\sim {1}/{\lbb \rs}$,
width $\sim {\rp}/{(\lbb \rs)}$ and height $\sim {1}/{\lbb}$.
See the intermittent  flows in Figure~2.
See also the analytic properties in  Lemma~\ref{buildingblockestlemma} below.
\begin{figure}[H]\label{intersections}
	\centering
	\includegraphics[width=0.4\textwidth]{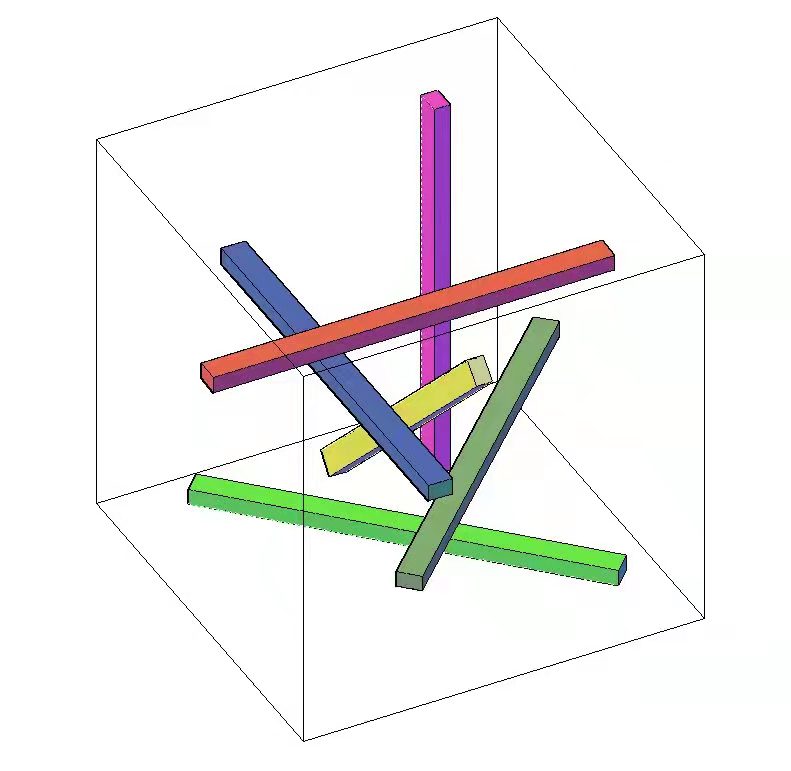}
	\caption{Intermittent flows}
\end{figure}

Moreover,
as there are only finite number of vectors in the wavevector sets $\Lambda_{u}$
and $\Lambda_{B}$, by choosing $k_2\neq k_2'$ if $k\neq k'$,
we note that the intersections of distinct intermittent flows
have much smaller volume.
See Figure~3 below for the typical cases of intersections,
and Lemma \ref{product estimate lemma} below for the corresponding analytic properties.

\begin{figure}[H]
	\centering
	\subfigure[Case 1]{
		\begin{minipage}[b]{0.45\textwidth}
			\centering
			\includegraphics[width=0.68\textwidth]{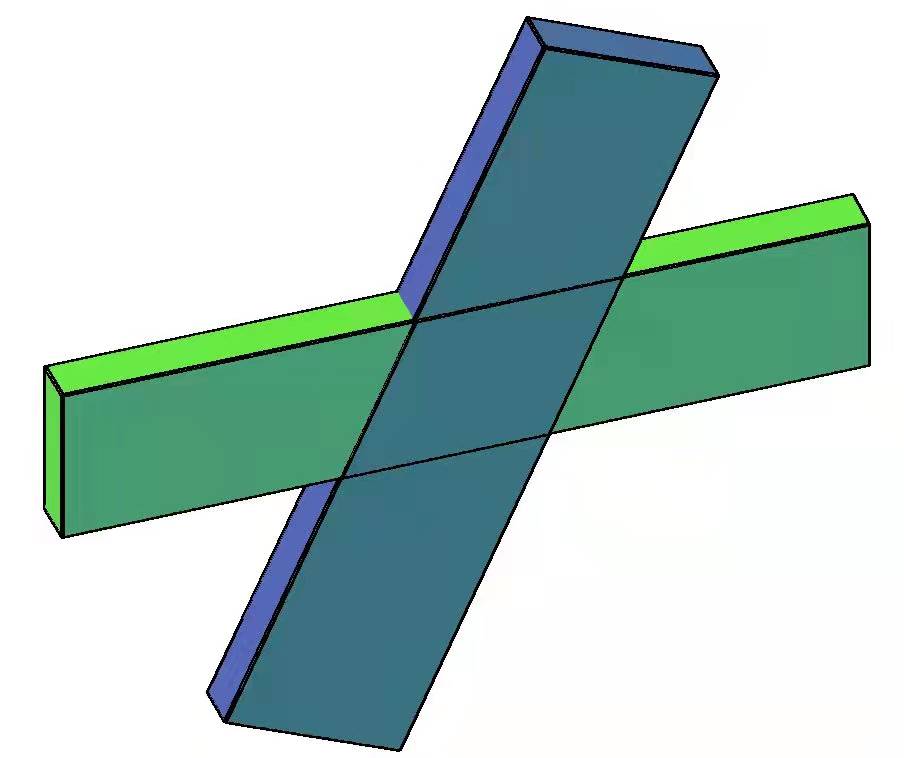}
		\end{minipage}
	}
	\subfigure[Case 2]{
		\begin{minipage}[b]{0.45\textwidth}
			\centering
			\includegraphics[width=0.68\textwidth]{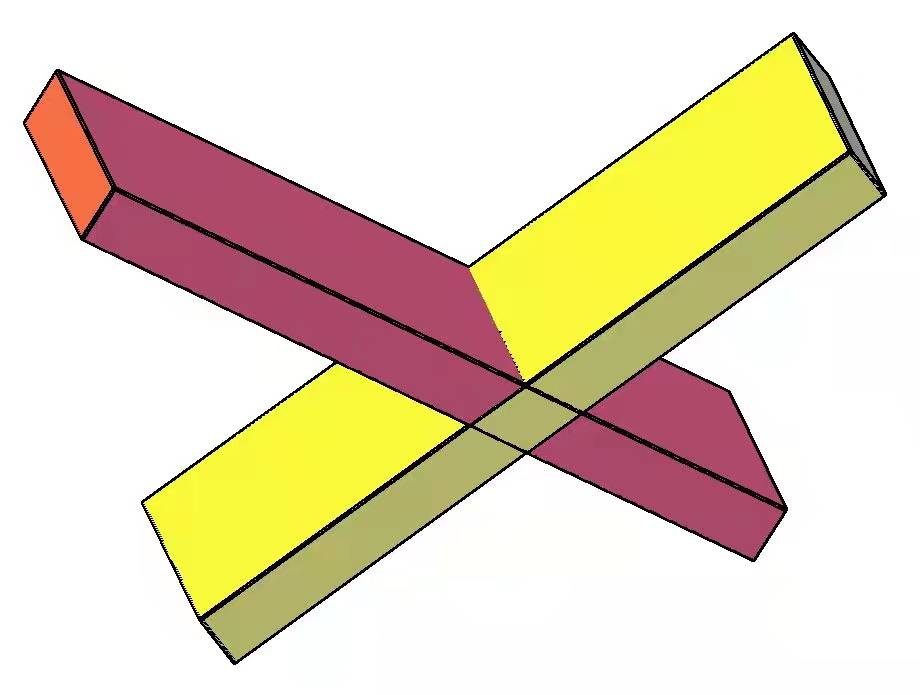}
		\end{minipage}
	}\label{PM}
	\caption{Typical cases of intersections}
\end{figure}

To ease notations, we set
\begin{equation}\label{snp}
	\begin{array}{ll}
		&\phi_{(k)}(x) := \phi_{\rs}( \lambda \rs N_{\Lambda}k\cdot x), \ \
		\Phi_{(k)}(x) := \Phi_{\rs}( \lambda \rs N_{\Lambda}k\cdot x),  \\
		&\psi_{(k_1)}(x) :=\psi_{\rp}(\lambda \rs N_{\Lambda}(k_1\cdot x+\mu t)),
	\end{array}
\end{equation}
and thus
\begin{equation}\label{snwd}
	W_{(k)} = \psi_{(k_1)}\phi_{(k)} k_1,\  k\in \Lambda_u\cup \Lambda_B,
	\ \ \ D_{(k)} = \psi_{(k_1)}\phi_{(k)} k_2,\ k\in \Lambda_B.
\end{equation}
Note that,
$W_{(k)}$ and $D_{(k)}$ are  mean zero on $\T^3$.
Moreover, since
\begin{equation}
	\label{normalization}
	\aint_{\mathbb{T}^3} \phi_{(k)}^2(x) \d x = 1 ,\quad \text{and} \quad  \aint_{\mathbb{T}^3} \psi_{(k_1)}^2(x) \d x = 1,
\end{equation}
it follows that
\begin{equation}\label{wdp}
	\begin{array}{ll}
		&\displaystyle \aint_{\mathbb{T}^3}W_{(k)}\otimes W_{(k)}\d x= k_1\otimes k_1,\ \
    \aint_{\mathbb{T}^3}D_{(k)}\otimes D_{(k)}\d x= k_2\otimes k_2,
	\end{array}
\end{equation}
\begin{equation}\label{wdp1}
	\begin{array}{ll}
		&\displaystyle\aint_{\mathbb{T}^3}W_{(k)}\otimes D_{(k)}\d x= k_1\otimes k_2,\ \
    \aint_{\mathbb{T}^3}D_{(k)}\otimes W_{(k)}\d x= k_2\otimes k_1.
	\end{array}
\end{equation}

It is worth noting that,
the building blocks in \cite{bbv20}
have one oscillation direction $k$,
orthogonal to the  direction vectors $k_1$ and $k_2$,
and have the 1D intermittency.

In order to control the stronger viscosity and resistivity $(-\Delta)^{\a_i}$
when $\a_i\geq 1/2$, $i=1,2$,
the idea here  to explore more intermittency,
inspired by the work \cite{bv19b} in the NSE context,
is to introduce the new concentration function $\psi_{(k_1)}$ in \eqref{snwd}.
Two  new parameters $\rs$ (the rescaling parameter in the direction vertical to the flow)
and $\rp$ (the rescaling parameter in the direction parallel to the flow)
parameterize the concentration of the intermittent flows $W_{(k)}$ and $D_{(k)}$.
In particular,
one has
\begin{align}
	\sum\limits_{k\in \Lambda_u\cup \Lambda_B} \|W_{(k)}\|_{C_tL^1_x}
	+ \sum\limits_{k\in  \Lambda_B} \|D_{(k)}\|_{C_tL^1_x}
	\lesssim \rs^\frac 12 \rp^{\frac 12} = \lbb_{q+1}^{-1 + 4\ve},
\end{align}
which gives the almost $2D$ intermittency and thus
permits to control the fractional viscosity  and resistivity $(-\Delta)^{\a_i}$ in \eqref{equa-gMHD}
for any $\a_1, \a_2\in [0,1)$.

Furthermore,
the new parameter $\mu$ permits to control the temporal correctors later
to balance the high spatial oscillations,
arising from the concentration function $\psi_{(k)}.$
More precisely,
using the orthogonality
\begin{equation}  \label{naphi-napsi-0}
\begin{array}{ll}
	&\nabla \phi_{(k)}\cdot k_1 = \nabla \phi_{(k)}\cdot k_2 =0,  \\
	&\nabla \psi_{(k_1)}\cdot k =  \nabla \psi_{(k_1)}\cdot k_2 =0,
\end{array}
\end{equation}
we have the important algebraic identities
\begin{equation}\label{wkwk}
	\begin{array}{ll}
		&\displaystyle \div ( W_{(k)}\otimes W_{(k)}) =2(W_{(k)}\cdot \nabla \psi_{(k_1)})\phi_{(k)}k_1
		= \frac{1}{\mu}\p_t \left(\psi_{(k_1)}^2 \phi_{(k)}^2 k_1\right), \\
		&\displaystyle \div ( D_{(k)}\otimes D_{(k)}) =2(D_{(k)}\cdot \nabla \psi_{(k_1)})\phi_{(k)}k_2
		= 0,
	\end{array}
\end{equation}
and
\begin{equation}\label{wd}
	\begin{array}{ll}
		& \displaystyle\div ( D_{(k)}\otimes W_{(k)}) =(\div W_{(k)}) D_{(k)}+(W_{(k)}\cdot \nabla)D_{(k)} = \frac{1}{\mu} \p_t \left(\psi_{(k_1)}^2 \phi_{(k)}^2 k_2\right)  , \\
		&\displaystyle \div ( W_{(k)}\otimes D_{(k)}) =(\div D_{(k)}) W_{(k)}+(D_{(k)}\cdot \nabla)W_{(k)} =0.
	\end{array}
\end{equation}
Thus, when combining with the temporal correctors,
module a pressure term,
the time derivative can be moved onto the low oscillating amplitudes
(see \eqref{utem}, \eqref{btem} below).

Because the intermittent flow $W_{(k)}$ is not divergence-free,
we also need the corrector
\begin{equation}
	\begin{aligned}
		\label{corrector vector}
		\wt W_{(k)}^c := \frac{1}{\lambda^2N_{ \Lambda }^2} \nabla\psi_{(k_1)}\times\curl(\Phi_{(k)} k_1) .
	\end{aligned}
\end{equation}
Then, straightforward computations show that
\begin{equation}\label{wcwc}
	  W_{(k)} + \wt W_{(k)}^c
	=\curl \curl \left(\frac{1}{\lambda^2N_{\Lambda}^2 } \psi_{(k_1)}\Phi_{(k)} k_1\right)
	=\curl \curl W^c_{(k)},
\end{equation}
where $W^c_{(k)}$ is given by
\begin{align} \label{Vk-def}
	W^c_{(k)} := \frac{1}{\lambda^2N_{\Lambda}^2 } \psi_{(k_1)}\Phi_{(k)} k_1.
\end{align}
Thus, it follows that
\begin{align} \label{div-Wck-Wk-0}
	\div (W_{(k)}+ \wt W^c_{(k)}) =0.
\end{align}

Regarding the magnetic  flows,
we introduce the correctors
\begin{equation}\label{dkc}
		D_{(k)}^c:=\frac{1}{\lambda^2N_{\Lambda}^2}\psi_{(k_1)}\Phi_{(k)}k_2,
    \quad \wt D_{(k)}^c:=-\frac{1}{ \lbb^2N_\Lambda^2}\Delta \psi_{(k_1)} \Phi_{(k)} k_2,\quad  k\in \Lambda_B.
\end{equation}
It holds that
\begin{align} \label{D-wtD-Dc}
   D_{(k)} + \wt D_{(k)}^c
	=\curl \curl D^c_{(k)}.
\end{align}

Lemma \ref{buildingblockestlemma} below contains the key estimates of the intermittent
velocity and magnetic flows.

\begin{lemma} [Estimates of spatial intermittency] \label{buildingblockestlemma}
	For $p \in [1,+\infty]$, $N,\,M \in \mathbb{N}$, we have
	\begin{align}
		&\left\|\nabla^{N} \partial_{t}^{M} \psi_{(k_1)}\right\|_{C_t L^{p}_{x}}
		\lesssim r_{\|}^{\frac 1p- \frac 12}\left(\frac{r_{\perp} \lambda}{r_{\|}}\right)^{N}
          \left(\frac{r_{\perp} \lambda \mu}{r_{\|}}\right)^{M}, \label{intermittent estimates} \\
		&\left\|\nabla^{N} \phi_{(k)}\right\|_{L^{p}_{x}}+\left\|\nabla^{N} \Phi_{(k)}\right\|_{L^{p}_{x}}
		\lesssim r_{\perp}^{\frac 1p- \frac 12}  \lambda^{N}, \label{intermittent estimates2}
	\end{align}
	where the implicit constants are independent of $\rs,\,\rp,\,\lambda$ and $\mu$.
	In particular, we have
	\begin{align}
		&\displaystyle\left\|\nabla^{N} \partial_{t}^{M} W_{(k)}\right\|_{C_t  L^{p}_{x}}
          +\frac{r_{\|}}{r_{\perp}}\left\|\nabla^{N} \partial_{t}^{M} \wt W_{(k)}^{c}\right\|_{C_t L^{p}_{x}}
          +\lambda^{2}\left\|\nabla^{N} \partial_{t}^{M} W_{(k)}^c\right\|_{C_t L^{p}_{x}}\displaystyle \nonumber \\
		&\qquad \lesssim r_{\perp}^{\frac 1p- \frac 12} r_{\|}^{\frac 1p- \frac 12} \lambda^{N}
          \left(\frac{r_{\perp} \lambda \mu}{r_{\|}}\right)^{M}, \ \ k\in \Lambda_u \cup \Lambda_B, \label{ew}  \\
		&\displaystyle\left\|\nabla^{N} \partial_{t}^{M} D_{(k)}\right\|_{C_t  L^{p}_{x}}
        +\frac{r_{\|}}{r_{\perp}}\left\|\nabla^{N} \partial_{t}^{M} \wt D_{(k)}^{c}\right\|_{C_t L^{p}_{x}}
         +\lambda^{2}\left\|\nabla^{N} \partial_{t}^{M} D_{(k)}^c\right\|_{C_t L^{p}_{x}}   \nonumber \\
		&\qquad  \lesssim r_{\perp}^{\frac 1p- \frac 12} r_{\|}^{\frac 1p- \frac 12}
          \lambda^{N}\left(\frac{r_{\perp} \lambda \mu}{r_{\|}}\right)^{M},\ \ k\in \Lambda_B.    \label{ed}
	\end{align}
\end{lemma}

\begin{proof}
	By the definitions of $ \psi_{(k_1)}$, $\phi_{(k)}$ and $\Phi_{(k)}$,
	the $L^\infty$-norms can be bounded by
	\begin{align}
		&\left\|\nabla^{N} \partial_{t}^{M} \psi_{(k_1)}\right\|_{C_{t,x}}
          \lesssim r_{\|}^{-\frac 1 2}\left(\frac{r_{\perp} \lambda}{r_{\|}}\right)^{N}\left(\frac{r_{\perp} \lambda \mu}{r_{\|}}\right)^{M},  \label{naNptMpsik1-L9}\\
		&\left\|\nabla^{N} \phi_{(k)}\right\|_{L^{\infty}_x}+\left\|\nabla^{N} \Phi_{(k)}\right\|_{L^{\infty}_x}
          \lesssim \lambda^{N}r_{\perp}^{-\frac 1 2}.     \label{naNphik-L9}
	\end{align}
	
	In order to control the $L^p$-norm, $1\leq p<+\9$,
	let us estimate the volume of the support of $\psi_{(k_1)}$ below.
	As $\psi_{(k_1)}$ is $(\mathbb{T}/(\lambda\rs)^3)$-periodic,
	it suffices to estimate the volume of the support of $\psi_{(k_1)}$ on the cubes with side length $2\pi/(\lambda\rs)$
	and then multiply the resulting estimate by $(\lambda\rs)^3$.
	One may let the  oscillation parameter $\mu=0$ without changing the volume of the support of $\psi_{(k_1)}$.
	
	In each one of these cubes,
	$\psi_{(k_1)}$ is supported on at most $2N_{\Lambda}$ thickened planes with thickness
    $\sim {\rp}/{(\lambda\rs)}$,
	which are parallel to each other and perpendicular to the direction $k_1$.
	Hence,
	the volume of the support of $\psi_{(k_1)}$ on $\mathbb{T}^3$ can be bounded by
	\begin{equation}\label{suppsi}
		|\supp \psi_{(k_1)}|
		\lesssim (\lambda\rs)^{-2}\frac{\rp}{\lambda\rs} (\lambda\rs)^3\lesssim \rp.
	\end{equation}
	This along with \eqref{naNptMpsik1-L9} yields that
	\begin{align}\label{psi}
		\|\nabla^{N} \partial_t^M \psi_{(k_1)}\|_{C_t L^p_x}
		\lesssim&  |\supp\psi_{(k_1)}|^{\frac 1p}\, \left\|\nabla^{N} \partial_t^M \psi_{(k_1)}\right\|_{C_t L^{\infty}_x} \notag \\
		\lesssim&  \rp^{\frac 1p-\frac 12} \(\frac{\rs \lbb}{\rp}\)^N \(\frac{\rs \lbb \mu}{\rp}\)^M.
	\end{align}
	Since there are only finite wavevectors in $\Lambda_u$ and $\Lambda_B$,
	the implicit constant can be taken independent of $k_1$ by taking the maximum over the wavevector sets
	$\Lambda_u$ and $\Lambda_B$.
	Thus, \eqref{intermittent estimates} is verified.
	
	Estimate \eqref{intermittent estimates2} can be proved similarly.
	Actually,
	in each of the cubic with length $2\pi/(\lbb \rs)$,
	$\phi_{(k)}$ is supported on at most $2N_{\Lambda}$ thickened planes
	with thickness $\sim \lbb^{-1}$,
	perpendicular to the direction $k$.
	Hence,
	\begin{align*}
		|\supp \phi_{(k)}| \lesssim (\lbb \rs)^{-2} \lbb^{-1} (\lbb \rs)^3
		\lesssim \rs,
	\end{align*}
	which along with \eqref{naNphik-L9}  yields that
	\begin{align*}
		\|\na^N \phi_{(k)}\|_{L^p_x}
		\lesssim |\supp \phi_{(k)}|^{\frac 1p} \|\na^N \phi_{(k)}\|_{L^\9_x}
		\lesssim \rs^{\frac 1p} \lbb^N \rs^{-\frac 12}
		\lesssim \rs^{\frac 1p-\frac 12} \lbb^N.
	\end{align*}
	Thus,  \eqref{intermittent estimates2} is verified.
	
Finally, estimates \eqref{ew} and \eqref{ed} follow  from \eqref{intermittent estimates}, \eqref{intermittent estimates2}
and Fubini's theorem.
\end{proof}

The following lemma illustrates that,
more intermittency can be gained for
the interactions betweens different intermittent building blocks.
Intuitively,  the supports of  intermittent building blocks can be regarded as thin cuboids,
because $k_2\not = k_2'$ for different wavevectors,
the intersection of  two nonparallel thin cuboids has much smaller support.

\begin{lemma}[Product estimate]  \label{product estimate lemma}
For $k \neq k'\in \Lambda_{u}\cup \Lambda_{B}$ and $p \in [1, \infty]$, we have
\begin{equation}  \label{intersect-phik1-phik1'}
\|\psi_{(k_1)}\phi_{(k)}\psi_{(k'_1)}\phi_{(k')} \|_{C_tL^p_x}\lesssim \rs^{\frac{1}{p}-1}\rp^{\frac{2}{p}-1},
\end{equation}
where the implicit constants are independent of the parameters $\rs,\,\rp$ and $\lambda$.
\end{lemma}

\begin{proof}
As in the proof of Lemma~\ref{buildingblockestlemma},
let us start with the $L^\infty$-estimate.
By \eqref{snp},
\begin{equation}   \label{inter-psik-psik'-bdd}
\|\psi_{(k_1)}\phi_{(k)} \psi_{(k'_1)}\phi_{(k')}\|_{C_{t,x}}\lesssim \rs^{-1}\rp^{-1}.
\end{equation}

Next, we estimate the volume of the support of $\psi_{(k_1)}\phi_{(k)}\psi_{(k'_1)}\phi_{(k')}$.
As $ \psi_{(k_1)}\phi_{(k)}\psi_{(k'_1)}\phi_{(k')}$ is $(\mathbb{T}/\lambda\rs)^3$-periodic,
it suffices to consider the support on the
small cubes with side length $2\pi/\lambda\rs$ and then multiply the result by $(\lambda\rs)^3$.

In each of these cubes, the support of $\psi_{(k_1)}\phi_{(k)}$ consists of at most $2N_\Lambda$
parallel thin cuboids with length $\sim (\lambda\rs)^{-1}$,
width $\displaystyle\sim {\rp}{(\lambda\rs)^{-1}}$ and height $\sim \lambda^{-1}$.
Note that, the width and height is much smaller than the length.
Since $k_2 \not = k_2'$,
the intersections between different  supports of $\psi_{(k_1)}\phi_{(k)}$
and  $\psi_{(k_1')}\phi_{(k')}$
are contained in much smaller cuboids,
with the length and width bounded by $\displaystyle\sim {\rp}{(\lambda\rs)^{-1}}$ and
with the height $\sim \lambda^{-1}$.
As there are at most $2N_{\Lambda}$ cuboids in the support of $\psi_{(k_1)}\phi_{(k)}$,
the number of intersections for two distinct cuboids is at most $4N_{\Lambda}^2$.
Thus, the volume of the support of $ \psi_{(k_1)}\phi_{(k)}\psi_{(k'_1)}\phi_{(k')}$ is bounded by
\begin{equation} \label{psupp}
	 \abs{\supp(\psi_{(k_1)}\phi_{(k)}\psi_{(k'_1)}\phi_{(k')})}
     \lesssim \rp^2(\lambda\rs)^{-2} \lambda^{-1} (\lambda\rs)^3\lesssim \rp^2\rs.
\end{equation}
The implicit constant can be taken independent of $k_1$ and $k_1'$,
as there are only finitely many vectors in the wavevector sets $\Lambda_{u}$ and $\Lambda_{B}$.

Therefore, by virtue of \eqref{inter-psik-psik'-bdd}  and \eqref{psupp},
we arrive at
\begin{align}
\|\psi_{(k_1)}\phi_{(k)}\psi_{(k'_1)}\phi_{(k')}\|_{C_tL^p_x}
 \leq& \sup_t|\supp\,(\psi_{(k_1)}\phi_{(k)}\psi_{(k'_1)}\phi_{(k')})|^{\frac 1p}\,
       \| \psi_{(k_1)}\phi_{(k)}\psi_{(k'_1)}\phi_{(k')}\|_{C_{t,x}} \notag \\
 \lesssim&  \rs^{\frac 1p -1} \rp^{\frac 2p -1},
\end{align}
which yields \eqref{intersect-phik1-phik1'} and finishes the proof.
\end{proof}

\subsection{Temporal building blocks.}
As mentioned in Section 1, the previous construction of intermittent velocity and magnetic  flows
cannot provide us with enough intermittency to handle the
stronger viscosity and resistivity  $(-\Delta)^{\a_i}$ when $\a_i \in [1,5/4)$, $i=1,2$.

The key idea here is to exploit more intermittency
from temporal oscillations.
Two new parameters $\tau$ and $\sigma$  will be introduced to
parameterize the  concentrations of
the temporal function.

More precisely, let $g\in C_c^\infty(\T)$ be any cut-off function such that
\begin{align*}
	\aint_{\T} g^2(t) \d t=1.
\end{align*}
Let $\tau \in \mathbb{N}_+$ and rescale the cut-off function $g$ by
\begin{align}\label{gk1}
	g_\tau(t)=\tau^{\frac 12} g(\tau t).
\end{align}
Then, we periodize $g_\tau$ such that the resulting functions
(by an abuse of notion, still denoted by $g_\tau$)
are $\T-$periodic functions.

In order to construct the temporal correctors  $w_{q+1}^{(o)}$ and $d_{q+1}^{(o)}$ (see \eqref{wo}-\eqref{do})
to balance the high temporal oscillations
arising from the concentration function $g_{(\tau)}$,
we also need the function
$h_\tau: \mathbb{T} \to \bbr$, defined by
\begin{align} \label{hk}
	h_\tau(t):= \int_{-\pi}^t \left(g_\tau^2(s)  - 1\right)\ ds.
\end{align}

Set
\begin{align}\label{gk}
	\g:=g_\tau(\sigma t),\ \
    h_{(\tau)}(t):= h_\tau(\sigma t).
\end{align}

Note that, $h_{(\tau)}$ has the uniform $L_t^\9$-bound
\begin{align}\label{hk-est}
	\|h_{(\tau)}\|_{L_t^\infty(\T)}\leq 1.
\end{align}
Moreover,
we have the following  temporal intermittent estimates of $\g$.

\begin{lemma} [Estimates of temporal intermittency]   \label{Lem-gk-esti}
	For  $\gamma \in [1,+\infty]$, $M  \in \mathbb{N}$,
	we have
	\begin{align}
		\left\|\partial_{t}^{M}\g \right\|_{L^{\gamma}_t} \lesssim \sigma^{M}\tau^{M+\frac12-\frac{1}{\gamma}},\label{gk estimate}
	\end{align}
	where the implicit constants are independent of $\tau$ and $\sigma$.
\end{lemma}
\begin{proof}
	As in the proof of Lemma~\ref{buildingblockestlemma},
	we first estimate the $L^\infty$-norm of $\g$.
	By definition,
	\begin{align}\label{ginfty}
		\left\|\partial_{t}^{M}\g \right\|_{L^{\9}_t(\T)} \lesssim \sigma^{M}\tau^{M+\frac12}.
	\end{align}
	Next, we give a bound on the volume of the support of $\g$ on $\T$.
	Since $\g$ is $\T/\sigma$-periodic,
     it is sufficient to obtain the estimate on the interval $[-\pi/\sigma,\pi/\sigma]$
	and then multiply the result by $\sigma$.
	Note that, in the interval $[-\pi/\sigma,\pi/\sigma]$,
	$\g$ is supported on an interval with length $\sim (\sigma\tau)^{-1}$,
	and thus the volume of the support of $\g$ on $\T$
	is bounded by
	\begin{align}\label{gsupp}
		|\supp \g| \lesssim (\sigma\tau)^{-1} \sigma  \lesssim \tau^{-1}.
	\end{align}
	Combining \eqref{ginfty} and \eqref{gsupp} together we obtain
	\begin{align}\label{glq1}
		\| \partial_t^M \g\|_{L^\gamma_t}  \lesssim  |\supp\g|^{\frac{1}{\gamma}} \left\| \partial_t^M \g\right\|_{L^{\infty}_t(\T)}
		\lesssim  \sigma^{M}\tau^{M+\frac12-\frac{1}{\gamma}}.
	\end{align}
	Therefore, estimate \eqref{gk estimate} is proved.
\end{proof}

\section{Velocity and magnetic perturbations}   \label{Sec-Pert}

This section is devoted to the construction of the key velocity and magnetic perturbations,
including the principal parts,
the incompressibility correctors,
the spatial and temporal correctors.
To begin with,
let us first define  suitable amplitudes of perturbations,
which are the key to decrease the effects of old velocity and magnetic stresses.

\subsection{Amplitudes}   \label{Subsec-Amplitude}

\subsubsection{\bf The magnetic amplitudes.}

Let $\chi: [0, +\infty) \to \mathbb{R}$ be a smooth cut-off function such that
\begin{equation}\label{e4.0}
	\chi (z) =
	\left\{\aligned
	& 1,\quad 0 \leq z\leq 1, \\
	& z,\quad z \geq 2,
	\endaligned
	\right.
\end{equation}
and
\begin{equation}\label{e4.1}
	\frac 12 z \leq \chi(z) \leq 2z \quad \text{for}\quad z \in (1,2).
\end{equation}

Set
\begin{equation}\label{rhob}
	\rho_B(t,x) := 2 \varepsilon_B^{-1}  \delta_{q+ 1} \chi\left( \frac{|\mathring{R}_{\ell}^B(t, x) |}{\delta_{q+1} } \right),
\end{equation}
where $\varepsilon_{B}$ is the small radius in the geometric Lemma \ref{geometric lem 1}.
Then, by \eqref{e4.0}, \eqref{e4.1} and \eqref{rhob},
\begin{equation}\label{rhor}
	\left|  \frac{\mathring{R}_{\ell}^B}{\rho_B} \right|
     = \left| \frac{\mathring{R}_{\ell}^B}{2 \varepsilon_B^{-1} \delta_{q+ 1}\chi
          (  \delta_{q+1} ^{-1} |\mathring{R}_{\ell}^B | )} \right| \leq \varepsilon_B,
\end{equation}
and for any $p\in[1,+\infty]$,
\begin{align}
	\label{rhoblowbound}
	&\rho_B\geq  \varepsilon_B^{-1} \delta_{q+ 1},\\
	\label{rhoblp}
	&\norm{ \rho_B }_{L^p_{t,x}} \leq 8  \varepsilon_B^{-1} \(  (16\pi^4)^{\frac{1}{p}}\delta_{q+1}  + \norm{\mathring{R}_{\ell}^{B} }_{L^p_{t,x}} \).
\end{align}
Furthermore, by the standard H\"{o}lder estimate (see \cite[(130)]{bdis15}),
the iterative estimate \eqref{rubc1} and \eqref{rhoblowbound}, for $1\leq N\leq 4$,
\begin{align}
	& \norm{ \rho_B }_{C_{t,x}} \lesssim  \ell^{-1} , \quad   \norm{ \rho_B }_{C_{t,x}^N}  \lesssim \ell^{-N}\delta^{-N+1}_{q+1}, \label{rhoB-Ctx.1}\\
	&\norm{ \rho_B^{1/2}}_{C_{t,x}} \lesssim \ell^{-1}, \quad  \norm{  \rho_B^{1/2} }_{C_{t,x}^N} \lesssim \ell^{-N} \delta^{-2N}_{q+1},  \label{rhoB-Ctx.2} \\
	&\norm{ \rho_B^{-1}}_{C_{t,x}}\lesssim \delta_{q+ 1}^{-1} ,
      \quad\norm{ \rho_B^{-1} }_{C_{t,x}^N } \lesssim \ell^{-N}\delta^{-2N}_{q+1} .  \label{rhoB-Ctx.3}
\end{align}

Moreover,
we choose the smooth temporal cut-off function $f_B$ such that
\begin{itemize}
	\item $0\leq f_B\leq 1$, $f_B \equiv 1$ on $\supp_t \rr_{\ell}^B$;
	\item $\supp_t f_B \subseteq N_l(\supp_t \rr_{\ell}^B)$;
	\item $\|f_B\|_{C_t^N}\lesssim \ell^{-N}$,\ \  $1\leq N\leq 4$.
\end{itemize}

The amplitudes of the magnetic perturbations are defined by
\begin{equation}\label{akb}
	a_{(k)}(t,x):= a_{k, B}(t,x)
    =  \rho_B^{\frac{1}{2} } (t,x) f_B(t)\gamma_{(k)}
       \left(\frac{-\mathring{R}_{\ell}^B(t,x)}{\rho_B(t,x)}\right), \quad k \in \Lambda_B,
\end{equation}
where $\gamma_{(k)}$ is the smooth function in the geometric Lemma~\ref{geometric lem 1}.

Note that, by virtue of the geometric Lemma~\ref{geometric lem 1},
the identity \eqref{wdp1}
and the expression \eqref{akb} of $a_{(k)}$,
the following algebraic  identity holds:
\begin{align}\label{magcancel}
	&  \sum\limits_{ k \in  \Lambda_B} a_{(k)}^2 \g^2
	( D_{(k)} \otimes W_{(k)} - W_{(k)} \otimes  D_{(k)} ) \notag\\
	= & -\mathring{R}_{\ell}^B
	+  \sum\limits_{ k \in \Lambda_B}  a_{(k)}^2\g^2\P_{\neq 0}(  D_{(k)} \otimes W_{(k)} -W_{(k)} \otimes  D_{(k)} ) \notag\\
	& + \sum_{k \in \Lambda_B}  a_{(k)}^2 (\g^2-1) \aint_{\T^3}D_{(k)}\otimes W_{(k)}-W_{(k)}\otimes D_{(k)}\d x  ,
\end{align}
where $\P_{\neq 0}$ denotes the spatial projection onto nonzero Fourier modes.

Moreover, in view of \eqref{rhoblowbound}-\eqref{rhoB-Ctx.3}, we also have the following analytic estimates.

\begin{lemma} [Estimates of magnetic amplitudes] \label{mae}
For $1\leq N\leq 4$, $k\in \Lambda_B$, we have
	\begin{align}
		\label{e3.15}
		&\norm{a_{(k)}}_{L^2_{t,x}} \lesssim \delta_{q+1}^{\frac{1}{2}} ,\\
		\label{mag amp estimates}
		& \norm{ a_{(k)} }_{C_{t,x}} \lesssim \ell^{-1},\ \ \norm{ a_{(k)} }_{C_{t,x}^N} \lesssim \ell^{-4N}.
	\end{align}
\end{lemma}

\begin{proof}
	To obtain the $L^2_{t,x}$-estimate of $a_{(k)}$, $k\in \Lambda_{B}$,
	we use \eqref{rub} and \eqref{rhoblp} to bound
	\begin{align}
		\norm{a_{(k)}}_{L^2_{t,x}}
        &\leq \norm{ \rho_B }_{L^1_{t,x}}^{\frac{1}{2}} \norm{ \gamma_{(k)} }_{C_x(B_{\varepsilon_B}(0))} \|f_B\|_{C_t}   \notag\\
		&\leq M_* (8 \varepsilon_B^{-1})^{\frac{1}{2}}(16\pi^4  \delta_{q+1}  + \norm{\mathring{R}_{\ell}^{B} }_{L^1_{t,x}})^{\frac{1}{2}} \notag \\
		&\leq  M_* (8 \varepsilon_B^{-1})^{\frac{1}{2}}( 16\pi^4   + 1)^{\frac12} \delta_{q+1}^{\frac{1}{2}},
	\end{align}
	which yields \eqref{e3.15}.
	
	Moreover, for $N\geq 1$,
    by \eqref{rhoB-Ctx.2} and the standard H\"older estimate (cf. \cite[(129)]{bdis15}),
	\begin{align}
		\norm{a_{(k)}}_{C^N_{t,x}}
		&\lesssim \sum_{0\leq N_1+N_2+N_3\leq N}\norm{\rho_B^{\frac12} }_{C^{N_1}_{t,x}}
		\norm{  \gamma_{(k)} (-\rho_B^{-1} \rr^B_{\ell} ) }_{C^{N_2}_{t,x}} \norm{f_B}_{C_t^{N_3}} \notag\\
		&\lesssim \sum_{0\leq N_1+N_2+N_3\leq N} \ell^{-N_1-1}\ell^{-N_3}
		\left(1+ \norm{\rho_B^{-1}\rr^B_{\ell}}_{C^{N_2}_{t,x}} +  \norm{\rho_B^{-1}\rr^B_{\ell}}_{C_{t,x}}^{N_2-1}
		\norm{\rho_B^{-1}\rr^B_{\ell}}_{C^{N_2}_{t,x}}\right ) .\label{e4.21}
	\end{align}
	Since by \eqref{rubc1}, \eqref{e3.92} and \eqref{rhoB-Ctx.3},
	\begin{align}\label{e4.14}
		\norm{\rho_B^{-1}\rr^B_{\ell}}_{C^{N_2}_{t,x}}
		&\lesssim \sum_{0\leq N_{21}+N_{22} \leq N_2}\norm{\rho_B^{-1}}_{C^{N_{21}}_{t,x}}\norm{\rr^B_{\ell}}_{C^{N_{22}}_{t,x}} \nonumber \\
		& \lesssim \sum_{0\leq N_{21}+ N_{22}\leq N_2}\ell^{-N_{21}-1}\ell^{-N_{22}}\lesssim \ell^{-N_{2}-1},
	\end{align}
	we get
	\begin{align}\label{e4.23}
		\norm{a_{(k)}}_{C^N_{t,x}} &\lesssim \sum_{0\leq  N_1+N_2+N_3\leq N}
		\ell^{-N_1-1}\ell^{-N_3}(\ell^{-N_2-1}+\ell^{-2N_2} )\lesssim \ell^{-4N}.
	\end{align}
	
	The $C_{t,x}$-estimate of $a_{(k)}$ can be bounded by,
    via \eqref{rhoB-Ctx.2},
	\begin{align*}
		\norm{a_{(k)}}_{C_{t,x}}
		&\lesssim \norm{ \rho_B^{\frac{1}{2}} }_{C_{t,x}}\norm{f_B}_{C_t} \norm{ \gamma_{(k)} }_{C_x(B_{\varepsilon_B}(0))}
		\lesssim \ell^{-1}.
	\end{align*}
	Therefore, the estimates in  \eqref{mag amp estimates} are verified.
	The proof is complete.
\end{proof}

\subsubsection{\bf The velocity amplitudes.}

Below we define the velocity amplitudes.
Unlike in the previous magnetic case,
because of the strong coupling   between the velocity and magnetic fields,
a new matrix $\mathring{G}^{B}$ is  needed
in order to maintain the cancellation between the perturbations and the old stresses,
\begin{equation}
	\label{def:G}
	\mathring{G}^{B}: = \sum_{k \in \Lambda_B}a_{(k)}^2 \aint_{\mathbb{T}^3} W_{(k)} \otimes W_{(k)} - D_{(k)} \otimes D_{(k)}\d x.
\end{equation}
In view of estimates \eqref{e3.15} and \eqref{mag amp estimates},
we have that for $N\geq 1$,
\begin{align}
	\label{G estimates}
	\norm{\mathring{G}^{B}}_{C_{t,x}} \lesssim \ell^{-2},\quad
    \norm{\mathring{G}^{B}}_{C_{t,x}^N} \lesssim \ell^{-4N-1} ,\quad
    \norm{\mathring{G}^{B}}_{L^1_{t,x}} &\lesssim \delta_{q+1}.
\end{align}

Set
\begin{equation}\label{defrho}
	\begin{aligned}
		&  \rho_u(t,x):= 2 \varepsilon_u^{-1}  \delta_{q+ 1}\chi
              \left( \frac{|\mathring{R}_{\ell}^u(t,x) + \mathring{G}^{B}(t,x)|}{\delta_{q+1}} \right).
	\end{aligned}
\end{equation}
Then, by \eqref{e4.0} and \eqref{e4.1},
\begin{align}  \label{R-G-veu}
	& \left|  \frac{\mathring{R}_{\ell}^u + \mathring{G}^{B}}{\rho_u} \right| \leq \varepsilon_u.
\end{align}

We choose the smooth temporal cut-off function $f_u$ such that
\begin{itemize}
	\item $0\leq f_u\leq 1$, $f_u\equiv 1$ on $\supp_t \rr_{\ell}^u\cup \supp_t \mathring{G}^B$;
	\item $\supp_t f_u\subseteq N_l(\supp_t \rr_{\ell}^u\cup \supp_t \mathring{G}^B)
                       \subseteq N_{2l}(\supp_t \rr_{\ell}^u\cup \supp_t \mathring{R}_{\ell}^B)$;
	\item $\|f_u\|_{C_t^N}\lesssim \ell^{-N}$, $1\leq N\leq 4$.
\end{itemize}
The velocity amplitudes are defined by
\begin{equation}\label{velamp}
	\begin{aligned}
		&a_{(k)} (t,x) := a_{k, u}(t, x)
		= \rho_u^{\frac{1}{2}} f_u(t)\gamma_{(k)}
           \left(\Id - \frac{\mathring{R}_{\ell}^u(t,x) +  \mathring{G}^{B}(t,x)}{\rho_u(t,x)} \right),
		\quad k \in \Lambda_u.
	\end{aligned}
\end{equation}
In view of the normalization \eqref{wdp}, the geometric Lemma~\ref{geometric lem 2}
and the expression \eqref{velamp} of $a_{(k)}$, we get the following algebraic identity:
\begin{align}\label{velcancel}
	\sum\limits_{ k \in  \Lambda_u} a_{(k)}^2 \g^2
	W_{(k)} \otimes W_{(k)}
	& = \rho_u f^2_u \Id - \mathring{R}_{\ell}^u -  \mathring{G}^{B}
	+  \sum\limits_{ k \in \Lambda_u}  a_{(k)}^2 \g^2 \P_{\neq 0}(  W_{(k)} \otimes  W_{(k)} )\notag\\
	& \quad  +\sum_{k\in \Lambda_u}a_{(k)}^{2}\left(\g^2-1 \right)\aint_{\T^3}W_{(k)}\otimes W_{(k)}\d x .
\end{align}

Moreover, similarly to \eqref{rhoblowbound}-\eqref{rhoB-Ctx.3},
we have for any $p \in [1, +\infty)$,
\begin{align}\label{rhoulowbound}
	&\rho_u\geq   \varepsilon_u^{-1} \delta_{q+ 1}, \\
	\label{rhou lp estimate}
	&\norm{ \rho_u }_{L^p_{t,x}} \leq 8  \varepsilon_u^{-1}
       \left(  (16\pi^4)^{\frac{1}{p}} \delta_{q+1} + \norm{\mathring{R}_{\ell}^u +  \mathring{G}^{B}}_{L^p_{t,x}} \right),
\end{align}
and  for $1\leq N\leq 4$,
\begin{align}
	& \norm{ \rho_u }_{C_{t,x}} \lesssim  \ell^{-2} , \quad   \norm{ \rho_u }_{C_{t,x}^N}  \lesssim \ell^{-5N} \delta_q^{-N+1},  \label{rhou-Ctx.1} \\
	&\norm{ \rho_u^{1/2}}_{C_{t,x} } \lesssim \ell^{-1}, \quad  \norm{  \rho_u^{1/2} }_{C_{t,x}^N} \lesssim \ell^{-5N} \delta_{q+1}^{-2N} ,  \label{rhou-Ctx.2} \\
	&\norm{ \rho_u^{-1}}_{C_{t,x}}\lesssim  \delta_{q+ 1}^{-1} , \quad\norm{ \rho_u^{-1} }_{C_{t,x}^N } \lesssim \ell^{-5N}\delta_{q+1}^{-2N} . \label{rhou-Ctx.3}
\end{align}

Arguing as in the proof of Lemma~\ref{mae},
using the standard H\"older estimate (cf. \cite[(129)]{bdis15})
and estimates \eqref{rub}, \eqref{G estimates},  \eqref{rhoulowbound}-\eqref{rhou-Ctx.3}
we have the following analytic estimates for the velocity amplitude functions.
\begin{lemma} [Estimates of velocity amplitudes]  \label{vae}
	For $1\leq N\leq 4$, $k\in \Lambda_{u} $,
	we have
	\begin{align}\label{akul2}
		&\norm{a_{(k)}}_{L^2_{t,x}} \lesssim \delta_{q+1}^{\frac{1}{2}} ,\\
		\label{akucn}
		& \norm{ a_{(k)} }_{C_{t,x}} \lesssim \ell^{-1},\quad\norm{a_{(k)}  }_{C_{t,x}^N} \lesssim \ell^{-8N}.
	\end{align}
\end{lemma}

\subsection{Principal parts of perturbations}

We define the principal parts $w_{q+1}^{(p)}$ and $d_{q+1}^{(p)}$,
respectively,
of the velocity
and  magnetic perturbations by
\begin{subequations}\label{pp}
	\begin{align}
		w_{q+1}^{(p)} &:= \sum_{k \in \Lambda_u \cup \Lambda_B } a_{(k)}\g W_{(k)},
		\label{pv}\\
		d_{q+1}^{(p)} &:= \sum_{k \in \Lambda_B} a_{(k)}\g D_{(k)} .
		\label{ph}
	\end{align}
\end{subequations}

By the algebraic identity \eqref{magcancel},
\begin{align} \label{mag oscillation cancellation calculation}
	 & d_{q+ 1}^{(p)} \otimes w_{q+ 1}^{(p)} - w_{q+1}^{(p)}\otimes d_{q+ 1}^{(p)}\notag + \rr^B_{\ell}   \nonumber \\
	=& \sum_{k \in \Lambda_B} a_{(k)}^2\g^2 (D_{(k)}\otimes W_{(k)}-W_{(k)}\otimes D_{(k)}) +  \rr^B_{\ell}  \nonumber \\
	&+ \(\sum_{k \neq k' \in \Lambda_B}+ \sum_{k \in \Lambda_u, k' \in \Lambda_B}\)a_{(k)}a_{(k')}\g^2(D_{(k')}\otimes W_{(k)}-W_{(k)}\otimes D_{(k')}) \notag\\
	=& \sum_{k \in \Lambda_B}  a_{(k)}^2 \g^2 \P_{\neq 0}(D_{(k)}\otimes W_{(k)}-W_{(k)}\otimes D_{(k)}) \nonumber \\
	&+ \sum_{k \in \Lambda_B}  a_{(k)}^2 (\g^2-1) \aint_{\T^3}D_{(k)}\otimes W_{(k)}-W_{(k)}\otimes D_{(k)}\d x \nonumber \\
	&+ \(\sum_{k \neq k' \in \Lambda_B}+ \sum_{k \in \Lambda_u, k' \in \Lambda_B}\)a_{(k)}a_{(k')}\g^2(D_{(k')}\otimes W_{(k)}-W_{(k)}\otimes D_{(k')}).
\end{align}

Moreover, for the nonlinearity  in the velocity equation,
by \eqref{def:G} and  \eqref{velcancel},
\begin{align}  \label{vel oscillation cancellation calculation}
	 & w_{q+ 1}^{(p)} \otimes  w_{q+ 1}^{(p)} - d_{q+ 1}^{(p)} \otimes d_{q+1}^{(p)} + \mathring{R}^u_\ell  \nonumber  \\
	=& \sum_{k \in \Lambda_u} a_{(k)}^2\g^2 W_{(k)}\otimes W_{(k)}
	    + \sum_{k \in \Lambda_B} a_{(k)}^2 \g^2(W_{(k)}\otimes W_{(k)}-D_{(k)}\otimes D_{(k)}) + \mathring{R}^u_\ell  \nonumber \\
	& + \sum_{k \neq k' \in \Lambda_u\cup\Lambda_B }a_{(k)}a_{(k')}\g^2W_{(k)}\otimes W_{(k')}-\sum_{k \neq k' \in \Lambda_B}a_{(k)}a_{(k')}\g^2D_{(k)}\otimes D_{(k')}\nonumber \\
	=&  \rho_{u} f^2_u Id   +\sum_{k \in \Lambda_u} a_{(k)}^2 \g^2\P_{\neq0}( W_{(k)}\otimes W_{(k)}) \nonumber \\
	& + \sum_{k \in \Lambda_B} a_{(k)}^2 \g^2\P_{\neq0} (W_{(k)}\otimes W_{(k)}-D_{(k)}\otimes D_{(k)})\nonumber \\
	& +\sum_{k\in \Lambda_u}a_{(k)}^{2}\left(\g^2-1 \right)\aint_{\T^3}W_{(k)}\otimes W_{(k)}\d x \nonumber \\
	&+ \sum_{k\in \Lambda_B}a_{(k)}^{2}\left(\g^2-1 \right)\aint_{\T^3}W_{(k)}\otimes W_{(k)}-D_{(k)}\otimes D_{(k)}\d x \nonumber \\
	& + \sum_{k \neq k' \in \Lambda_u\cup\Lambda_B }a_{(k)}a_{(k')}\g^2W_{(k)}\otimes W_{(k')}-\sum_{k \neq k' \in \Lambda_B}a_{(k)}a_{(k')}\g^2D_{(k)}\otimes D_{(k')}.
\end{align}

\begin{remark} \label{Rem-principal} \rm
	The key fact here is that,
	module the high frequency part and negligible intersections,
	the nonlinearities of the principal parts
	cancel with the old magnetic stresses $\mathring{R}_{\ell}^B$ and $\mathring{R}_{\ell}^u$,
	which enables to decrease the amplitudes of the old stresses
    and to yield the inductive estimate \eqref{rub}.
\end{remark}

\subsection{Incompressibility correctors}

Because the amplitude functions $\{a_{(k)}, k\in \Lambda_u \cup \Lambda_B\}$ depend on the space,
the principal parts of perturbations are not divergence free.
This leads to the introduction of the incompressibility correctors
\begin{subequations} \label{wqc-dqc}
	\begin{align}
		w_{q+1}^{(c)}
		&:=   \sum_{k\in \Lambda_u \cup\Lambda_B  } \g\left(\curl (\nabla a_{(k)} \times W^c_{(k)})
		+ \nabla a_{(k)} \times \curl W^c_{(k)} +a_{(k)} \wt W_{(k)}^c \right) , \label{wqc} \\
		d_{q+1}^{(c)} &:=   \sum_{k\in \Lambda_B } \g\left(\curl (\nabla a_{(k)} \times D_{(k)}^c)
		+ \nabla a_{(k)} \times \curl D_{(k)}^c
		+a_{(k)}\wdc\right ), \label{dqc}
	\end{align}
\end{subequations}
where $W^c_{(k)}$ and  $\wt W_k^c $  are given by \eqref{Vk-def} and \eqref{corrector vector}, respectively,
and
$D_{(k)}^c$ and $\wt D_{(k)}^c$ are as in \eqref{dkc}.
Then, it follows from \eqref{wcwc} and \eqref{D-wtD-Dc}  that
\begin{subequations}
	\begin{align}
		&  w_{q+1}^{(p)} + w_{q+1}^{(c)}
		=\curl \curl \left(  \sum_{k \in \Lambda_u \cup \Lambda_B} a_{(k)} \g W^c_{(k)} \right), \label{div free velocity} \\
		&  d_{q+1}^{(p)} + d_{q+1}^{(c)}=\curl  \curl \left(  \sum_{k \in \Lambda_B} a_{(k)}\g D_{(k)}^c \right).  \label{div free magnetic}
	\end{align}
\end{subequations}
In particular,
\begin{align} \label{div-wpc-dpc-0}
	\div (w_{q+1}^{(p)} + w_{q +1}^{(c)}) = \div (d_{q+1}^{(p)} + d_{q +1}^{(c)}) =  0,
\end{align}
which justifies the definition of the incompressibility correctors.

\subsection{Temporal correctors}  \label{Subsec-Tem-cor}

Two new types of temporal correctors will be introduced
in order to balance
the high spatial and temporal oscillations in \eqref{mag oscillation cancellation calculation} and \eqref{vel oscillation cancellation calculation}.

\subsubsection{\bf  Temporal correctors to balance spatial oscillations.}

In order to balance the high spatial oscillations in \eqref{mag oscillation cancellation calculation} and \eqref{vel oscillation cancellation calculation}
caused by the spatial concentration function $\psi_{(k_1)}$,
we introduce the temporal correctors $w_{q+1}^{(t)}$ and $d_{q+1}^{(t)}$,
defined by
\begin{subequations}  \label{veltemcor-magtemcor}
	\begin{align}
		&w_{q+1}^{(t)} := -{\mu}^{-1}  \sum_{k\in \Lambda_u\cup \Lambda_B} \P_{H}\P_{\neq 0}(a_{(k)}^2\g^2 \psi_{(k_1)}^2 \phi_{(k)}^2 k_1),\label{veltemcor}\\
		\label{magtemcor}
		&d_{q+1}^{(t)}= -{\mu}^{-1}  \sum_{k\in \Lambda_B}\P_{H}\P_{\neq 0}(a_{(k)}^2\g^2 \psi_{(k_1)}^2 \phi_{(k)}^2 k_2),
	\end{align}
\end{subequations}
where $\P_{H}$ denotes the Helmholtz-Leray projector, i.e.,
$\P_{H}=\Id-\nabla\Delta^{-1}\div. $

Then, by the Leibniz rule,
\eqref{wkwk} and \eqref{wd},
\begin{align} \label{utem}
	&\partial_{t} w_{q+1}^{(t)}+    \sum_{k \in \Lambda_u \cup \Lambda_B}  \P_{\neq 0}
	\(a_{(k)}^{2}\g^2 \div(W_{(k)} \otimes W_{(k)})\)  \nonumber  \\
	=&- {\mu}^{-1}   \sum_{k \in \Lambda_u \cup \Lambda_B}  \P_{H} \P_{\neq 0} \partial_{t}
	\left(a_{(k)}^{2}\g^2 \psi_{(k_1)}^{2} \phi_{(k)}^{2}  k_1\right)  \nonumber \\
	& + {\mu}^{-1}     \sum_{k \in \Lambda_u \cup \Lambda_B}  \P_{\neq 0}
	\(a_{(k)}^{2}\g^2 \partial_{t}(\psi_{(k_1)}^{2} \phi_{(k)}^{2} k_1)\)  \nonumber \\
	=&(\nabla\Delta^{-1}\div)  {\mu}^{-1}  \sum_{k \in \Lambda_u \cup \Lambda_B}  \P_{\neq 0} \partial_{t}
	\(a_{(k)}^{2}\g^2 \psi_{(k_1)}^{2} \phi_{(k)}^{2}  k_1\)   \nonumber  \\
	& - {\mu}^{-1}    \sum_{k \in \Lambda_u \cup \Lambda_B}   \P_{\neq 0}
(\partial_{t}\( a_{(k)}^{2}\g^2)  \psi_{(k_1)}^{2} \phi_{(k)}^{2} k_1\),
\end{align}
and
\begin{align} \label{btem}
	&\partial_{t} d_{q+1}^{(t)}+ \sum_{k\in \Lambda_B} \P_{\neq 0}
	\(a_{(k)}^{2}\g^2 \div(D_{(k)} \otimes  W_{(k)}-W_{(k)} \otimes D_{(k)} )\)   \nonumber  \\
	=&(\nabla\Delta^{-1}\div) {\mu}^{-1}   \sum_{k \in \Lambda_B} \P_{\neq 0} \partial_{t}
	\(a_{(k)}^{2}\g^2 \psi_{(k_1)}^{2} \phi_{(k)}^{2}  k_2\)  \nonumber  \\
	&  - {\mu}^{-1}  \sum_{k \in \Lambda_B}
	\P_{\neq 0}(\partial_{t} \(a_{(k)}^{2}\g^2) \psi_{(k_1)}^{2}\phi_{(k)}^{2} k_2 \).
\end{align}

\begin{remark}   \label{Rem-wtdt-correct} \rm
	The important fact here is that,
	the first terms on the R.H.S. of \eqref{utem} and \eqref{btem}
	can be treated as the pressures
    and can be removed by using the Helmholtz-Leray projector,
	while the remaining terms are of low spatial oscillations
	and thus can be controlled by the large parameter $\mu$.
    See \eqref{I2-esti} and \eqref{J2-esti} below.
\end{remark}

\subsubsection{\bf Temporal correctors to balance temporal oscillations.}

Another new type of temporal correctors
is introduced below,
in order to balance the high temporal oscillations
in \eqref{mag oscillation cancellation calculation} and \eqref{vel oscillation cancellation calculation}
caused by the temporal concentration function $\g$,
\begin{subequations} \label{wo-do-def}
\begin{align}
	& w_{q+1}^{(o)}:= -\sigma^{-1}\sum_{k\in\Lambda_u }\P_{H}\P_{\neq 0}\(h_{(\tau)}\aint_{\T^3} W_{(k)}\otimes W_{(k)}\d x\nabla (a_{(k)}^2) \)\notag\\
	&\qquad\quad\  -\sigma^{-1}\sum_{k\in\Lambda_B }\P_{H}\P_{\neq 0}\(h_{(\tau)}\aint_{\T^3} W_{(k)}\otimes W_{(k)}-D_{(k)}\otimes D_{(k)}\d x\nabla (a_{(k)}^2)\)     ,\label{wo}\\
	& d_{q+1}^{(o)}:= -\sigma^{-1}\sum_{k\in\Lambda_B }\P_{H}\P_{\neq 0}\(h_{(\tau)} )\aint_{\T^3} D_{(k_)}\otimes W_{(k)}-W_{(k)}\otimes D_{(k)}\d x\nabla (a_{(k)}^2)\), \label{do}
\end{align}
where we recall that $h_{(\tau)}$ is given by \eqref{hk}.
\end{subequations}

Then, by the Leibniz rule
and the fact that $\frac{d}{dt}h_{(\tau)}= \sigma(g^2_{(\tau)}-1)$,
\begin{align} \label{utemcom}
	&\partial_{t} w_{q+1}^{(o)}+
	\sum_{k\in \Lambda_u}
	\P_{\neq 0}\(\left(\g^2-1 \right)\aint_{\T^3}W_{(k)}\otimes W_{(k)}\d x \nabla(a_{(k)}^{2}) \) \nonumber  \\
	&+ \sum_{k\in \Lambda_B}\P_{\neq 0}\( \left(\g^2-1 \right)\aint_{\T^3}W_{(k)}\otimes W_{(k)}-D_{(k)}\otimes D_{(k)}\d x\nabla (a_{(k)}^2)\)  \nonumber  \\
	=&\left(\nabla\Delta^{-1}\div\right) \sigma^{-1}  \sum_{k \in \Lambda_u} \P_{\neq 0} \partial_{t}\(h_{(\tau)}\aint_{\T^3}W_{(k)}\otimes W_{(k)}\d x\nabla (a_{(k)}^2)\) \nonumber  \\
	&+\left(\nabla\Delta^{-1}\div\right) \sigma^{-1}  \sum_{k \in \Lambda_B} \P_{\neq 0} \partial_{t}\(h_{(\tau)}\aint_{\T^3}W_{(k)}\otimes W_{(k)}-D_{(k)}\otimes D_{(k)}\d x\nabla (a_{(k)}^2)\) \nonumber  \\
	&-\sigma^{-1}\sum_{k\in \Lambda_u}\P_{\neq 0}\(h_{(\tau)}\aint_{\T^3}W_{(k)}\otimes W_{(k)}\d x\p_t\nabla (a_{(k)}^2)\) \nonumber  \\
	&-\sigma^{-1}\sum_{k\in \Lambda_B}\P_{\neq 0}\(h_{(\tau)}\aint_{\T^3}W_{(k)}\otimes W_{(k)}-D_{(k)}\otimes D_{(k)}\d x\p_t\nabla (a_{(k)}^2)\),
\end{align}
and
\begin{align} \label{btemcom}
	&\partial_{t} d_{q+1}^{(o)}+\sum_{k\in \Lambda_B}
	\P_{\neq 0}\(\left(\g^2-1 \right)\aint_{\T^3}D_{(k)}\otimes W_{(k)}-W_{(k)}\otimes D_{(k)}\d x\nabla(a_{(k)}^{2})\)  \nonumber  \\
	=&\left(\nabla\Delta^{-1}\div\right) \sigma^{-1}  \sum_{k \in \Lambda_B} \P_{\neq 0} \partial_{t}\(h_{(\tau)}\aint_{\T^3}D_{(k)}\otimes W_{(k)}-W_{(k)}\otimes D_{(k)}\d x\nabla(a_{(k)}^{2})\) \nonumber  \\
	&-\sigma^{-1}\sum_{k\in \Lambda_B}\P_{\neq 0}\(h_{(\tau)}\aint_{\T^3}D_{(k)}\otimes W_{(k)}-W_{(k)}\otimes D_{(k)}\d x\p_t\nabla(a_{(k)}^{2})\) .
\end{align}

\begin{remark} \label{Rem-wodo-correct} \rm
	As in the previous spatial oscillation case in Remark \ref{Rem-wtdt-correct},
	the first two terms on the R.H.S. of \eqref{utemcom}
	are the harmless pressures,
	and the remaining terms are of low temporal oscillations
	which can be controlled by the large parameter $\sigma$.
	Similar structure also appear in the formula \eqref{btemcom}.
\end{remark}

\subsection{Velocity and magnetic perturbations}

We are now in position to define the velocity and magnetic perturbations $w_{q+1}$ and $d_{q+1}$
at level $q+1$:
\begin{subequations}   \label{perturbation}
	\begin{align}
		w_{q+1} &:= w_{q+1}^{(p)} + w_{q+1}^{(c)}+ w_{q+1}^{(t)}+\wo,
		\label{velocity perturbation} \\
		d_{q+1} &:= d_{q+1}^{(p)} + d_{q+1}^{(c)}+ d_{q+1}^{(t)}+\dqo.
		\label{magnetic perturbation}
	\end{align}
\end{subequations}
Then, the velocity and magnetic fields at level $q+1$ are defined by
\begin{subequations}
	\label{q+1 iterate}
	\begin{align}
		& u_{q+1}:= u_{\ell} + w_{q+1},
		\label{q+1 velocity}\\
		& B_{q+1}:= B_{\ell}+ d_{q+1} .
		\label{q+1 magnetic}
	\end{align}
\end{subequations}

The estimates of  velocity and magnetic perturbations are summarized in Lemma \ref{totalest}.

\begin{lemma}  [Estimates of perturbations] \label{totalest}
	For any $\rho \in(1,\9), \gamma \in [1,+\infty]$ and integer $0\leq N\leq 4$,
the following estimates hold:
	\begin{align}
		&\norm{\na^N w_{q+1}^{(p)} }_{L^ \gamma_tL^\rho_x } + \norm{\na^N d_{q+1}^{(p)} }_{L^ \gamma_tL^\rho_x   } \lesssim \ell^{-1} \lbb^N\rs^{\frac{1}{\rho}-\frac12}\rp^{\frac{1}{\rho}-\frac12}\tau^{\frac12-\frac{1}{ \gamma}},\label{uprinlp}\\
		&\norm{\na^N w_{q+1}^{(c)} }_{L^\gamma_tL^\rho_x   } +\norm{\na^N d_{q+1}^{(c)} }_{L^\gamma_tL^\rho_x  } \lesssim \ell^{-1}\lbb^N\rs^{\frac{1}{\rho}+\frac12}\rp^{\frac{1}{\rho}-\frac{3}{2}}\tau^{\frac12-\frac{1}{\gamma}}, \label{ucorlp}\\
		&\norm{ \na^Nw_{q+1}^{(t)} }_{L^\gamma_tL^\rho_x   }+ \norm{\na^N d_{q+1}^{(t)} }_{L^\gamma_tL^\rho_x  }\lesssim \ell^{-2}\lbb^N\mu^{-1}\rs^{\frac{1}{\rho}-1}\rp^{\frac{1}{\rho}-1}\tau^{1-\frac{1}{\gamma}} ,\label{dco rlp}\\
		&\norm{\na^N \wo }_{L^\gamma_tL^\rho_x  }\lesssim \ell^{-8N-10}\sigma^{-1},\ \
          \norm{\na^N \dqo }_{L^\gamma_tL^\rho_x  }\lesssim \ell^{-4N-6}\sigma^{-1} ,\label{dcorlp}\\
		& \norm{ w_{q+1}^{(p)} }_{C_{t,x}^N }  + \norm{ w_{q+1}^{(c)} }_{C_{t,x}^N }+\norm{ w_{q+1}^{(t)} }_{C_{t,x}^N }+\norm{ \wo }_{C_{t,x}^N }
          \lesssim \lambda^{\frac{5N}{2}+3},\label{principal c1 est}\\
		&\norm{ d_{q+1}^{(p)} }_{C_{t,x}^N }  + \norm{ d_{q+1}^{(c)} }_{C_{t,x}^N }+\norm{ d_{q+1}^{(t)} }_{C_{t,x}^N }+\norm{ \dqo }_{C_{t,x}^N }
          \lesssim  \lambda^{\frac{5N}{2}+3}.\label{dprincipal c1 est}
	\end{align}
\end{lemma}

\begin{proof}
	
	By Lemmas \ref{buildingblockestlemma}, \ref{Lem-gk-esti}, \ref{mae} and \ref{vae}
    and   $\ell^{-6N_1}\lesssim \lbb^{N_1}$,
    for any $\rho \in (1,+\infty)$,
	\begin{align}\label{uplp}
			& \norm{\nabla^N w_{q+1}^{(p)} }_{L^\gamma_tL^\rho_x } + \norm{\nabla^N d_{q+1}^{(p)} }_{L^\gamma_tL^\rho_x} \notag \\
			\lesssim&  \sum_{k \in \Lambda_u \cup \Lambda_B}
             \sum\limits_{N_1+N_2 = N}
            \|a_{(k)}\|_{C^{N_1}_{t,x}}\|\g\|_{L_t^\gamma}
			\norm{ \nabla^{N_2} W_{(k)} }_{C_tL^\rho_x } \notag \\
			&  + \sum_{k \in \Lambda_B}\sum\limits_{N_1+N_2 = N}\|a_{(k)}\|_{C^{N_1}_{t,x}}\|\g\|_{L_t^\gamma} \norm{\nabla^{N_2} D_{(k)} }_{C_tL^\rho_x}  \notag  \\
		\lesssim&	 \ell^{-1}\lbb^N\rs^{\frac{1}{\rho}-\frac12}\rp^{\frac{1}{\rho}-\frac12}\tau^{\frac12-\frac{1}{\gamma}},
	\end{align}
	which verifies \eqref{uprinlp}.
	
	Moreover, using \eqref{b-beta-ve}, \eqref{ew}, \eqref{wqc} and Lemmas \ref{mae} and \ref{vae}
	we have
	\begin{align} \label{lp vel corrector estimate}
		& \quad \norm{\na^N w_{q+1}^{(c)} }_{L^\gamma_tL^\rho_x} \notag \\
		 \lesssim &  \sum_{k\in \Lambda_u \cup \Lambda_B}
		\left\|\g\na^N\( \curl (\nabla a_{(k)} \times W^c_{(k)})+ \nabla a_{(k)} \times \curl W_{(k)}^c+a_{(k)}\wt W^c_{(k)}\) \right\|_{L^\gamma_tL^\rho_x} \notag \\
		\lesssim&
		\sum\limits_{k\in \Lambda_u \cup \Lambda_B}\|\g\|_{L^\gamma_t}  \sum_{N_1+N_2=N}
         \( \norm{ a_{(k)} }_{C_{t,x}^{N_1+2}} \norm{\na^{N_2} W^c_{(k)}}_{C_tW^{1,\rho}_x }
		 +  \norm{ a_{(k)} }_{C_{t,x}^{N_1}} \norm{ \na^{N_2}\wt W^c_{(k)}}_{C_tL^\rho_x }  \)  \nonumber   \\
		 \lesssim & \sum_{N_1+N_2=N} ( \ell^{-8N_1-17}r_{\perp}^{\frac{1}{\rho} - \frac{1}{2}} r_{\parallel}^{\frac{1}{\rho} - \frac{1}{2}}\lambda^{N_2-1}
           + \ell^{-8N_1-1}\lambda^{N_2}\rs^{\frac{1}{\rho} + \frac{1}{2}} \rp^{\frac{1}{\rho} - \frac{3}{2}})\tau^{\frac12-\frac{1}{\gamma}} \notag \\
		 \lesssim& (\ell^{-17}r_{\perp}^{\frac{1}{\rho} - \frac{1}{2}} r_{\parallel}^{\frac{1}{\rho}
                - \frac{1}{2}}\lambda^{N-1}+\ell^{-1}\rs^{\frac{1}{\rho}
                   + \frac{1}{2}} \rp^{\frac{1}{\rho} - \frac{3}{2}}\lambda^{N})\tau^{\frac12-\frac{1}{\gamma}} \notag \\
		 \lesssim&  \ell^{-1}\lambda^{N}\rs^{\frac{1}{\rho} + \frac{1}{2}} \rp^{\frac{1}{\rho} - \frac{3}{2}}\tau^{\frac12-\frac{1}{\gamma}}.
	\end{align}
	Similarly, by \eqref{ed}, \eqref{mag amp estimates} and \eqref{dqc},
	\begin{align} \label{lp mag corrector estimate}
		&\quad\ \norm{\na^N d_{q+1}^{(c)} }_{L^\gamma_tL^\rho_x} \notag\\
		&\leq   \sum_{ k \in \Lambda_B} \left\|\g\na^N  \( \curl (\nabla a_{(k)} \times D^c_{(k)})
		+ \nabla a_{(k)} \times \curl D^c_{(k)}+ a_{(k)}\wdc\) \right\|_{L^\gamma_tL^\rho_x} \notag  \nonumber  \\
		&\lesssim  \sum\limits_{k\in \Lambda_B} \|\g\|_{L^\gamma_t}
		\sum_{N_1+N_2=N}  (\norm{ a_{(k)} }_{C_{t,x}^{N_1+2}} \norm{\na^{N_2} D^c_{(k)}}_{C_tW^{1,\rho}_x }
		     +   \|a_{(k)}\|_{C^{N_1}_{t,x}}  \|\na^{N_2}\wt D^c_{(k)}\|_{C_tL^\rho_x}  )  \nonumber  \\
		&\lesssim \sum_{N_1+N_2=N}  (\ell^{-4N_1-9} \lbb^{N_2-1} \rs^{\frac 1 \rho - \frac 12} \rp^{\frac 1\rho - \frac 12}
             + \ell^{-4N_1-1} \lbb^{N_2} \rs^{\frac 1 \rho+\frac 12} \rp^{\frac 1 \rho-\frac 32} ) \tau^{\frac12-\frac{1}{\gamma}}  \nonumber   \\
		& \lesssim  \ell^{-1}  \lbb^{N} \rs^{\frac 1 \rho + \frac 12} \rp^{\frac 1 \rho - \frac 32} \tau^{\frac12-\frac{1}{\gamma}}.
	\end{align}
	Thus, we obtain \eqref{ucorlp}.
	
	Regarding the temporal correctors,
by \eqref{veltemcor}, Lemmas \ref{buildingblockestlemma}, \ref{Lem-gk-esti}, \ref{mae} and \ref{vae}
and the boundedness of operators $P_{\not =0}$ and $\mathbb{P}_H$ in $L^\rho$,
\begin{align}
			\label{lp vel time estimate}
			&\norm{\na^N w_{q+1}^{(t)} }_{L^\gamma_tL^\rho_x}  \notag  \\
			\lesssim & \,\mu^{-1}    \sum_{k \in \Lambda_u \cup \Lambda_B}
			\norm{ \g }_{L^{2\gamma}_t }^2
              \sum_{N_1+N_2+N_3=N}
              \|\nabla^{N_1}(a_{(k)}^2)\|_{C_{t,x} }\norm{  \na^{N_2} (\psi_{(k_1)}^2) }_{C_tL^{\rho}_x }\norm{  \na^{N_3}(\phi_{(k)}^2) }_{L^{\rho}_x }  \notag  \\
			\lesssim & \,\mu^{-1}  \tau^{1-\frac{1}{\gamma}}
             \sum_{N_1+N_2+N_3=N} \ell^{-8N_1-2} \lambda^{N_2} r_{\parallel}^{\frac{1}{\rho} -1}\lambda^{N_3} r_{\perp}^{\frac{1}{\rho}-1}    \notag   \\
			\lesssim & \, \ell^{-2} \lambda^{N}\mu^{-1}r_{\perp}^{\frac{1}{\rho}-1} r_{\parallel}^{\frac{1}{\rho} -1}\tau^{1-\frac{1}{\gamma}},
\end{align}
	and similarly,
	\begin{equation}
		\begin{aligned}
			\label{lp mag time estimate}
			\norm{\na^N d_{q+1}^{(t)} }_{L^\gamma_tL^\rho_x}
           &\lesssim  \ell^{-2} \lambda^{N}\mu^{-1}r_{\perp}^{\frac{1}{\rho}-1} r_{\parallel}^{\frac{1}{\rho} -1}\tau^{1-\frac{1}{\gamma}},
		\end{aligned}
	\end{equation}
	which yields \eqref{dco rlp}.
	
	Concerning the estimates of temporal correctors $\wo$ and $\dqo$,
    by \eqref{hk-est}, Lemmas \ref{mae} and \ref{vae},
	\begin{equation}\label{tem-esti}
		\begin{aligned}
			\norm{ \na^N \dqo }_{L^\gamma_tL^\rho_x }
			\lesssim \sigma^{-1}\sum_{k \in \Lambda_B}\|h_{(\tau)}\|_{C_{t}} \|\nabla^{N+1} (a^2_{(k)})\|_{C_{t,x}}
			\lesssim  \ell^{-4N-6} \sigma^{-1},
		\end{aligned}
	\end{equation}
	and
	\begin{equation}\label{wtem-esti}
		\begin{aligned}
			\norm{ \na^N\wo }_{L^\gamma_tL^\rho_x} \lesssim  \ell^{-8N-10} \sigma^{-1}.
		\end{aligned}
	\end{equation}
	
	It remains to prove the $C^N$-estimates \eqref{principal c1 est} and \eqref{dprincipal c1 est} of perturbations.
	By
    Lemmas \ref{buildingblockestlemma}, \ref{Lem-gk-esti}, \ref{mae} and \ref{vae},
	\begin{align} \label{wprincipal c1 est}
		\norm{ w_{q+1}^{(p)} }_{C_{t,x}^N }
		\lesssim&   \sum_{k \in \Lambda_u \cup \Lambda_B}
		\|a_{(k)}\|_{C_{t,x}^N }
		 \sum_{0\leq N_1+N_2 \leq N} \norm{ \g}_{C_{t}^{N_1}}\norm{ W_{(k)} }_{C_{t,x}^{N_2}}  \notag \\
        \lesssim& \sum_{0\leq N_1+N_2 \leq N}  \sum_{  N_{21}+N_{22} = N_2}
         \ell^{-8N-1} \sigma^{N_1} \tau^{N_1+\frac 12} \rs^{-\frac 12} \rp^{-\frac 12} \lbb^{N_{21}} \(\frac{\rs \lbb \mu}{\rp}\)^{N_{22}} \notag \\
        \lesssim& \lambda^{\frac{5N+3}{2}},
	\end{align}
    where we also used \eqref{b-beta-ve} and \eqref{larsrp} in the last step.

	Similarly, we have
	\begin{align} \label{uc c1 est}
		&\quad \norm{ w_{q+1}^{(c)} }_{C_{t,x}^N } \notag \\
		& \lesssim   \sum_{k\in \Lambda_u \cup \Lambda_B}
		\left\|\g\( \curl (\nabla a_{(k)} \times W^c_{(k)})+ \nabla a_{(k)} \times \curl W_{(k)}^c+a_{(k)}\wt W^c_{(k)}\) \right\|_{C_{t,x}^N } \notag \\
		& \lesssim   \sum_{k \in \Lambda_u \cup \Lambda_B}
		\|a_{(k)}\|_{C_{t,x}^{N+2}}
		  \sum_{0\leq N_1+N_2 \leq N} \norm{ \g}_{C_{t}^{N_1}}
         \(\norm{ W^c_{(k)} }_{C_{t,x}^{N_2}} + \norm{ \nabla W^c_{(k)} }_{C_{t,x}^{N_2}}
		+ \norm{\wt W^c_{(k)}}_{C^{N_2}_{t,x}} \)   \nonumber \\
        & \lesssim   \sum_{0\leq N_1+N_2 \leq N}
                     \ell^{-8N-1} \sigma^{N_1} \tau^{N_1+\frac 12} \rs^{-\frac 12} \rp^{-\frac 12} \(\frac{\rs \lbb \mu}{\rp}\)^{N_2}
                    \(\ell^{-16}\lbb^{-1} + \ell^{-16} \frac{\rs}{\rp}\) \notag \\
		& \lesssim  \lambda^{\frac{5N+3}{2}}.
	\end{align}
	Moreover, by Sobolev's embedding $W^{1,6}(\mathbb{T}^3) \hookrightarrow L^\9(\mathbb{T}^3)$
	and the boudedness of operators $\mathbb{P}_H \mathbb{P}_{\not =0}$ in the space $L^6_x$,
	\begin{align} \label{ut c1 est}
		\norm{ w_{q+1}^{(t)} }_{C_{t,x}^N }
         & \lesssim \mu^{-1} \sum_{k \in \Lambda_u \cup \Lambda_B}
           \|a_{(k)}^2 g_{(\tau)}^2 \psi^2_{(k_1)} \phi^2_{(k)}\|_{C^N_t W^{N+1,6}_x} \notag \\
		 & \lesssim \mu^{-1}  \sum_{k \in \Lambda_u \cup \Lambda_B}
           \sum_{0\leq N_1+N_2 \leq N+1} \|a_{(k)}^2\|_{C_{t,x}^N }
		\|\g^2\psi^2_{(k_1)} \phi^2_{(k)}\|_{C_{t}^{N_1} C^{N_2}_x} \notag \\
        & \leq \frac13 \lambda^{\frac{5N}{2}+3},
	\end{align}
    where we also used Lemmas \ref{buildingblockestlemma}, \ref{Lem-gk-esti}, \ref{mae} and \ref{vae} in the last step.

	Finally, arguing as above we get
	\begin{align} \label{wo c1 est.2}
		\norm{ \wo }_{C_{t,x}^N }
        & \lesssim \sigma^{-1} \sum_{k \in \Lambda_u \cup \Lambda_B} \|h_{(\tau)} \na (a_{(k)}^2) \|_{C^N_tW^{N+1,6}_x} \notag \\
		& \lesssim \sigma^{-1} \sum_{k \in \Lambda_u \cup \Lambda_B} \|h_{(\tau)}\|_{C_{t}^{N+1}} \|\nabla (a^2_{(k)})\|_{C_{t,x}^{N+1}} \notag \\
		& \lesssim \sigma^N \tau^{N+1} \ell^{-8(N+1)-2}   \notag \\
        & \lesssim \lbb^{2\ve N} \lbb^{(1-8\ve)(N+1)-2} \lbb^{\frac{84}{100}\ve}
          \lesssim \lbb^{N+1},
	\end{align}
    where the last two steps were due to \eqref{b-beta-ve} and \eqref{larsrp}. 	
	
	Thus, combining \eqref{wprincipal c1 est}-\eqref{wo c1 est.2} altogether and using \eqref{larsrp}
	we conclude that
	\begin{align} \label{utotalc1}
		& \norm{ w_{q+1}^{(p)} }_{C_{t,x}^N }  + \norm{ w_{q+1}^{(c)} }_{C_{t,x}^N }+\norm{ w_{q+1}^{(t)} }_{C_{t,x}^N }+\norm{ \wo }_{C_{t,x}^N }
          \leq \frac 12 \lbb^{\frac{5N}{2}+3}.
	\end{align}
	The $C^N$-estimates of the magnetic perturbations in \eqref{dprincipal c1 est} can be proved similarly.
\end{proof}

\subsection{Verification of inductive estimates \eqref{ubc}, \eqref{u-B-L2tx-conv} and \eqref{u-B-L1L2-conv}}  \label{Subsec-induc-vel-mag}

We are now in position to verify the inductive estimates \eqref{ubc}, \eqref{u-B-L2tx-conv} and \eqref{u-B-L1L2-conv}
for the velocity and magnetic fields.

Note that,
the previous estimate \eqref{uprinlp} cannot yield the decay properties required by the
inductive estimates  \eqref{u-B-L2tx-conv} and \eqref{u-B-L1L2-conv}.
An important ingredient here is the $L^p$ decorrelation Lemma \ref{Decorrelation1} below,
which permits to derive the decay of the $L^2_{t,x}$-norm of  principal parts.

\begin{lemma}[\cite{cl21}, Lemma 2.4; see also \cite{bv19b},  Lemma 3.7]   \label{Decorrelation1}
	Let $\sigma\in \mathbb{N}$ and $f,g:\mathbb{T}^d\rightarrow \R$ be smooth functions. Then for every $p\in[1,+\9]$,
	\begin{equation}\label{lpdecor}
		\big|\|fg(\sigma\cdot)\|_{L^p(\T^d)}-\|f\|_{L^p(\T^d)}\|g\|_{L^p(\T^d)} \big|\lesssim \sigma^{-\frac{1}{p}}\|f\|_{C^1(\T^d)}\|g\|_{L^p(\T^d)}.
	\end{equation}
\end{lemma}

We  apply Lemma~\ref{Decorrelation1} with $f= a_{(k)}$, $g = \g\psi_{(k_1)}\phi_{(k)}$ and $\sigma = \lambda^{2\ve}$.
By \eqref{la}, \eqref{b-beta-ve} and Lemmas \ref{mae} and \ref{vae},
\begin{align}
	\label{Lp decorr vel}
	\norm{w^{(p)}_{q+1}}_{L^2_{t,x}}
	&\lesssim \sum\limits_{k\in \Lambda_u \cup \Lambda_B}
       \Big(\|a_{(k)}\|_{L^2_{t,x}}\norm{ \g }_{L^2_{t}} \norm{ \psi_{(k_1)}\phi_{(k)}}_{C_tL^2_{x}} \notag\\
      &\qquad\qquad \qquad  +\sigma^{-\frac12}\|a_{(k)}\|_{C^1_{t,x}}\norm{ \g }_{L^2_{t}} \norm{  \psi_{(k_1)}\phi_{(k)}}_{C_tL^2_{x}}\Big) \notag\\
	&\lesssim  \delta_{q+1}^{\frac{1}{2}}+\ell^{-8}\lambda^{-\ve}_{q+1}   \lesssim \delta_{q+1}^{\frac{1}{2}}, \\
	\label{Lp decorr mag}
	\norm{d^{(p)}_{q+1}}_{L^2_{t,x}}
	&\lesssim \sum\limits_{k \in \Lambda_B}
      \Big(\|a_{(k)}\|_{L^2_{t,x}}\norm{ \g }_{L^2_{t}} \norm{  \psi_{(k_1)}\phi_{(k)}}_{C_tL^2_{x}}\notag\\
      &\qquad\qquad\quad
       +\sigma^{-\frac12}\|a_{(k)}\|_{C^1_{t,x}}\norm{ \g }_{L^2_{t}} \norm{  \psi_{(k_1)}\phi_{(k)}}_{C_tL^2_{x}} \Big) \notag\\
       &\lesssim  \delta_{q+1}^{\frac{1}{2}}+\ell^{-4}\lambda^{-\ve}_{q+1}   \lesssim \delta_{q+1}^{\frac{1}{2}}.
\end{align}

Below we verify the iterative estimates for $u_{q+1}$ and $B_{q+1}$.
For the velocity perturbation $w_{q+1}$,
using \eqref{velocity perturbation}, \eqref{Lp decorr vel} and Lemma \ref{totalest}
we have
\begin{align}  \label{e3.41}
	\norm{w_{q+1}}_{L^2_{t,x}} &\lesssim\norm{w_{q+1}^{(p)} }_{L^2_{t,x}} + \norm{ w_{q+1}^{(c)} }_{L^2_{t,x}} +\norm{ w_{q+1}^{(t)} }_{L^2_{t,x}}+\norm{ \wo }_{L^2_{t,x}}\notag \\
	&\lesssim \delta_{q+1}^{\frac{1}{2}} +\ell^{-1}r_{\perp} r_{\parallel}^{-1}
     + \ell^{-2} \mu^{-1} r_{\perp}^{-\frac{1}{2}} r_{\parallel}^{-\frac12} \tau^\frac 12 + \ell^{-10}\sigma^{-1}\lesssim \delta_{q+1}^{\frac{1}{2}},
\end{align}
and
\begin{align}  \label{wql1}
	\norm{w_{q+1}}_{L^1_tL^2_x} &\lesssim\norm{w_{q+1}^{(p)} }_{L^1_tL^2_x} + \norm{ w_{q+1}^{(c)} }_{L^1_tL^2_x} +\norm{ w_{q+1}^{(t)} }_{L^1_tL^2_x}+\norm{ \wo }_{L^1_tL^2_x}\notag \\
	&\lesssim \ell^{-1}\tau^{-\frac12}+\ell^{-1} r_{\perp} r_{\parallel}^{-1} \tau^{-\frac12}+ \ell^{-2} \mu^{-1} r_{\perp}^{-\frac{1}{2}} r_{\parallel}^{-\frac12} + \ell^{-10}\sigma^{-1}\lesssim \lambda_{q+1}^{-\ve} ,
\end{align}
where we also used \eqref{b-beta-ve} in the last steps.

Then, in view of \eqref{ubc},  \eqref{principal c1 est},
\eqref{q+1 velocity} and the standard mollification estimates,
we get for $a$ sufficiently large,
\begin{align}
	& \norm{u_{q+1}}_{C^1_{t,x}} \lesssim \norm{u_{\ell}}_{C^1_{t,x}}+\norm{w_{q+1}}_{C^1_{t,x}}\lesssim  \lambda_{q}^7+ \lambda_{q+1}^{\frac{13}{2}}\leq\lambda_{q+1}^7, \label{verifyuc1} \\
	& \norm{u_{q} - u_{q+1}}_{L^2_{t,x}}  \leq \norm{ u_{q} - u_{\ell}  }_{L^2_{t,x}} + \norm{u_{\ell}  - u_{q+1}}_{L^2_{t,x}} \nonumber   \\
	&\qquad \qquad \qquad \  \lesssim  \norm{ u_q - u_{\ell} }_{C_{t,x}}+ \norm{w_{q+1}}_{L^2_{t,x}}  \nonumber  \\
	&\qquad \qquad \qquad \  \lesssim  \ell \norm{u_q }_{C_{t,x}^1}+\delta_{q+1}^{\frac{1}{2}}   \leq M\delta_{q+1}^{\frac{1}{2}}, \label{e3.43}
\end{align}
and
\begin{align}  \label{uql1l2}
	\norm{u_{q} - u_{q+1}}_{L^1_tL^2_x} 	
	  \lesssim& \norm{ u_q - u_{\ell} }_{C_{t,x}}+ \norm{w_{q+1}}_{L^1_tL^2_x} \nonumber \\
	  \lesssim& \ell \lambda_q^7+\lambda_{q+1}^{-\ve} \leq \delta_{q+2}^{\frac{1}{2}},
\end{align}
where we also used the inequalities
$\lbb_q^{-13} \ll \delta_{q+1}^{\frac 12}$
and $\ell \lbb_q^7 \ll \delta_{q+2}^{\frac 12}$,
due to \eqref{b-beta-ve} and $M$ is a fixed large universal constant.

Similarly, for the magnetic perturbation $d_{q+1}$,
by \eqref{magnetic perturbation}, \eqref{Lp decorr mag} and Lemma~\ref{totalest},
\begin{align}  \label{e3.40}
	\norm{d_{q+1} }_{L^2_{t,x}}  & \leq  \norm{ d_{q+1}^{(p)}}_{L^2_{t,x}} +\norm{ d_{q+1}^{(c)} }_{L^2_{t,x}}+\norm{ d_{q+1}^{(t)} }_{L^2_{t,x}}+\norm{ \dqo }_{L^2_{t,x}}  \nonumber   \\
	& \lesssim \delta_{q+1}^{\frac{1}{2}} +\ell^{-1}r_{\perp} r_{\parallel}^{-1}+ \ell^{-2} \mu^{-1}r_{\perp}^{-\frac{1}{2}} r_{\parallel}^{-\frac{1}{2}}\tau^{\frac12} +\ell^{-6}\sigma^{-1} \lesssim \delta_{q+1}^{\frac{1}{2}},
\end{align}
and
\begin{align}  \label{dql1}
	\norm{d_{q+1}}_{L^1_tL^2_x} &\lesssim\norm{d_{q+1}^{(p)} }_{L^1_tL^2_x} + \norm{ d_{q+1}^{(c)} }_{L^1_tL^2_x} +\norm{d_{q+1}^{(t)} }_{L^1_tL^2_x}+ \|\dqo \|_{L^1_tL^2_x}\notag \\
	&\lesssim \ell^{-1}\tau^{-\frac12}+\ell^{-1} r_{\perp} r_{\parallel}^{-1} \tau^{-\frac12}
       + \ell^{-2} \mu^{-1} r_{\perp}^{-\frac{1}{2}} r_{\parallel}^{-\frac12} + \ell^{-6}\sigma^{-1}\leq \lambda_{q+1}^{-\ve} ,
\end{align}
which along with \eqref{ubc}, \eqref{q+1 magnetic} and Lemma \ref{totalest} yield that
\begin{align}
     \norm{B_{q+1}}_{C^1_{t,x}}& \lesssim \norm{B_{\ell}}_{C^1_{t,x}}+\norm{d_{q+1}}_{C^1_{t,x}}\lesssim  \lambda_{q}^7+ \lambda_{q+1}^{\frac{13}{2}}\leq\lambda_{q+1}^7,  \label{verifybc1} \\
    \norm{B_{q} - B_{q+1} }_{L^2_{t,x}}
	& \lesssim \norm{ B_q - B_{\ell} }_{C_{t,x}}+ \norm{d_{q+1}}_{L^2_{t,x}} \nonumber \\
	&\lesssim \ell \norm{B_q }_{C_{t,x}^1}+ \delta_{q+1}^{\frac{1}{2}}   \leq M\delta_{q+1}^{\frac{1}{2}},   \label{e3.42}
\end{align}
and
\begin{align}  \label{bql1l2}
	\norm{B_{q} - B_{q+1}}_{L^1_tL^2_x} 	& \lesssim \norm{ B_q - B_{\ell} }_{C_{t,x}}+ \norm{d_{q+1}}_{L^1_tL^2_x}
	 \lesssim \ell \lambda_q^7+\lambda_{q+1}^{-\ve}  \leq \delta_{q+2}^{\frac{1}{2}}.
\end{align}
Therefore, the inductive estimates in \eqref{ubc}, \eqref{u-B-L2tx-conv} and \eqref{u-B-L1L2-conv} are verified.

\section{Reynolds and magnetic stresses}   \label{Sec-Rey-mag-stress}

The main purpose of this section is to determine the Reynolds and magnetic stresses
and to prove the main iteration in Theorem \ref{Prop-Iterat}.

The important roles here are played by the inverse-divergence operators $\mathcal{R}^u$ and $\mathcal{R}^B$, defined by
\begin{align}
	& (\mathcal{R}^u v)^{kl} := \partial_k \Delta^{-1} v^l + \partial_l \Delta^{-1} v^k - \frac{1}{2}(\delta_{kl} + \partial_k \partial_l \Delta^{-1})\div \Delta^{-1} v,  \\
	& (\mathcal{R}^Bf)_{ij} :=  \varepsilon_{ijk} (-\Delta)^{-1}(\curl f)_k,\label{operarb}
\end{align}
where $\int_{\mathbb{T}^3} v dx =0$, $\div f=0$,
and $\varepsilon_{ijk}$ is the Levi-Civita tensor, $i,j,k,l \in \{1,2,3\}$.

The operator $\mathcal{R}^u$ returns symmetric and trace-free matrices,
while the operator $\mathcal{R}^B$ returns skew-symmetric matrices.
Moreover, one has the algebraic identities
\begin{align*}
   \div \mathcal{R}^u(v) = v,\ \ \div \mathcal{R}^B(f) = f.
\end{align*}
Both
$|\nabla|\mathcal{R}^u$ and $|\nabla|\mathcal{R}^B$ are Calderon-Zygmund operators
and thus they are bounded in the spaces $L^p$, $1<p<+\infty$.
See \cite{bbv20,dls13} for more details.

\subsection{Decomposition of magnetic stress}

Using \eqref{mhd1} with $q+1$ replacing $q$,
\eqref{me}, \eqref{perturbation}, and \eqref{q+1 iterate}
we derive the equation for the magnetic stress $\rr^B_{q+1}$:
\begin{align}
		&\displaystyle\div\mathring{R}_{q+1}^B   \notag\\
		&=\underbrace{\partial_t (d^{(p)}_{q+ 1}+d^{(c)}_{q+ 1})+\nu_2(-\Delta)^{\alpha_2} d_{q+1}  +\div (d_{q + 1} \otimes u_{\ell} - u_{\ell} \otimes d_{q+1}+ B_{\ell}\otimes w_{q+1} -w_{q +1} \otimes B_{\ell} )}_{ \div\mathring{R}_{lin}^B  }   \notag\\
		&\quad+\underbrace{\div (d_{q+ 1}^{(p)} \otimes w_{q+1}^{(p)} -w_{q+ 1}^{(p)} \otimes  d_{q+ 1}^{(p)} + \mathring{R}_{\ell}^B)+ \partial_t d_{q+1}^{(t)}+ \partial_t \dqo}_{\div\mathring{R}_{osc}^B }  \notag\\
		&\quad+\div\Big( d_{q+1}^{(p)} \otimes (w_{q+1}^{(c)}+ w_{q+1}^{(t)}+\wo) -(w_{q+1}^{(c)}+w_{q+1}^{(t)}+\wo) \otimes d_{q+1}  \notag  \\
		&\qquad \underbrace{\qquad\quad+(d_{q+1}^{(c)}+d_{q+1}^{(t)}+\dqo)\otimes w_{q+1}-w_{q+1}^{(p)} \otimes (d_{q+1}^{(c)}+d_{q+1}^{(t)}+\dqo)\Big) }_{\div\mathring{R}_{cor}^B } \notag   \\
		&\quad+\underbrace{ \div\(B_{\ell} \otimes u_{\ell}-u_{\ell}\otimes B_{\ell}-\left(\h \otimes\u-\u \otimes \h   \right) *_{x} \phi_{\ell} *_{t} \varphi_{\ell}\)}_{\div\mathring{R}_{com}^B } .  \label{rb}
	\end{align}

Using the inverse divergence operator $\mathcal{R}^B$
we define the magnetic stress by
\begin{align}\label{rbcom}
	\mathring{R}_{q+1}^B := \mathring{R}_{lin}^B +   \mathring{R}_{osc}^B+ \mathring{R}_{cor}^B+\mathring{R}_{com}^B,
\end{align}
where
\begin{align}
	\mathring{R}_{lin}^B
	:= &  \mathcal{R}^B\(\partial_t (d^{(p)}_{q+1}+d^{(c)}_{q+1})\)
	+ \nu_2 \mathcal{R}^B (-\Delta)^{\a_2} d_{q+1} \nonumber \\
	&  +\mathcal{R}^B\P_{H}\div\( d_{q + 1} \otimes u_{\ell} - u_{\ell} \otimes d_{q+1}+ B_{\ell}\otimes w_{q+1} -w_{q +1} \otimes B_{\ell}\),\label{rbp}
\end{align}
the oscillation error
\begin{align}
	\mathring{R}_{osc}^B &:=  \sum_{k \in \Lambda_B}\mathcal{R}^B\P_{H}\P_{\neq 0}\left (\g^2 \P_{\neq 0}(D_{(k)}\otimes W_{(k)}-W_{(k)}\otimes D_{(k)} )\nabla (a_{(k)}^2)\right)\notag\\
	&\quad -\mu^{-1} \sum_{k \in \Lambda_B}\mathcal{R}^B\P_{H}\P_{\neq 0}(\p_t (a_{(k)}^2\g^2)\psi_{(k_1)}^2\phi_{(k)}^2k_2)\notag \\
	&\quad-\sigma^{-1}\sum_{k\in \Lambda_B}\mathcal{R}^B\P_{H}\P_{\neq 0}\(h_{(\tau)}\aint_{\T^3}D_{(k)}\otimes W_{(k)}-W_{(k)}\otimes D_{(k)} \d x\p_t\nabla(a_{(k)}^{2})\)\notag\\
	&\quad+ \(\sum_{k \neq k' \in \Lambda_B}+ \sum_{k \in \Lambda_u, k' \in \Lambda_B}\)\mathcal{R}^B\P_{H}\div\(a_{(k)}a_{(k')}\g^2(D_{(k')}\otimes W_{(k)}-W_{(k)}\otimes D_{(k')})\), \label{rob}
\end{align}
the corrector error
\begin{align}    \label{rbp2}
	\mathring{R}_{cor}^B := &\,\mathcal{R}^B\P_{H}\div\bigg( d_{q+1}^{(p)} \otimes (w_{q+1}^{(c)}+ w_{q+1}^{(t)}+\wo) -(w_{q+1}^{(c)}+w_{q+1}^{(t)}+\wo) \otimes d_{q+1}\notag\\
	&\qquad \qquad + (d_{q+1}^{(c)}+d_{q+1}^{(t)}+\dqo)\otimes w_{q+1} -w_{q+1}^{(p)} \otimes (d_{q+1}^{(c)}+d_{q+1}^{(t)}+\dqo) \bigg),
\end{align}
and the commutator error
\begin{align} \label{RBcom-def}
   \mathring{R}_{com}^B
   = \mathcal{R}^B \P_H \div\(B_{\ell} \otimes u_{\ell}-u_{\ell}\otimes B_{\ell}-\left(\h \otimes\u-\u \otimes \h   \right) *_{x} \phi_{\ell} *_{t} \varphi_{\ell}\).
\end{align}

\begin{remark} \rm
Since
\begin{align} \label{PHdiv-PHP0}
   \mathbb{P}_H \P_{\not =0} =\P_{\not =0}  \mathbb{P}_H,\ \
   \mathbb{P}_H \div = \P_{\not =0} \mathbb{P}_H \div,
\end{align}
by \eqref{div-wpc-dpc-0} and
the definitions \eqref{magtemcor} and \eqref{do},
\begin{align}   \label{div-dpc-dt0-0}
   \div(d^{(p)}_{q+1}+ d^{(c)}_{q+1})=0, \ \
    \div (-\Delta)^{\a_2} d_{q+1} =0.
\end{align}
Taking into account $\div \P_H =0$ we see that
the terms in \eqref{rbp}-\eqref{RBcom-def} are
in the domain of the operator $\mathcal{R}^B$.
\end{remark}

By \eqref{div-dpc-dt0-0},
\begin{align}  \label{dpc-Phdpc}
   d^{(p)}_{q+1} + d^{(c)}_{q+1} = \mathbb{P}_H (d^{(p)}_{q+1} + d^{(c)}_{q+1}), \ \
   (-\Delta)^{\a_2} d_{q+1} =  \mathbb{P}_H (-\Delta)^{\a_2} d_{q+1}.
\end{align}

We also note that,
in contrast to the nonlinearities in the velocity equation,
the nonlinear terms in the  magnetic equation is skew symmetric,
which, in particular, yields that
\begin{align} \label{divdivRB-0}
   \div (\div \mathring{R}^B_{q+1}) =0.
\end{align}

Thus, by virtue of \eqref{dpc-Phdpc} and \eqref{divdivRB-0},
we  obtain
\begin{align}
   \div \mathring{R}^B_{q+1}
    = \mathbb{P}_H \div \mathring{R}^B_{q+1}.
\end{align}

Then, using the algebraic identities
$\div \mathcal{R}^B = \Id$,
\eqref{mag oscillation cancellation calculation}, \eqref{btem} and \eqref{btemcom},
together with the facts that
$\div = \div \P_{\not =0} =  \P_{\not =0} \div$,
we obtain  that
$\mathring{R}^B_{q+1} $ satisfies the relaxation magnetic equation in \eqref{mhd1} and
the following identity holds
\begin{align} \label{calRBPHdiv-RB}
   \mathring{R}^B_{q+1}  =  \mathcal{R}^B  \mathbb{P}_H \div \mathring{R}^B_{q+1}  .
\end{align}

\subsection{Decomposition of Reynolds stress}

Concerning the Reynolds stress we compute
\begin{align}
		&\displaystyle\div\mathring{R}_{q+1}^u - \nabla P_{q+1}  \notag\\
		&\displaystyle = \underbrace{\partial_t (w_{q+1}^{(p)}+w_{q+1}^{(c)}) +\nu_1(-\Delta)^{\alpha_1} w_{q+1} +\div\big(u_{\ell} \otimes w_{q+1} + w_{q+ 1} \otimes u_{\ell} - B_{\ell} \otimes d_{q+1} - d_{q+1} \otimes B_{\ell}\big) }_{ \div\mathring{R}_{lin}^u +\nabla P_{lin} }   \notag\\
		&\displaystyle\quad+ \underbrace{\div (w_{q+1}^{(p)} \otimes w_{q+1}^{(p)} - d_{q+1}^{(p)} \otimes d_{q+1}^{(p)} +  \mathring{R}_{\ell}^u)+\partial_t w_{q+1}^{(t)}+\partial_t \wo}_{\div\mathring{R}_{osc}^u +\nabla P_{osc}}  \notag\\
		&\displaystyle\quad+\div\Big((w_{q+1}^{(c)}+ w_{q+1}^{(t)}+\wo)\otimes w_{q+1}+ w_{q+1}^{(p)} \otimes (w_{q+1}^{(c)}+ w_{q+1}^{(t)}+\wo)  \notag \\
		&\qquad \underbrace{\qquad - (d_{q+1}^{(c)}+ d_{q+1}^{(t)}+\dqo)\otimes d_{q+1}- d_{q+1}^{(p)} \otimes (d_{q+1}^{(c)}+ d_{q+1}^{(t)}+\dqo) \Big)}_{\div\mathring{R}_{cor}^u +\nabla P_{cor}}\notag\\
		&\displaystyle\quad+ \underbrace{ \div\(u_{\ell}\mathring{\otimes} u_{\ell}-B_{\ell} \mathring{\otimes}B_{\ell}-\left(\u \mathring{\otimes}\u-\h \mathring{\otimes}\h   \right) *_{x} \phi_{\ell} *_{t} \varphi_{\ell}\)}_{\div\mathring{R}_{com}^u} -\nabla P_{\ell}. \label{ru}
\end{align}

Then, using the inverse divergence operator $\mathcal{R}^u$
we define the Reynolds stress by
\begin{align}\label{rucom}
	\mathring{R}_{q+1}^u := \mathring{R}_{lin}^u +   \mathring{R}_{osc}^u+ \mathring{R}_{cor}^u+\mathring{R}_{com}^u,
\end{align}
where
\begin{align}
	\mathring{R}_{lin}^u & := \mathcal{R}^u\(\partial_t (w_{q+1}^{(p)} +w_{q+1}^{(c)}  )\)
	+ \nu_1 \mathcal{R}^u (-\Delta)^{\a_1} w_{q+1} \nonumber \\
	&\quad  + \mathcal{R}^u\P_H \div \(u_{\ell} \mathring{\otimes} w_{q+1} + w_{q+ 1}
	\mathring{\otimes} u_{\ell}- B_{\ell} \mathring{\otimes} d_{q+1} - d_{q+1} \mathring{\otimes} B_{\ell}\), \label{rup}
\end{align}
the oscillation error
\begin{align}\label{rou}
	\mathring{R}_{osc}^u :=& \sum_{k \in \Lambda_u} \mathcal{R}^u \P_H\P_{\neq 0}\left(\g^2 \P_{\neq 0}(W_{(k)}\otimes W_{(k)})\nabla (a_{(k)}^2)\right) \notag\\
	&+ \sum_{k \in \Lambda_B} \mathcal{R}^u \P_H\P_{\neq 0}\left(\g^2 \P_{\neq 0}(W_{(k)}\otimes W_{(k)}-D_{(k)}\otimes D_{(k)})\nabla (a_{(k)}^2)\right)\notag\\
	& -\mu^{-1}\sum_{k \in \Lambda_u\cup\Lambda_B}\mathcal{R}^u \P_H \P_{\neq 0}\(\p_t (a_{(k)}^2\g^2)\psi_{(k_1)}^2\phi_{(k)}^2k_1\)\notag\\
	&-\sigma^{-1}\sum_{k\in \Lambda_u}\mathcal{R}^u \P_H \P_{\neq 0}\(h_{(\tau)}\aint_{\T^3}W_{(k)}\otimes W_{(k)}\d x\p_t\nabla(a_{(k)}^{2})\)\notag\\
	&-\sigma^{-1}\sum_{k\in \Lambda_B}\mathcal{R}^u \P_H \P_{\neq 0}\( h_{(\tau)}\aint_{\T^3}W_{(k)}\otimes W_{(k)}-D_{(k)}\otimes D_{(k)}\d x\p_t\nabla(a_{(k)}^{2})\)\notag\\
	&+\sum_{k \neq k' \in \Lambda_u\cup\Lambda_B }\mathcal{R}^u \P_H \div\(a_{(k)}a_{(k')}\g^2W_{(k)}\otimes W_{(k')}\)\notag\\
	&  -\sum_{k \neq k' \in \Lambda_B}\mathcal{R}^u \P_H \div\(a_{(k)}a_{(k')}\g^2D_{(k)}\otimes D_{(k')}\),
\end{align}
the corrector error
\begin{align}
	\mathring{R}_{cor}^u &
	:= \mathcal{R}^u \P_H \div \bigg( w^{(p)}_{q+1} \mathring{\otimes} (w_{q+1}^{(c)}+w_{q+1}^{(t)}+\wo)
      + (w_{q+1}^{(c)}+w_{q+1}^{(t)}+\wo) \mathring{\otimes} w_{q+1}  \nonumber \\
	&\qquad \qquad  \qquad  - d^{(p)}_{q+1} \mathring{\otimes} (d_{q+1}^{(c)}+d_{q+1}^{(t)}+\dqo)- (d_{q+1}^{(c)}+d_{q+1}^{(t)}+\dqo) \mathring{\otimes} d_{q+1} \bigg), \label{rup2}
\end{align}
and the commutator error
\begin{align} \label{Rucom-def}
   \mathring{R}^u_{com}:=
   \mathcal{R}^u \P_H \div
    \(u_{\ell} \mathring{\otimes}u_{\ell}-B_{\ell} \mathring{\otimes} B_{\ell}
       -(\u \mathring{\otimes} \u-\h \mathring{\otimes} \h   ) *_{x} \phi_{\ell} *_{t} \varphi_{\ell}\).
\end{align}

\begin{remark} \rm
We infer from \eqref{PHdiv-PHP0}  that
the above terms in \eqref{rup}-\eqref{Rucom-def} are
in the domain of the operator $\mathcal{R}^u$.
\end{remark}

As in \eqref{div-dpc-dt0-0}, by \eqref{div-wpc-dpc-0},
the definitions \eqref{veltemcor} and \eqref{wo},
and $\div \mathbb{P}_H=0$,
\begin{align} \label{div-wpc-wt0-0}
   \div ( w^{(p)}_{q+1} + w^{(c)}_{q+1} ) =0, \ \
   \div ( w^{(t)}_{q+1} + w^{(o)}_{q+1} ) =0,
\end{align}
which yield that
\begin{align}
   w^{(p)}_{q+1} + w^{(c)}_{q+1} = \mathbb{P}_H (w^{(p)}_{q+1} + w^{(c)}_{q+1}), \ \
   (-\Delta)^{\a_1} w_{q+1} =  \mathbb{P}_H (-\Delta)^{\a_1} w_{q+1}.
\end{align}
Then, using the algebraic identities $\div \mathcal{R}^u = \Id$,
\eqref{vel oscillation cancellation calculation}, \eqref{utem} and \eqref{utemcom}
we obtain
\begin{align}
   \div \mathring{R}^u_{q+1}
   = \mathbb{P}_H \( \p_t \u+\nu_1(-\Delta)^{\alpha_1} \u+ \div(\u\otimes\u-\h\otimes\h)\),
\end{align}
which yields that $\mathring{R}^u_{q+1}$ satisfies the relaxation velocity equation in \eqref{mhd1}.

Moreover,
since $ \mathbb{P}^2_H =  \mathbb{P}_H$,
it also holds that
\begin{align} \label{calRuPHdiv-Ru}
    \mathring{R}^u_{q+1}  = \mathcal{R}^u  \mathbb{P}_H \div \mathring{R}^u_{q+1}.
\end{align}

\subsection{Estimates of magnetic stress}

The purpose of this subsection is to estimate the magnetic stress $\mathring{R}^B_{q+1}$
given by \eqref{rbcom} at level $q+1$.

Since the Calder\'{o}n-Zygmund operators are bounded in the space $L^p_x$ for any $1<p<+\9$,
we choose
\begin{align}\label{defp}
p: =\frac{2-8\varepsilon}{2-9\varepsilon}\in (1,2),
\end{align}
where $\ve$ is given by \eqref{e3.1}.
In particular,
\begin{equation}\label{setp}
	(1-4\varepsilon)(1-\frac{1}{p})=\frac{\varepsilon}{2},
\end{equation}
which, via \eqref{larsrp}, yields that
\begin{align}  \label{rs-rp-p-ve}
\rs^{\frac 1p-1}\rp^{\frac 1p-1} = \lambda^{\varepsilon},
   \quad \rs^{\frac 1p-\frac 12}\rp^{\frac 1p-\frac 12} = \lambda^{-1+5\varepsilon}.
\end{align}

Let us estimate the four parts in the decomposition of $\mathring{R}^B_{q+1}$ separately.

\paragraph{\bf Linear error.}

Note that,
by   \eqref{div free magnetic},
\begin{align}
	 & \| \mathcal{R}^B\partial_t( d_{q+1}^{(p)}+ d_{q+1}^{(c)})\|_{L_t^1L_x^p}  \nonumber \\
	\lesssim& \sum_{k \in \Lambda_B}\| \mathcal{R}^B \curl\curl\partial_t(\g a_{(k)} D^c_{(k)}) \|_{L_t^1L_x^p} \nonumber \\
	\lesssim& \sum_{k \in \Lambda_B}\| \curl\partial_t (\g a_{(k)} D^c_{(k)})  \|_{L_t^1L_x^p} \nonumber\\
	\lesssim& \sum_{k \in \Lambda_B}\Big(\| \g\|_{L^1_t}\(\|  a_{(k)} \|_{C_{t,x}^2}\| D^c_{(k)} \|_{C_t W_x^{1,p}}
            +\|  a_{(k)} \|_{C_{t,x}^1}\| \p_t D^c_{(k)} \|_{C_t W^{1,p}_x}\) \nonumber \\
	&\qquad\qquad+\| \p_t\g\|_{L_t^1}\| a_{(k)} \|_{C_{t,x}^1}\| D^c_{(k)} \|_{C_t W_x^{1,p}}\Big).
\end{align}
Then, by Lemmas  \ref{buildingblockestlemma} and \ref{mae},
\eqref{e3.1}, \eqref{larsrp} and \eqref{rs-rp-p-ve},
\begin{align}  \label{mag time derivative}
	& \| \mathcal{R}^B\partial_t( d_{q+1}^{(p)}+ d_{q+1}^{(c)})\|_{L_t^1L_x^p} \notag\\
    \lesssim& \tau^{-\frac12}(\ell^{-8}\rs^{\frac{1}{p}-\frac12}\rp^{\frac{1}{p}-\frac{1}{2}}\lambda^{-1}
        +\ell^{-4}\rs^{\frac{1}{p}+\frac12}\rp^{\frac{1}{p}-\frac{3}{2}}\mu)
        + \sigma\tau^{\frac12}\ell^{-4}\rs^{\frac{1}{p}-\frac12}\rp^{\frac{1}{p}-\frac{1}{2}}\lambda^{-1}  \notag\\
	\lesssim& \ell^{-8}(\lambda^{-\frac{3}{2}+4\ve}+ \lambda^{-2\ve})
     \lesssim  \ell^{-8}\lambda^{-2\ve}.
\end{align}

The control of the fractional viscosity $(-\Delta)^{\a_2}$
requires both the temporal and spatial intermittency of the magnetic flows.
More precisely, by Lemma \ref{totalest},
for $\a_2\in [0,1/2]$,
\begin{align} \label{Delta-d-1/2}
	\norm{ \mathcal{R}^B(-\Delta)^{\alpha_2} d_{q+1} }_{L_t^1L^p_x}
	\lesssim& \|d_{q+1}\|_{L_t^1L^p_x} \notag \\
	\lesssim& \ell^{-1} \rs^{\frac 1p - \frac 12} \rp^{\frac 1p - \frac 12}\tau^{-\frac12}+\ell^{-6}\sigma^{-1}  \notag \\
    \lesssim& \ell^{-6}\lambda^{-2\ve}
    \lesssim \ell^{-1} \lbb^{-\ve}.
\end{align}
For the stronger viscosity $(-\Delta)^{\a_2}$ with $\a_2\in ( 1/2,{5}/{4})$,
by \eqref{magnetic perturbation},
\begin{align}
	\norm{ \mathcal{R}^B(-\Delta)^{\alpha_2} d_{q+1} }_{L_t^1L^p_x} \lesssim & \norm{ \mathcal{R}^B(-\Delta)^{\alpha_2} d_{q+1}^{(p)} }_{L_t^1L^p_x}+\norm{ \mathcal{R}^B(-\Delta)^{\alpha_2} d_{q+1}^{(c)} }_{L_t^1L^p_x}\notag \\
	& +\norm{ \mathcal{R}^B(-\Delta)^{\alpha_2} d_{q+1}^{(t)} }_{L_t^1L^p_x}+\norm{ \mathcal{R}^B(-\Delta)^{\alpha_2} \dqo }_{L_t^1L^p_x}.\label{e5.17}
\end{align}
In order to control the R.H.S. of \eqref{e5.17},
using the interpolation inequality (cf. \cite{BM18}) and \eqref{uprinlp},
\begin{align}
	\norm{ \mathcal{R}^B(-\Delta)^{\alpha_2} d_{q+1}^{(p)} }_{L_t^1L^p_x}
    & \lesssim \norm{ |\na|^{2\a_1-1} d_{q+1}^{(p)} }_{L_t^1L^p_x}\notag\\
	& \lesssim  \norm{d_{q+1}^{(p)}}_{L_t^1L^p_x} ^{\frac{3-2\a_2}{2}} \norm{d_{q+1}^{(p)}}_{L_t^1W^{2,p}_x} ^{\frac{2\a_2-1}{2}}\notag\\
	& \lesssim \ell^{-1}\lbb^{2\alpha_2-1}\rs^{\frac{1}{p}-\frac12}\rp^{\frac{1}{p}-\frac{1}{2}}\tau^{-\frac12}.\label{e5.18}
\end{align}
Similarly, by Lemma \ref{totalest},
\begin{align}
	&\norm{ \mathcal{R}^B(-\Delta)^{\alpha_2} d_{q+1}^{(c)} }_{L_t^1L^p_x}
     \lesssim \ell^{-1}\lbb^{2\alpha_2-1}\rs^{\frac{1}{p}+\frac12}\rp^{\frac{1}{p}-\frac{3}{2}}\tau^{-\frac12},\label{e5.19}\\
	&\norm{ \mathcal{R}^B(-\Delta)^{\alpha_2} d_{q+1}^{(t)} }_{L_t^1L^p_x}
      \lesssim \ell^{-2}\lbb^{2\alpha_2-1}\mu^{-1}\rs^{\frac{1}{p}-1}\rp^{\frac{1}{p}-1},\label{e5.20}\\
	&\norm{ \mathcal{R}^B(-\Delta)^{\alpha_2} \dqo }_{L_t^1L^p_x}   \lesssim \ell^{-14}\sigma^{-1}.\label{e5.21}
\end{align}
Hence, we conclude from \eqref{e5.17}-\eqref{e5.21}
and the fact that $2\alpha_2-1 \leq  {3}/{2}-10\va$ that
\begin{align}  \label{mag viscosity}
	\norm{ \mathcal{R}^B(-\Delta)^{\alpha_2} d_{q+1} }_{L_t^1L^p_x}
	& \lesssim \ell^{-1}\lbb^{2\alpha_2-1}\rs^{\frac{1}{p}-\frac12}\rp^{\frac{1}{p}-\frac{1}{2}}\tau^{-\frac12}
	+\ell^{-1}\lbb^{2\alpha_2-1}\rs^{\frac{1}{p}+\frac12}\rp^{\frac{1}{p}-\frac{3}{2}}\tau^{-\frac12} \notag\\
	&\quad +\ell^{-2}\lbb^{2\alpha_2-1}\mu^{-1}\rs^{\frac{1}{p}-1}\rp^{\frac{1}{p}-1}
	+\ell^{-14}\sigma^{-1}\notag\\
	& \lesssim \ell^{-1}\lambda^{-\ve} .
\end{align}

Regarding the nonlinear terms in \eqref{rbp},
estimating as in \eqref{Delta-d-1/2} we have
\begin{align} \label{magnetic linear estimate1}
	&\norm{ \mathcal{R}^B\P_H\div\(d_{q + 1} \otimes u_{\ell} - u_{\ell} \otimes d_{q+1}+ B_{\ell}\otimes w_{q+1} -w_{q +1} \otimes B_{\ell}\) }_{L_t^1L^p_x}  \nonumber \\	
    \lesssim\,&\norm{d_{q + 1} \otimes u_{\ell} - u_{\ell} \otimes d_{q+1}+ B_{\ell}\otimes w_{q+1} -w_{q +1} \otimes B_{\ell} }_{L_t^1L^p_x}  \nonumber \\
	\lesssim\,& \norm{u_{\ell}}_{C_{t,x}} \norm{d_{q+1}}_{L_t^1L^p_x} +\norm{B_{\ell}}_{C_{t,x}} \norm{w_{q+1}}_{L_t^1L^p_x}  \nonumber \\
	\lesssim\, &\lambda^7_q (\ell^{-1} \rs^{\frac{1}{p}-\frac12}\rp^{\frac{1}{p}-\frac12} \tau^{-\frac 12}
                +\ell^{-6}\sigma^{-1} + \ell^{-10} \sigma^{-1} )
    \lesssim \ell^{-10}\lambda^{-2\ve}.
\end{align}

Therefore, combining \eqref{mag time derivative}, \eqref{Delta-d-1/2},
\eqref{mag viscosity}, \eqref{magnetic linear estimate1} together
and using  \eqref{b-beta-ve}
and the fact that $2\alpha_2-1\leq  {3}/{2}- 10 \va$		
we arrive at
\begin{align}   \label{magnetic linear estimate}
	\norm{\mathring{R}_{lin}^B }_{L_t^1L^p_x}
     & \lesssim \ell^{-8}\lambda^{-2\ve} +\ell^{-1}\lambda^{-\ve}+\ell^{-10}\lambda^{-2\ve}
       \lesssim \ell^{-1}\lambda^{-\ve}.
\end{align}

\paragraph{\bf Oscillation error.}
In order to treat the delicate magnetic oscillations,
we decompose
\begin{align*}
	\mathring{R}_{osc}^B = \mathring{R}_{osc.1}^B +  \mathring{R}_{osc.2}^B+  \mathring{R}_{osc.3}^B+ \mathring{R}_{osc.4}^B,
\end{align*}
where $\mathring{R}_{osc.1}^B$ contains the low-high spatial  oscillations
\begin{align*}
	\mathring{R}_{osc.1}^B
	&:=   \sum_{k \in \Lambda_B}\mathcal{R}^B \P_{H}\P_{\neq 0}\left(\g^2 \P_{\neq 0}(D_{(k)}\otimes W_{(k)}-W_{(k)}\otimes D_{(k)})\nabla (a_{(k)}^2) \right),
\end{align*}
$\mathring{R}_{osc.2}^B$ contains the high temporal oscillation
\begin{align*}
	\mathring{R}_{osc.2}^B
	&:= - \mu^{-1} \sum_{k \in \Lambda_B}\mathcal{R}^B\P_{H}\P_{\neq 0}\(\p_t (a_{(k)}^2\g^2) \psi_{(k_1)}^2\phi_{(k)}^2k_2\),
\end{align*}
 $\mathring{R}_{osc.3}^B$ is of low frequency
\begin{align*}
	\mathring{R}_{osc.3}^B &
      := -\sigma^{-1}\sum_{k\in \Lambda_B}\mathcal{R}^B\P_{H}\P_{\neq 0}
       \(h_{(\tau)}\aint_{\T^3}D_{(k)}\otimes W_{(k)}-W_{(k)}\otimes D_{(k)}\d x\p_t\nabla(a_{(k)}^{2})\),
\end{align*}
and $\mathring{R}_{osc.4}^B$ contains the interaction osillation
\begin{align*}
	\mathring{R}_{osc.4}^B &
:=\(\sum_{k \neq k' \in \Lambda_B}+ \sum_{k \in \Lambda_u, k' \in \Lambda_B}\)\mathcal{R}^B\P_{H}\div\(a_{(k)}a_{(k')}\g^2(D_{(k')}\otimes W_{(k)}-W_{(k)}\otimes D_{(k')})\).
\end{align*}
In order to estimate  $\mathring{R}_{osc.1}^B $,
the key fact is that,
the velocity and magnetic flows are of high oscillations
\begin{align*}
   \P_{\not=0} (D_{(k)} \otimes W_{(k)} - W_{(k)} \otimes D_{(k)})
    = \P_{\geq \frac 12 \lbb \rs} (D_{(k)} \otimes W_{(k)} - W_{(k)} \otimes D_{(k)}),
\end{align*}
while the amplitude function $a_{(k)}$ is slowly varying.
Hence, intuitively,
the frequency of $\mathring{R}^B_{osc.1}$ concentrates at the high mode,
and thus the inverse-divergence operator $\mathcal{R}^B$ permits to gain
a small factor  $(\lbb \rs)^{-1}$.
This is the content of Lemma \ref{commutator estimate1} below.

\begin{lemma}[\cite{lt20}, Lemma 6; see also \cite{bv19b}, Lemma B.1] \label{commutator estimate1}
	Let $a \in C^{2}\left(\mathbb{T}^{3}\right)$. For all $1<p<+\infty$ we have
	$$
	\left\||\nabla|^{-1} \P_{\neq 0}\left(a \P_{\geq k} f\right)\right\|_{L^{p}\left(\mathbb{T}^{3}\right)} \lesssim k^{-1}\left\|\nabla^{2} a\right\|_{L^{\infty}\left(\mathbb{T}^{3}\right)}\|f\|_{L^{p}\left(\mathbb{T}^{3}\right)},
	$$
	holds for any smooth function $f \in L^{p}\left(\mathbb{T}^{3}\right)$.
\end{lemma}

We apply Lemma~\ref{commutator estimate1}
with $a = \nabla (a_{(k)}^2)$ and $f =  \psi_{(k_1)}^2\phi_{(k)}^2$
to get
\begin{align}  \label{I1-esti}
	\norm{\mathring{R}_{osc.1}^B }_{L^1_tL^p_x}
	&\lesssim  \sum_{ k \in \Lambda_B}
	\|\g\|_{L^2_t}^2\norm{|\nabla|^{-1} \P_{\not =0}
		\left(\P_{\geq (\lambda \rs/2)}(D_{(k)}\otimes W_{(k)}-W_{(k)}\otimes D_{(k)} )\nabla (a_{(k)}^2)\right)}_{C_tL^p_x} \notag \nonumber  \\
	& \lesssim  \sum_{ k \in \Lambda_B} \lambda^{-1}  \rs^{-1} \norm{ \na^3(a^2_{(k)})}_{C_{t,x}}
        \norm{\psi_{(k_1)}^2\phi_{(k)}^2}_{C_tL^p_x}  \nonumber  \\
	& \lesssim \sum_{ k \in \Lambda_B} \ell^{-14}  \lambda^{-1} \rs^{-1} \norm{ \psi^2_{(k_1)}}_{C_tL^{p}_x} \norm{\phi^2_{(k)} }_{C_tL^{p}_x}  \nonumber  \\
	& \lesssim \ell^{-14}  \lambda^{-1}  \rs^{\frac{1}{p}-2}\rp^{\frac{1}{p}-1},
\end{align}
where we also used Lemmas \ref{buildingblockestlemma} and \ref{mae} in the third step.

Regarding $\mathring{R}_{osc.2}^B $,
as mentioned in Remark \ref{Rem-wtdt-correct},
the high temporal oscillations arising from $g_{(\tau)}$
can be balanced by the large parameter $\mu$,
namely, we have
\begin{align}  \label{I2-esti}
	\norm{\mathring{R}_{osc.2}^B }_{L^1_tL_x^p}
	&\lesssim  {\mu}^{-1}  \sum_{k\in\Lambda_B}\norm{|\nabla|^{-1} \P_{H} \P_{\neq 0}\(\p_t (a_{(k)}^2\g^2)\psi_{(k_1)}^2\phi_{(k)}^2k_2\)}_{L^1_tL_x^p} \nonumber  \\
	&\lesssim   {\mu}^{-1}  \sum_{k\in\Lambda_B}
	\( \norm{\p_t (a_{(k)}^2) }_{C_{t,x}}\norm{\g^2 }_{L^1_t}+ \norm{a_{(k)} }_{C_{t,x}}^2\norm{\p_t(\g^2)}_{ L_t^1 } \)
	\norm{\psi_{(k_1)}}_{C_tL^{2p}_x}^2\norm{\phi_{(k)}}_{L^{2p}_x}^2 \nonumber \\
	&\lesssim (\ell^{-6}+\ell^{-2}\tau\sigma)\mu^{-1}\rs^{\frac{1}{p}-1}\rp^{\frac{1}{p}-1}
      \lesssim  \ell^{-6}\tau\sigma\mu^{-1}\rs^{\frac{1}{p}-1}\rp^{\frac{1}{p}-1}.
\end{align}

The low  frequency part $\mathring{R}_{osc.3}^B $
can be estimated easily by \eqref{hk-est} and \eqref{mag amp estimates},
\begin{align}  \label{I3-esti}
	\norm{\mathring{R}_{osc.3}^B }_{L^1_tL^p_x}
    &\lesssim  \sigma^{-1} \sum_{k\in\Lambda_B}
     \left\|h_{(\tau)}(k_2\otimes k_1-k_1\otimes k_2)\p_t\nabla(a_{(k)}^{2})\right\|_{L^1_tL^p_x} \nonumber  \\
	&\lesssim \sigma^{-1} \sum_{k\in\Lambda_B} \|h_{(\tau)}\|_{C_t}\( \norm{a_{(k)} }_{C_{t,x}} \norm{a_{(k)} }_{C_{t,x}^2} +\norm{a_{(k)} }_{C_{t,x}^1}^2\)\nonumber \\
	&\lesssim \ell^{-9} \sigma^{-1}.
\end{align}

Finally, thanks to the small interactions between different intermittent flows,
the interaction oscillation $\mathring{R}_{osc.4}^B$ can be controlled by the product estimate in Lemma~\ref{product estimate lemma},
\begin{align}\label{I4-esti}
	\norm{\mathring{R}_{osc.4}^B}_{L^1_tL^p_x} & \lesssim \( \sum_{ k \neq k' \in \Lambda_B}+\sum_{ k \in \Lambda_u ,  k' \in \Lambda_B }\) \norm{a_{(k)}a_{(k')}\g^2(D_{(k')}\otimes W_{(k)}-W_{(k)}\otimes D_{(k')})}_{L^1_tL^p_x} \notag\\	
	& \lesssim \( \sum_{ k \neq k' \in \Lambda_B}+\sum_{ k \in \Lambda_u ,  k' \in \Lambda_B }\)\norm{a_{(k)} }_{C_{t,x}}\norm{a_{(k')} }_{C_{t,x}}  \norm{\g^2 }_{L^1_t}\|\psi_{(k_1)}\phi_{(k)}\psi_{(k'_1)}\phi_{(k')}\|_{C_tL^p_x}\notag\\
	& \lesssim  \ell^{-2}\rs^{\frac{1}{p}-1}\rp^{\frac{2}{p}-1}  \, .
\end{align}

Therefore, putting estimates \eqref{I1-esti}-\eqref{I4-esti} altogether
and using \eqref{larsrp}, \eqref{rs-rp-p-ve} we obtain
\begin{align}
	\label{magnetic oscillation estimate}
	\norm{\mathring{R}_{osc}^B}_{L_t^1L^p_x}
    &\lesssim   \ell^{-14}  \lambda^{-1} \rs^{\frac{1}{p}-2}\rp^{\frac{1}{p}-1}
	+\ell^{-6}\tau\sigma\mu^{-1}\rs^{\frac{1}{p}-1}\rp^{\frac{1}{p}-1}+\ell^{-9} \sigma^{-1}+\ell^{-2}\rs^{\frac{1}{p}-1}\rp^{\frac{2}{p}-1} \notag\\
	&\lesssim \ell^{-14}\lbb^{-\va}+\ell^{-6}\lbb^{-\frac12+3\va}+\ell^{-9} \lbb^{-2\va}+ \ell^{-2}\lbb^{-1+9\ve} \notag \\
    & \lesssim \ell^{-14} \lambda^{-\va}.
\end{align}

\paragraph{\bf Corrector error.}
We take $p_1,p_2\in(1,\9)$ such that
$1/p_1=1-\eta$ with $\eta$ small enough such that
$\eta\leq \ve/(2(2-8\ve))$,
and $1/{p_1}={1}/{p_2}+1/2$.
Then,
using H\"older's inequality,
applying Lemma \ref{totalest} to \eqref{rbp2}
and using  \eqref{Lp decorr vel}, \eqref{Lp decorr mag} we get
\begin{align}
	\norm{\mathring{R}_{cor}^B }_{L^1_{t}L^{p_1}_x}
	\lesssim& \norm{ d_{q+1}^{(p)} \otimes (w_{q+1}^{(c)}+ w_{q+1}^{(t)}+\wo) -(w_{q+1}^{(c)}+w_{q+1}^{(t)}+\wo) \otimes d_{q+1}\notag\\
		&\quad+ (d_{q+1}^{(c)}+d_{q+1}^{(t)}+\dqo)\otimes w_{q+1} -w_{q+1}^{(p)} \otimes (d_{q+1}^{(c)}+d_{q+1}^{(t)}+\dqo) }_{L^1_{t}L^{p_1}_x} \notag \\
	\lesssim& \norm{w_{q+1}^{(c)}+w_{q+1}^{(t)}+\wo }_{L^2_{t}L^{p_2}_x} (\norm{d^{(p)}_{q+1} }_{L^2_{t,x}} + \norm{d_{q+1} }_{L^2_{t,x}}) \notag \\
    &  +  (\norm{ w_{q+1}^{(p)}}_{L^2_{t,x}} + \norm{ w_{q+1}}_{L^2_{t,x}}) \norm{ d_{q+1}^{(c)}+ d_{q+1}^{(t)}+\dqo }_{L^2_{t}L^{p_2}_x}\notag  \\
     \lesssim&  \( \ell^{-1}\rs^{\frac{1}{p_2}+\frac12}\rp^{\frac{1}{p_2}-\frac32}+ \ell^{-2}\mu^{-1} \rs^{\frac{1}{p_2}-1}\rp^{\frac{1}{p_2}-1} \tau^{\frac 12} +\ell^{-10}\sigma^{-1}\) \notag \\
             & \times \(\delta_{q+1}^\frac 12 + \ell^{-1} \rs \rp^{-1} + \ell^{-2} \mu^{-1} \rs^{-\frac 12} \rp^{-\frac 12} \tau^{\frac 12}
             + \ell^{-10} \sigma^{-1}\) \notag \\
	 \lesssim& \ell^{-20}\delta_{q+1}^\frac 12  (\lambda^{-4\ve+2\eta-8\eta\ve}+\lambda^{-\ve+2\eta-8\eta\ve}+ \lambda^{-2\ve} ) \lesssim \ell^{-20} \lambda^{-\frac{\ve}{2}}, \label{magnetic corrector estimate}
\end{align}
where the last step is due to the fact that
$-\ve /2 + 2\eta - 8\eta \ve \leq 0$.

\paragraph{\bf Commutator error.}
This error can be bounded easily by using \eqref{e3.11}:
\begin{align}  \label{RB-com-esti}
   \|\mathring{R}^B_{com}\|_{L^1_tL^p_x}
   &\lesssim \|B_{\ell} \otimes u_{\ell}-u_{\ell}\otimes B_{\ell}-\left(\h \otimes\u-\u \otimes \h   \right) *_{x} \phi_{\ell} *_{t} \varphi_{\ell} \|_{L^1_tL^p_x} \notag \\
   &\lesssim \ell \lbb_q^{14}.
\end{align}

\subsection{Estimates of Reynolds stress}

This subsection treats the Reynolds stress $\rr_{q+1}^u$ given by \eqref{rucom}.

\paragraph{\bf Linear error.}
By the  $L^p$-boundness of $\mathcal{R}^u$,
\eqref{div free magnetic} and Lemmas \ref{Lem-gk-esti}, \ref{mae}-\ref{totalest},
\begin{align*}
	 &\| \mathcal{R}^u\partial_t( w_{q+1}^{(p)}+w_{q+1}^{(c)})\|_{L_t^1L_x^p}  \nonumber \\
	\lesssim& \sum_{k \in \Lambda_u\cup\Lambda_B}\| \mathcal{R}^u \curl\curl\partial_t(\g a_{(k)} W^c_{(k)}) \|_{L_t^1L_x^p}\nonumber \\
	\lesssim& \sum_{k \in \Lambda_u\cup\Lambda_B}\| \curl\partial_t (\g a_{(k)} W^c_{(k)})  \|_{L_t^1L_x^p} \nonumber\\
	\lesssim& \sum_{k \in \Lambda_u\cup\Lambda_B}\Big(\| \g\|_{L^1_t}(\|  a_{(k)} \|_{C_{t,x}^2}\| W^c_{(k)} \|_{C_tW_x^{1,p}}+\|  a_{(k)} \|_{C_{t,x}^1}\| \p_t W^c_{(k)} \|_{C_tW^{1,p}_x}) \nonumber \\
	&\qquad\qquad\qquad+\| \p_t\g\|_{L_t^1}\| a_{(k)} \|_{C_{t,x}^1}\| W^c_{(k)} \|_{C_tW_x^{1,p}}\Big)\nonumber \\
	\lesssim&
    \tau^{-\frac12}(\ell^{-16}\rs^{\frac{1}{p}-\frac12}\rp^{\frac{1}{p}-\frac{1}{2}}\lambda^{-1}
      +\ell^{-8}\rs^{\frac{1}{p}+\frac12}\rp^{\frac{1}{p}-\frac{3}{2}}\mu )
      + \sigma\tau^{\frac12}\ell^{-8}\rs^{\frac{1}{p}-\frac12}\rp^{\frac{1}{p}-\frac{1}{2}}\lambda^{-1}    \notag\\
	\lesssim& \ell^{-16}\lambda^{-2\ve}.
\end{align*}

Regarding the viscosity term $(-\Delta)^{\a_1}$ in \eqref{rup},
arguing as in the proof of \eqref{Delta-d-1/2} and \eqref{mag viscosity},
the spatial and temporal intermittencies yield that
\begin{equation}
	\label{vol viscosity}
	\begin{aligned}
		\norm{ \mathcal{R}^u(-\Delta)^{\alpha_1} w_{q+1} }_{L_t^1L^p_x}
		\lesssim  \ell^{-1}\lambda^{-\ve}.
	\end{aligned}
\end{equation}

Moreover, similarly to  \eqref{magnetic linear estimate1},
\begin{equation}
	\label{vel cross terms}
	\norm{\mathcal{R}^u \P_H \div
       \(u_{\ell} \mathring{\otimes}  w_{q+1} + w_{q + 1} \mathring{\otimes}  u_{\ell}
		- d_{q +1} \mathring{\otimes}   B_{\ell} - B_{\ell} \mathring{\otimes}  d_{q+1} \) }_{L_t^1L^p_x}
	\lesssim  \ell^{-10} \lbb^{-2\ve} .
\end{equation}

Thus, combining the above estimates together we arrive at
\begin{align}   \label{Reynolds linear estimate}
	\norm{ \mathring{R}_{lin}^u }_{L_t^1L^p_x}
   &\lesssim \ell^{-16}\lbb^{-2\ve}+\ell^{-1}\lambda^{-\ve}+ \ell^{-10} \lbb^{-2\ve}
    \lesssim \ell^{-1}\lbb^{-\va}.
\end{align}

\paragraph{\bf Oscillation error.}
For the velocity oscillation $\mathring{R}_{osc}^u $,
using \eqref{rou} we decompose
\begin{align*}
	\mathring{R}_{osc}^u  &= \mathring{R}_{osc.1}^u + \mathring{R}_{osc.2}^u+ \mathring{R}_{osc.3}^u+ \mathring{R}_{osc.4}^u,
\end{align*}
where $\mathring{R}_{osc.1}^u $ is the low-high spatial frequency part
\begin{align}\label{ulhfp}
	\mathring{R}_{osc.1}^u  &:= \sum_{k \in \Lambda_u} \mathcal{R}^u\P_{\neq 0}\(\g^2 \P_{\neq 0}(W_{(k)}\otimes W_{(k)})\nabla (a_{(k)}^2) \) \nonumber \\
	&\quad+  \sum_{k \in \Lambda_B} \mathcal{R}^u \P_{\neq 0}\( \g^2\P_{\neq 0}(W_{(k)}\otimes W_{(k)}-D_{(k)}\otimes D_{(k)})\nabla (a_{(k)}^2)\),
\end{align}
$\mathring{R}_{osc.2}^u $ contains the high temporal oscillations
\begin{align}\label{utop}
	\mathring{R}_{osc.2}^u
	&:= - {\mu}^{-1}  \sum_{k \in \Lambda_u \cup \Lambda_B}
	\mathcal{R}^u \P_{\neq 0}\(\p_t (a_{(k)}^2\g^2)\psi_{(k_1)}^2\phi_{(k)}^2k_1\),
\end{align}
 $\mathring{R}_{osc.3}^u$ is of low frequency
\begin{align}
	\mathring{R}_{osc.3}^u
    :=&- \sigma^{-1}\sum_{k\in \Lambda_u}\mathcal{R}^u \P_{\neq 0}\(h_{(\tau)}\aint_{\T^3}W_{(k)}\otimes W_{(k)}\d x\p_t\nabla(a_{(k)}^{2})\) \notag\\
	&-\sigma^{-1}\sum_{k\in \Lambda_B}\mathcal{R}^u \P_{\neq 0}\(h_{(\tau)}\aint_{\T^3}W_{(k)}\otimes W_{(k)}-D_{(k)}\otimes D_{(k)}\d x\p_t\nabla(a_{(k)}^{2})\), \label{uhtfp}
\end{align}
and $\mathring{R}_{osc.4}^u$ contains the interaction osillation
\begin{align*}
	\mathring{R}_{osc.4}^u &
	:=\sum_{k \neq k' \in \Lambda_u\cup\Lambda_B }\mathcal{R}^u \P_H \div\(a_{(k)}a_{(k')}\g^2W_{(k)}\otimes W_{(k')}\)\\
	&\quad\  -\sum_{k \neq k' \in \Lambda_B}\mathcal{R}^u \P_H \div\(a_{(k)}a_{(k')}\g^2D_{(k)}\otimes D_{(k')}\).
\end{align*}

Let us treat $\mathring{R}_{osc.1}^u$, $\mathring{R}_{osc.2}^u$,  $\mathring{R}_{osc.3}^u$ and $\mathring{R}_{osc.4}^u $ separately.
By Lemmas \ref{buildingblockestlemma}, \ref{Lem-gk-esti}, \ref{mae}, \ref{vae} and \ref{commutator estimate1},
\begin{align} \label{J1-esti}
	\norm{\mathring{R}_{osc.1}^u}_{L^1_tL_x^p}
	& \lesssim
	\sum_{k \in \Lambda_u }
      \left\|\g^2\right\|_{L^1_t}
       \left\||\na|^{-1}\P_{\neq 0}\( \P_{\neq 0}(W_{(k)}\otimes W_{(k)})\nabla (a_{(k)}^2) \)\right\|_{L^1_tL_x^p} \nonumber  \\
	&\quad+  \sum_{k \in \Lambda_B}
     \left\|\g^2\right\|_{L^1_t} \left\||\na|^{-1}\P_{\neq 0}\(  \P_{\neq 0}(W_{(k)}\otimes W_{(k)}-D_{(k)}\otimes D_{(k)})\nabla (a_{(k)}^2)\)\right\|_{L^1_tL_x^p} \notag \\
	& \lesssim  \sum_{k \in \Lambda_u \cup \Lambda_B } (\lbb \rs)^{-1} \|\na^3 (a^2_{(k)})\|_{C_{t,x}}
      \|\psi^2_{(k_1)} \phi_{(k)}^2 \|_{C_tL^p} \notag \\
	& \lesssim  \ell^{-26}\lambda^{-1}   \rs^{\frac{1}{p}-2}\rp^{\frac{1}{p}-1} .
\end{align}

Regarding the  temporal oscillation term $\mathring{R}_{osc.2}^u$,
we estimate
\begin{align}  \label{J2-esti}
	&\norm{\mathring{R}_{osc.2}^u }_{L^1_tL_x^p} \notag\\
    \lesssim&  {\mu}^{-1}
     \sum_{k \in \Lambda_u \cup \Lambda_B}
     ( \norm{\p_t (a_{(k)}^2) }_{C_{t,x}}\norm{\g^2 }_{L^1_t}+ \norm{a_{(k)} }_{C_{t,x}}^2\norm{\p_t(\g^2)}_{ L_t^1 } )\norm{\psi_{(k_1)}}_{C_tL^{2p}_x}^2\norm{\phi_{(k)}}_{L^{2p}_x}^2 \nonumber \\
	\lesssim& (\ell^{-10}+\ell^{-2}\tau\sigma)\mu^{-1}\rs^{\frac{1}{p}-1}\rp^{\frac{1}{p}-1}   \nonumber \\
     \lesssim&  \ell^{-10}\tau\sigma\mu^{-1}\rs^{\frac{1}{p}-1}\rp^{\frac{1}{p}-1}.
\end{align}

Concerning the low frequency part $\mathring{R}_{osc.3}^u $,
it holds that
\begin{align}  \label{J3-esti}
	\norm{\mathring{R}_{osc.3}^u}_{L^1_tL^p_x}
     &\lesssim  \sigma^{-1}
     \sum_{k\in\Lambda_u }\left\||\na|^{-1}\P_{\neq 0}\(h_{(\tau)}\aint_{\T^3}W_{(k)}\otimes W_{(k)}\d x\p_t\nabla(a_{(k)}^{2})\)\right\|_{L^1_tL^p_x} \nonumber  \\
	&\quad+ \sigma^{-1}\sum_{k\in\Lambda_B} \left\| |\na|^{-1}\P_{\neq 0}\(h_{(\tau)}\aint_{\T^3}W_{(k)}\otimes W_{(k)}-D_{(k)}\otimes D_{(k)}\d x\p_t\nabla(a_{(k)}^{2})\)\right\|_{L^1_tL^p_x} \nonumber  \\
	&\lesssim \sigma^{-1} \sum_{k\in\Lambda_u\cup \Lambda_B}\|h_{(\tau)}\|_{C_t} \( \norm{a_{(k)} }_{C_{t,x}} \norm{a_{(k)} }_{C_{t,x}^2} +\norm{a_{(k)} }_{C_{t,x}^1}^2\)\nonumber \\
	&\lesssim \ell^{-17} \sigma^{-1}.
\end{align}

The interaction oscillation $\mathring{R}_{osc.4}^u$ can be bounded by Lemma \ref{product estimate lemma}:

\begin{align}\label{J4-esti}
	\norm{\mathring{R}_{osc.4}^u}_{L^1_tL^p_x} &\lesssim \sum_{k \neq k' \in \Lambda_u\cup \Lambda_B}\norm{a_{(k)}a_{(k')} \g^2W_{(k)}\otimes W_{(k')}}_{L^1_tL^p_x}   \notag\\
		&\quad+\sum_{k \neq k' \in \Lambda_B}\norm{a_{(k)}a_{(k')}\g^2D_{(k)}\otimes D_{(k')}}_{L^1_tL^p_x}\notag\\
		&\lesssim \sum_{k \neq k' \in \Lambda_u\cup \Lambda_B}
		\norm{a_{(k)}}_{C_{t,x}}\norm{a_{(k')}}_{C_{t,x}}\norm{\g^2}_{L^1_t}\|\psi_{(k_1)}\phi_{(k)}\psi_{(k'_1)}\phi_{(k')}\|_{C_tL^p_x}\notag\\
		&\lesssim  \ell^{-2}\rs^{\frac{1}{p}-1}\rp^{\frac{2}{p}-1}.
\end{align}

Therefore, putting estimates \eqref{J1-esti}-\eqref{J4-esti} altogether
we arrive at
\begin{align} \label{Reynolds oscillation estimate}
	\norm{\mathring{R}_{osc}^u }_{L_t^1L^p_x}
    &\lesssim \ell^{-26}\lambda^{-1}  \rs^{\frac{1}{p}-2}\rp^{\frac{1}{p}-1}
      +\ell^{-10}\tau\sigma\mu^{-1}\rs^{\frac{1}{p}-1}\rp^{\frac{1}{p}-1}+ \ell^{-17} \sigma^{-1}+\ell^{-2}\rs^{\frac{1}{p}-1}\rp^{\frac{2}{p}-1} \notag \\
     & \lesssim \ell^{-26}\lbb^{-\va}.
\end{align}

\paragraph{\bf Corrector error.}
Take $p_1, p_2\in (1,\9)$ as in the proof of \eqref{magnetic corrector estimate}.
Applying Lemma~\ref{totalest} and using \eqref{Lp decorr vel} and \eqref{Lp decorr mag},
we have, similarly to \eqref{magnetic corrector estimate},
\begin{align}  \label{Reynolds corrector estimate}
	\norm{\mathring{R}_{cor}^u }_{L^1_{t}L^{p_1}_x} &\lesssim \norm{w_{q+1}^{(c)}+w_{q+1}^{(t)}+ w_{q+1}^{(o)}}_{L^2_{t}L^{p_2}_x}
	(\norm{w^{(p)}_{q+1} }_{L^2_{t,x}} + \norm{w_{q+1} }_{L^2_{t,x}})\notag \\
	&\quad +   (\norm{d^{(p)}_{q+1} }_{L^2_{t,x}} + \norm{d_{q+1} }_{L^2_{t,x}}) \norm{ d_{q+1}^{(c)}+ d_{q+1}^{(t)}+d_{q+1}^{(o)} }_{L^2_{t}L^{p_2}_x} \notag \\
	& \lesssim \ell^{-20} \lambda^{-\frac{\ve}{2}}.
\end{align}

\paragraph{\bf Commutator error.}
Using \eqref{e3.10} one has
\begin{align} \label{Ru-com-esti}
   \|\mathring{R}^u_{com}\|_{L^1_tL^p_x}
   &\lesssim \|u_{\ell}\mathring{\otimes} u_{\ell}-B_{\ell} \mathring{\otimes}B_{\ell}
      -\left(\u \mathring{\otimes}\u-\h \mathring{\otimes}\h   \right) *_{x} \phi_{\ell} *_{t} \varphi_{\ell} \|_{L^1_tL^p_x} \notag \\
   &\lesssim \ell \lbb_q^{14}.
\end{align}

\subsection{Proof of main iteration in Theorem \ref{Prop-Iterat}} \label{Subsec-Proof-Iter}

The inductive estimates \eqref{ubc}, \eqref{u-B-L2tx-conv} and \eqref{u-B-L1L2-conv}
of the velocity and magnetic fields
have been verified in \S \ref{Subsec-induc-vel-mag}.

Hence, it remains to verify the estimates \eqref{rubc1} and \eqref{rub}
for the stresses and to verify the inductive inclusions \eqref{suppub} and \eqref{supprub} for the temporal supports.

(i) {\bf Verification of the inductive estimates \eqref{rubc1} and \eqref{rub}.}
Regarding the inductive estimates in \eqref{rubc1},
by the identity \eqref{calRuPHdiv-Ru},
Sobolev's embedding $W^{1,6}_x \hookrightarrow L^{\9}_x$,
\begin{align*}
	\norm{\rr^u_{q+1}}_{C_tC^1_x}
    &\lesssim \norm{\mathcal{R}^u \P_H (\div \rr^u_{q+1})}_{C_tW^{2,6}_x}\\
	&\lesssim \norm{\partial_t u_{q+1}+\nu_1 (-\Delta)^{\alpha_1} u_{q+1}+\div(u_{q+1}\otimes u_{q+1}-B_{q+1}\otimes B_{q+1}) }_{C_tW^{1,6}_x}.
\end{align*}
Then, using the interpolation inequality (cf. \cite{BM18}),
\eqref{e3.9}, \eqref{principal c1 est} and \eqref{dprincipal c1 est}
we get
\begin{align}
   \norm{\rr^u_{q+1}}_{C_tC^1_x}
	\lesssim& \norm{u_{q+1}}_{C^2_{t,x}}
      + \norm{ u_{q+1}}_{C_tL^6_x}^{\frac{3-2\a_1}{4}}  \norm{ u_{q+1}}_{C_t W^{4,6}_x}^{\frac{2\a_1+1}{4}}    \notag \\
     &   + \sum\limits_{0\leq N_1+N_2\leq 2}
      \(\norm{u_{q+1}}_{C^{N_1}_{t,x}}\norm{u_{q+1}}_{C^{N_2}_{t,x}}
           + \norm{B_{q+1}}_{C^{N_1}_{t,x}}\norm{B_{q+1}}_{C^{N_2}_{t,x}}   \) \notag\\
	\lesssim& \lambda_{q+1}^{14}.  \notag
\end{align}
Similarly,
\begin{align}\notag
	\norm{\partial_t \rr^u_{q+1}}_{C_{t,x}}
	&\lesssim \norm{\partial_t^2 u_{q+1}+\nu_1 \partial_t(-\Delta)^{\alpha_1} u_{q+1}
        +\div\partial_t(u_{q+1}\otimes u_{q+1}-B_{q+1}\otimes B_{q+1}) }_{C_{t}L^6_x}   \notag\\
	&\lesssim \norm{u_{q+1}}_{C^2_{t,x}}
      + \norm{ u_{q+1}}_{C^4_{t,x}}
      +  \norm{ u_{q+1} \otimes u_{q+1}}_{C^2_{t,x}}
      +  \norm{ B_{q+1} \otimes B_{q+1}}_{C^2_{t,x}}  \notag \\
	&\lesssim \lambda_{q+1}^{14}.\notag
\end{align}

For the $C_{t, x}^1$ estimate of $\rr_{q+1}^B$,
using the identity \eqref{calRBPHdiv-RB}
Sobolev's embedding $W^{1,6}_x\hookrightarrow L^\9_x$,
the interpolation inequality
and then using \eqref{e3.9} and
\eqref{dprincipal c1 est} we obtain
\begin{align}\notag
	\norm{\rr^B_{q+1}}_{C_tC^1_x}
    & \lesssim \norm{\mathcal{R}^B\P_H (\div \rr^B_{q+1})}_{C_tW^{2,6}_x}\\
	&\lesssim \norm{\partial_t B_{q+1}+\nu_2 (-\Delta)^{\alpha_2} B_{q+1}+\div(B_{q+1}\otimes u_{q+1}-u_{q+1}\otimes B_{q+1}) }_{C_tW^{1,6}_x}\notag\\
	&\lesssim \norm{\partial_t B_{q+1}}_{{C_tC^1_x}}
       + \norm{B_{q+1}}_{C_tL^6_x}^{\frac{3-2\a_2}{4}} \norm{B_{q+1}}_{C_tW^{4,6}_x}^{\frac{2\a_2+1}{4}} \notag\\
       &\quad
        +\norm{B_{q+1}\otimes u_{q+1}- u_{q+1}\otimes B_{q+1}}_{C_tC^{2}_x}\notag\\
	&\lesssim \lambda_{q+1}^{14}, \notag
\end{align}
and, similarly,
\begin{align}\notag
	\norm{\partial_t \rr^B_{q+1}}_{C_{t,x}}
	&\lesssim \norm{\partial_t^2 B_{q+1}+\nu_2 \partial_t(-\Delta)^{\alpha_2} B_{q+1}+\div\partial_t(B_{q+1}\otimes u_{q+1}-u_{q+1}\otimes B_{q+1}) }_{C_{t}L^6_x}\notag\\
	&\lesssim \lambda_{q+1}^{14}. \notag
\end{align}

Thus,
the estimates in \eqref{rubc1} are verified.

Moreover, combining  \eqref{Reynolds linear estimate},
\eqref{Reynolds oscillation estimate},  \eqref{Reynolds corrector estimate}
and \eqref{Ru-com-esti}
we obtain
\begin{align} \label{rq1u}
	\norm{\mathring{R}_{q+1}^u }_{L^1_{t,x}}
	&\leq \norm{ \mathring{R}_{lin}^u }_{L_t^1L^p_x} + \norm{ \mathring{R}_{osc}^u}_{L_t^1L^p_x} + \norm{\mathring{R}_{cor}^u }_{L^1_{t}L^{p_1}_x}+  \norm{\mathring{R}_{com}^u }_{L_{t}^1L^p_x}  \nonumber \\
	&\lesssim   \ell^{-1}\laq^{-\varepsilon}+\ell^{-26}\laq^{-\va}+  \ell^{-20}\laq^{-\frac{\ve}{2}} +\ell\lambda_q^{14}   \nonumber  \\
	&  \leq \delta_{q+2},
\end{align}
where $p$ and $p_1$ are as in \eqref{defp} and \eqref{Reynolds corrector estimate}.

Combining  \eqref{magnetic linear estimate},
\eqref{magnetic oscillation estimate},
\eqref{magnetic corrector estimate}
and \eqref{RB-com-esti} altogether  we also get
\begin{align} \label{rq1b}
	\|\mathring{R}_{q+1}^B \|_{L^1_{t,x}}
	&\leq \| \mathring{R}_{lin}^B \|_{L_t^1L^p_x} +  \| \mathring{R}_{osc}^B\|_{L_t^1L^p_x}
       +  \|\mathring{R}_{cor}^B \|_{L^1_{t}L^{p_1}_x} +  \|\mathring{R}_{com}^B \|_{L_{t}^1L^p_x}  \nonumber  \\
	&\lesssim   \ell^{-1}\laq^{-\varepsilon}+\ell^{-14}\laq^{-\va} + \ell^{-20}\lambda_{q+1}^{-\frac{\ve}{2}} +\ell\lambda_q^{14} \nonumber  \\
	& \leq  \delta_{q+2},
\end{align}
where  the last step is due to \eqref{la} and \eqref{b-beta-ve}.

Therefore, \eqref{rq1u} and \eqref{rq1b} together verify the estimates in \eqref{rub}.

(ii) {\bf Verification of the inductive inclusions \eqref{suppub} and \eqref{supprub}.}
By  definitions,
\begin{subequations}
\begin{align}
	& \supp_t w_{q+1} \subseteq \bigcup_{k\in \Lambda_{u}\cup\Lambda_{B}}\supp_t a_{(k)} \subseteq N_{3\ell} (\supp_t \rr_{q}^u\cup \supp_t \rr_{q}^B ),  \label{e4.43}\\
	& \supp_t d_{q+1} \subseteq \bigcup_{k\in \Lambda_{B}}\supp_t a_{(k)} \subseteq N_{3\ell} ( \supp_t \rr_{q}^u\cup \supp_t \rr_{q}^B ), \label{e4.44}
\end{align}
\end{subequations}
which yield that
\begin{subequations}
\begin{align}
	& \supp_t u_{q+1} \subseteq \supp_t u_{\ell} \cup \supp_t w_{q+1}\subseteq N_{3\ell} (\supp_t u_{q} \cup \supp_t \rr_{q}^u\cup \supp_t \rr_{q}^B ),\label{suppuq} \\
	& \supp_t B_{q+1}  \subseteq \supp_t B_{\ell} \cup \supp_t d_{q+1} \subseteq N_{3\ell} (\supp_t B_{q} \cup \supp_t \rr_{q}^u\cup \supp_t \rr_{q}^B ). \label{suppbq}
\end{align}
\end{subequations}
Thus,
taking into account $3\ell \ll \delta^\frac 12 _{q+1}$
we verify the inductive inclusion in \eqref{suppub}.

Moreover, by \eqref{rucom} and \eqref{rbcom},
\begin{align}
	& \supp_t \rr^u_{q+1}\subseteq \bigcup\limits_{k\in \Lambda_u \cup \Lambda_B}\supp_t a_{(k)}
             \cup \supp_t u_l \cup \supp_t B_l  , \label{e5.14}\\
	& \supp_t \rr^B_{q+1}
      \subseteq \bigcup\limits_{k\in \Lambda_u \cup \Lambda_B}\supp_t a_{(k)}
             \cup \supp_t u_l \cup \supp_t B_l  . \label{e5.15}
\end{align}
In view of \eqref{mol} and \eqref{e4.43},
we arrive at
\begin{align}  \label{supp-Ru-RB-q+1}
	 \supp_t \rr^u_{q+1} \cup  \supp_t \rr^B_{q+1}
      \subseteq N_{3\ell}\( \supp_t u_{q}  \cup \supp_t B_{q}\cup   \supp_t \rr_{q}^u\cup \supp_t \rr_{q}^B\).
\end{align}
Since $\ell \ll \delta_{q+2}^{\frac12}$ due to \eqref{b-beta-ve},
we then verify the inductive inclusion \eqref{supprub}.

Therefore, the proof of Theorem \ref{Prop-Iterat} is complete.
\hfill $\square$

\section{Proof of Main Theorems }   \label{Sec-Proof-Main}

\subsection{Proof of Theorem~\ref{Thm-Non-gMHD}} We prove the statements $(i)$-$(v)$ in Theorem~\ref{Thm-Non-gMHD} below.
\medskip

$(i)$. In the initial step $q=0$,
we take $u_0=\tilde{u}$ and $B_0=\tilde{B}$
and define $P_0$, $\rr_{0}^u$ and $\rr_{0}^B$ in the relaxation equation \eqref{mhd1} by
\begin{align}
	& P_0 = -\frac{1}{3} (|u_0|^2 - |B_0|^2),\\
	&\mathring{R}_0^u=\mathcal{R}^u\(\p_t u_0+\nu_1(-\Delta)^{\alpha_1} u_0\) + u_0\mathring\otimes u_0-B_0\mathring\otimes B_0, \label{r0u}\\
	& \mathring{R}_0^B=\mathcal{R}^B\(\p_t B_0+\nu_2(-\Delta)^{\alpha_2} B_0\) + B_0\otimes u_0-u_0\otimes B_0. \label{r0b}
\end{align}

Below we choose $a,\,M$ sufficiently large and set
\begin{align*}
	\delta_{1}:= \max \{  \|\mathring{R}_0^u \|_{L^1_{t,x}}, \|\mathring{R}_0^B \|_{L^1_{t,x}} \}.
\end{align*}
Then, \eqref{ubc}-\eqref{rub} are satisfied at level $q=0$.  Thus,
by Theorem~\ref{Prop-Iterat},
there exists a sequence of solutions $(u_{q+1},B_{q+1},\rr_{q+1}^u,\rr_{q+1}^B)$ to \eqref{mhd1}
obeying estimates \eqref{ubc}-\eqref{rub} for all $q\geq 0$.

Then, by the interpolation inequality
and \eqref{la}, \eqref{ubc} and \eqref{u-B-L2tx-conv}, for any $\beta'\in (0,\frac{\beta}{7+\beta})$,
\begin{align}
	&\,\sum_{q \geq 0} \norm{ u_{q+1} - u_q }_{H^{\beta'}_{t,x}} +\sum_{q \geq 0} \norm{ B_{q+1} - B_q }_{H^{\beta'}_{t,x}}\notag \\
	\leq  & \, \sum_{q \geq 0} \norm{ u_{q+1} - u_q }_{L^2_{t,x}}^{1- \beta'}\norm{ u_{q+1} - u_q }_{H^1_{t,x}}^{\beta'} +
	\sum_{q \geq 0} \norm{ B_{q+1} - B_q }_{L^2_{t,x}}^{1- \beta'}\norm{ B_{q+1} - B_q }_{H^1_{t,x}}^{\beta'}\notag\\
	\lesssim  &\,  \sum_{q \geq 0} M^{1-\beta'} \delta_{q+1}^{\frac{1-\beta'}{2}}\lambda_{q+1}^{7 \beta' } \lesssim M^{1-\beta'} \delta_{1}^{\frac{1-\beta'}{2}}\lambda_{1}^{7 \beta' } +
	\sum_{ q \geq 1} M^{1-\beta'} \lambda_{q+1}^{-\beta(1 - \beta')  + 7\beta'  } <\9, \label{interpo}
\end{align}
where the last step is due to the fact that
$-\beta(1 - \beta')  + 7\beta' <0$.

Thus, $\{(u_q,B_q)\}_{q\geq 0}$ is a Cauchy sequence in $H^{\beta'}_{t,x}$
and $\lim_{q\rightarrow+\infty}(u_q,B_q)=(u,B)$ for some $u,B\in H^{\beta'}_{t,x}$.
Taking into account
$\lim_{q \to +\infty} \mathring{R}_{q}^u = \lim_{q \to +\infty} \mathring{R}_{q}^B = 0 $ in $L^1_{t,x}$
we consequently conclude that $(u,B)$ is a weak solution to \eqref{equa-gMHD} in the sense of Definition~\ref{weaksolu}.

$(ii)$.
 Concerning the regularity of $(u,B)$, we first claim that for all $q\geq 0$,
\begin{align}\label{ubc4}
	\|u_q\|_{C^4_{t,x}}\leq \lambda_{q}^{14},\ \  \|B_q\|_{C^4_{t,x}}\leq \lambda_{q}^{14}.
\end{align}
To this end,  for $a$ sufficiently large,
\eqref{ubc4} holds at level $q=0$.
For $q\geq 1$, assuming that \eqref{ubc4} is correct at level $q$,
we apply \eqref{q+1 iterate}, \eqref{principal c1 est} and \eqref{dprincipal c1 est} to get
\begin{align*}
	&	\|u_{q+1}\|_{C^4_{t,x}}\leq \|u_{\ell}\|_{C^4_{t,x}}+\|w_{q+1}\|_{C^4_{t,x}}
       \leq \lambda_{q}^{14}+ \frac{1}{2}\lambda_{q+1}^{14}
       \leq \lambda_{q+1}^{14}, \\
	&	\|B_{q+1}\|_{C^4_{t,x}}\leq \|B_{\ell}\|_{C^4_{t,x}}+\|d_{q+1}\|_{C^4_{t,x}}
        \leq \lambda_{q}^{14}+\frac{1}{2}\lambda_{q+1}^{14}
        \leq \lambda_{q+1}^{14},
\end{align*}
which yields \eqref{ubc4} by inductive arguments, as claimed.

Therefore, for any $(s,p,\gamma)\in \mathcal{A}$,
using \eqref{b-beta-ve}, \eqref{larsrp}, \eqref{ubc4}
and Lemma \ref{totalest} we have
\begin{align}
	\sum_{q \geq 0}\norm{ u_{q+1} - u_q }_{L^\gamma_tW^{s,p}_x} & \leq \sum_{q \geq 0}\norm{u_{\ell}-u_q}_{L^\gamma_tW^{s,p}_x}
      +\sum_{q \geq 0}\norm{w_{q+1}}_{L^\gamma_tW^{s,p}_x}\notag\\
	&\lesssim \sum_{q \geq 0}( \ell \|u_q\|_{C_{t,x}^4} + \ell^{-1}\lbb_{q+1}^{s}\rs^{\frac{1}{p}-\frac12}\rp^{\frac{1}{p}-\frac12}\tau^{\frac12-\frac{1}{\gamma}}
      +{\ell}^{-26}\sigma^{-1})\notag\\
	 &\lesssim \sum_{q \geq 0} (\ell \lambda_{q}^{14}+ \ell^{-1}\lbb_{q+1}^{s}\rs^{\frac{1}{p}-\frac12}\rp^{\frac{1}{p}-\frac12}\tau^{\frac12-\frac{1}{\gamma}}
         +{\ell}^{-26}\lambda_{q+1}^{-2\ve} )\notag\\
	&\lesssim \sum_{q \geq 0}  (\lambda_{q}^{-6}
        + \lambda_{q+1}^{s+\frac32-\frac{2}{p}-\frac{1}{\gamma} +\ve(\frac{8}{p}+\frac{6}{\gamma}- 6) }
        + \lbb_{q+1}^{-\ve} ).   \label{uregularity}
\end{align}
Analogous arguments also yield that
\begin{align}
	\sum_{q \geq 0}\norm{B_{q+1} - B_q }_{L^\gamma_tW^{s,p}_x}
	&\lesssim \sum_{q \geq 0}
        (\lambda_{q}^{-6}
         + \lambda_{q+1}^{s+\frac32-\frac{2}{p}-\frac{1}{\gamma} +\ve\(\frac{8}{p}+\frac{6}{\gamma}- 6\) } + \lbb_{q+1}^{-\ve}). \label{bregularity}
\end{align}
Taking into account \eqref{s-p-gamma-ve} we have
$$
s+\frac32-\frac{2}{p}-\frac{1}{\gamma} +\ve\(\frac{8}{p}+\frac{6}{\gamma}- 6\)
  \leq s+\frac32-\frac{2}{p}-\frac{1}{\gamma} +8 \ve  <0,
$$
which yields that
\begin{align}
	\sum_{q \geq 0}\norm{ u_{q+1} - u_q }_{L^\gamma_tW^{s,p}_x}+\sum_{q \geq 0}\norm{B_{q+1} - B_q }_{L^\gamma_tW^{s,p}_x} < \9. \label{ne6.6}
\end{align}
Therefore, $\{(u_q,B_q)\}_{q\geq 0}$ is a Cauchy sequence in $L^\gamma_tW^{s,p}_x \times L^\gamma_tW^{s,p}_x$
for any $(s,p,\gamma)\in \mathcal{A}$.
Using the uniqueness of weak limits we then conclude that
$u, B \in H^{\beta^\prime}_{t,x}  \cap  L^\gamma_tW^{s,p}_x$.
Thus, the regularity statement $(ii)$ is proved.

$(iii).$ Concerning the temporal support,
let
\begin{align}
   K_q:= \supp_t u_q \cup \supp_t B_q\cup \supp_t \mathring{R}^u_q\cup \supp_t \mathring{R}^B_q, \ \ q\geq 1.
\end{align}
By \eqref{r0u} and \eqref{r0b},
\begin{align}
	& \supp_t \mathring{R}^u_0 \cup \supp_t \mathring{R}^B_0
      \subseteq K_0 := \supp_t  u_0 \cup \supp_t  B_0 .
\end{align}
Moreover, by \eqref{suppub} and \eqref{supprub},
\begin{align}
	 K_{q+1} \subseteq N_{\delta_{q+2}^\frac 12} K_{q}
      \subseteq \cdots
      \subseteq N_{\sum\limits_{j=2}^{q+2}\delta_{j}^\frac 12} K_{0},
\end{align}
which along with the fact that
$\sum_{q\geq 0}\delta_{q+2}^{1/2}\leq \ve_*$ yields that
\begin{align}
	&\supp_t u\cup \supp_t B
	\subseteq \bigcup_{q\geq 0} K_q
	\subseteq N_{\ve_*} ( \supp_t \tilde{u} \cup \supp_t \tilde{B} ),
\end{align}
thereby yielding the temporal support statement $(iii)$.

$(iv).$ Using \eqref{la} and \eqref{u-B-L1L2-conv}
we get that for $a$ sufficiently large (depending on $\ve_*$),
\begin{align}
	& \norm{ u - \tilde{u} }_{L^1_tL^2_x} +\norm{ B - \tilde{B} }_{L^1_tL^2_x} \notag \\
	\leq &\,\sum_{q \geq 0} \norm{ u_{q+1} - u_q }_{L^1_tL^2_x} +\sum_{q \geq 0} \norm{ B_{q+1} - B_q }_{L^1_tL^2_x}   \notag \\
	\leq &\,2\sum_{q\geq 0} \delta_{q+2} ^{\frac12} \leq 2\sum_{q\geq 2} a^{-\beta b^q} \leq 2\sum_{q\geq 2} a^{-\beta bq}
	=\frac{2a^{-2\beta b}}{1-a^{-\beta b}}\leq \ve_* , \label{e6.5}
\end{align}
which yields the small deviation statement $(iv)$.

$(v).$ It remains to prove the small deviations of magnetic helicity.
Note that
\begin{align}\label{e6.19}
	\|\mh_{B,B}-\mh_{\wt B,\wt B }\|_{L_t^1}
   \leq\|B-B_0\|_{L^1_tL^2_{x}}\|A_0\|_{C_tL^2_x}+\|A-A_0\|_{L^2_{t,x}}\|B\|_{L^2_{t,x}},
\end{align}
where $A$ and $A_0$ are the potential fields
corresponding to $B$ and $B_0$, respectively.

Since $B_q$ is divergence free, by the Biot-Savart law,
\begin{align}\label{defaq}
A_q:=\curl (-\Delta)^{-1} B_q, \quad q\geq 0,
\end{align}
which yields that
\begin{align}\label{e6.20}
    \norm{A_0}_{C_tL^2_x} = \norm{\curl (-\Delta)^{-1} B_0}_{C_tL^2_x} \lesssim 1.
	\end{align}

Moreover, by the inductive estimate \eqref{u-B-L1L2-conv},
\begin{align}\label{e6.21}
	\|B-B_0\|_{L^1_tL^2_{x}}
         \leq& \sum_{q\geq 0}\|B_{q+1}-B_q\|_{L^1_tL^2_{x}} \leq \sum_{q\geq 0} \delta_{q+2}^{\frac12}
          \lesssim a^{-2\beta b}.
\end{align}

Regarding the $L^2$ estimate of $B$, since $\|B_0\|_{L^2_{t,x}} \lesssim 1$ and
\begin{align*}
\|B-B_0\|_{L^2_{t,x}}\leq& \sum_{q\geq 0}\|B_{q+1}-B_q\|_{L^2_{t,x}} \leq \sum_{q\geq 0}M\delta_{q+1}^{\frac12}  \lesssim 1,
\end{align*}
we infer that
\begin{align}\label{bl2esti}
	\|B\|_{L^2_{t,x}}\leq \|B_0\|_{L^2_{t,x}}+ 	\|B-B_0\|_{L^2_{t,x}}\lesssim 1.
\end{align}

Furthermore,  we have,
via \eqref{defaq},
\begin{align}\label{e6.211}
	\|A-A_0\|_{L^2_{t,x}} &\leq \sum_{q \geq 0} \norm{ A_{q+1} - A_q }_{L^2_{t,x}} \leq \sum_{q \geq 0} \norm{ \curl (-\Delta)^{-1}(B_{q+1} - B_q)   }_{L^2_{t,x}}\notag\\
	&\leq \sum_{q \geq 0} \norm{ \curl (-\Delta)^{-1}(B_{\ell} - B_q)   }_{L^2_{t,x}}+\sum_{q \geq 0} \norm{ \curl (-\Delta)^{-1}d_{q+1}   }_{L^2_{t,x}}.
\end{align}
Since $B_{\ell} - B_q$ is mean free, by the Poincar\'{e} inequality,
\begin{align}\label{e6.22}
	\sum_{q \geq 0} \norm{ \curl(-\Delta)^{-1}(B_{\ell}-B_q)}_{L^2_{t,x}}
    &\lesssim \sum_{q\geq 0} \norm{ B_{\ell}-B_q}_{L^2_{t,x}}\lesssim \sum_{q \geq 0} \ell\|B_q\|_{C^1_{t,x}}\lesssim \sum_{q \geq 0} \lambda_q^{-13}\notag\\
	&\lesssim a^{-13}+(a^{13b}-1)^{-1}
     \lesssim a^{-13}.
\end{align}
Moreover, by \eqref{div free magnetic},
\eqref{magnetic perturbation}
and Lemmas \ref{buildingblockestlemma}, \ref{Lem-gk-esti},
\ref{mae} and \ref{totalest},
\begin{align}\label{e6.23}
	\norm{ \curl (-\Delta)^{-1}d_{q+1}}_{L^2_{t,x}} &\leq \sum_{k\in\Lambda_B} \norm{ \curl (a_{(k)}\g D_{(k)}^c)}_{L^2_{t,x}}+ \norm{ d_{q+1}^{(t)}}_{L^2_{t,x}}+ \norm{ d_{q+1}^{(o)}}_{L^2_{t,x}} \notag\\
	&\leq \sum_{k\in\Lambda_B} \norm{ \g}_{L^2_t}\norm{ \curl (a_{(k)} D_{(k)}^c)}_{C_tL^2_x}+ \norm{ d_{q+1}^{(t)}}_{L^2_{t,x}}+ \norm{ d_{q+1}^{(o)}}_{L^2_{t,x}} \notag\\
	&\lesssim\sum_{k\in\Lambda_B} \norm{ \g}_{L^2_t} \norm{ a_{(k)}}_{C^1_{t,x}} \norm{  D_{(k)}^c}_{C_tW^{1,2}_x}+ \norm{ d_{q+1}^{(t)}}_{L^2_{t,x}}+ \norm{ d_{q+1}^{(o)}}_{L^2_{t,x}}  \notag\\
	&\lesssim \ell^{-4}\lambda_{q+1}^{-1}+\ell^{-2}\mu^{-1}\tau^{\frac12}\rs^{-\frac12}\rp^{-\frac12}+\ell^{-6}\sigma^{-1}\lesssim \lambda_{q+1}^{-\frac{\ve}{2}}.
\end{align}
Thus, we conclude from \eqref{e6.211}-\eqref{e6.23} that
\begin{align}\label{e6.24}
	\|A-A_0\|_{L^2_{t,x}}
    \lesssim a^{-13}+ \sum_{q \geq 0}\lambda_{q+1}^{-\frac{\ve}{2}}
    \lesssim a^{-13}+\sum_{q \geq 1} (a^{\frac{\ve}{2}b})^{-q}
    \lesssim a^{-13}+ a^{-\frac \ve 2 b}.
\end{align}

Therefore, plugging \eqref{e6.20}, \eqref{e6.21}, \eqref{bl2esti} and \eqref{e6.24} into \eqref{e6.19}
and using \eqref{b-beta-ve}
we obtain that for $a$ large enough
\begin{align}\label{e6.25}
	\|\mh_{B,B}(t)-\mh_{\wt B,\wt B}(t)\|_{L_t^1}
     &\lesssim (a^{-2\beta b} + a^{-13}+ a^{-\frac \ve 2 b })
     \leq  \ve_*.
\end{align}

Consequently, the proof of Theorem~\ref{Thm-Non-gMHD} is complete.

\subsection{Proof of Corollary \ref{Cor-Nonuniq-Nonconserv-gMHD}}

For any $m \in \mathbb{N}_+$,
we choose the incompressible, mean-free fields $\wt u_m$ and $\wt B_m$,
defined by
\begin{subequations}
\begin{align}
	&\wt u_m:= {m\Psi(t)} (\sin x_{3}, 0,0)^\top, \label{defu0}\\
	&\wt B_m:= {m\Psi(t)} (\sin  x_{3}, \cos x_{3}, 0)^\top,\label{defb0}
\end{align}
\end{subequations}
where $\Psi:\T\rightarrow\R$ is any smooth cut-off function supported on the interval $[1/4,5/4]$, such that
$\Psi(t) = 1 $ for $1/2 \leq t\leq 1$ and $0\leq \Psi (t)\leq  1 $ for all $t\in \T$.

Then, for $\ve_* \leq 1/100$,
by Theorem \ref{Thm-Non-gMHD},
there exist $u_m, B_m \in H^{\beta'}_{t,x} \cap L^\gamma_tW^{s,p}_x$,
where $\beta'>0$ and $(s,p,\gamma) \in \mathcal{A}$
with $\mathcal{A}$ given by \eqref{A-regularity},
such that $(u_m, B_m)$ solves \eqref{equa-gMHD} and
satisfies the properties $(ii)$-$(v)$ in Theorem \ref{Thm-Non-gMHD}.

Straightforward computations show that
\begin{align*}
   \|\wt u_m\|_{L^1(\frac 12, 1;L^2_x)} = m\pi^{\frac32},\ \
   \|\wt B_m\|_{L^1(\frac 12, 1;L^2_x)} = \frac m2 (2\pi)^{\frac32}.
\end{align*}
Then, in view of the small deviations on average, i.e.,
\begin{align*}
    \|u_m- \wt u_m\|_{L^1(\frac 12, 1;L^2_x)} \leq \ve_*, \ \
     \|B_m- \wt B_m\|_{L^1(\frac 12, 1;L^2_x)} \leq \ve_*,
\end{align*}
we infer that for any $m\not =m'$,
\begin{align}
   \|u_m\|_{L^1(\frac 12, 1;L^2_x)} \not=  \|u_{m'}\|_{L^1(\frac 12, 1;L^2_x)}, \ \
   \|B_m\|_{L^1(\frac 12, 1;L^2_x)} \not=  \|B_{m'}\|_{L^1(\frac 12, 1;L^2_x)}.
\end{align}
In particular, this yields that
\begin{align*}
   (u_m, B_m) \not = (u_{m'}, B_{m'}),\ \  \forall m\not = m'.
\end{align*}

But since zero is not contained in the support of $\Psi$,
by the temporal support result in Theorem \ref{Thm-Non-gMHD} $(iii)$,
the temporal support of $(u_m, B_m)$ is contained in $[1/8, 11/8]$,
and so
\begin{align} \label{um-Bm-0}
   u_m(0) = B_m(0) =0, \ \ \forall m\geq 1.
\end{align}

Thus, we obtain infinitely many different weak solutions to \eqref{equa-gMHD}
with the same initial data at time zero.

Concerning the magnetic helicity,
on one hand,
we have from \eqref{um-Bm-0} that
\begin{align}  \label{HBB-0-0}
   \mathcal{H}_{B_m, B_m} (0) =0, \ \ \forall m\geq 1.
\end{align}

On the other hand, the vector potential $A_m$ corresponding to the magnetic field
$B_m$ can be computed explicitly by
\begin{align*}
	A_m(t,x)
   := {m\Psi(t)}
    (\sin x_{3} ,\cos  x_{3} , 0 ).
\end{align*}
This yields that,
for the magnetic helicity of $(u_m,B_m)$,
\begin{align*}
	\mh_{B_m,B_m}(t)= m^2 {\Psi(t)^2}(2\pi)^3,
\end{align*}
and so
\begin{align*}
	&\|\mh_{B_m,B_m}\|_{L_t^1(\frac12,1)}=\frac{m^2}{2 }(2\pi)^3.
\end{align*}
Taking into account $\ve_*\leq 1/100$
and  the small deviation of the magnetic helicity in Theorem \ref{Thm-Non-gMHD} $(v)$
we lead to
\begin{align}  \label{HBB-L1}
\|\mh_{B_m,B_m} \|_{L_t^1(\frac12,1)} >0.
\end{align}

Thus, we conclude from \eqref{HBB-0-0} and \eqref{HBB-L1} that
the magnetic helicity $\mathcal{H}_{B_m, B_m}$ is not conserved, $m\geq 1$.
Therefore, the proof is complete.
\hfill $\square$

\subsection{Proof of Theorem~\ref{Thm-iMHD-gMHD-limit}}

Let $\left\{\phi_{\varepsilon}\right\}_{\varepsilon>0}$
and $\left\{\varphi_{\varepsilon}\right\}_{\varepsilon>0}$ be two families of standard compactly support mollifiers
on $\T^{3}$ and $\T$, respectively.
Set
\begin{align} \label{un-u-Bn-B}
u_{n} :=\left(u *_{x} \phi_{\lambda_{n}^{-1}}\right) *_{t} \varphi_{\lambda_{n}^{-1}},
\quad B_{n} :=\left(u *_{x} \phi_{\lambda_{n}^{-1}}\right) *_{t} \varphi_{\lambda_{n}^{-1}}.
\end{align}

Since $(u,B)$ is a weak solution to the ideal MHD system,
we infer that there exists a mean-free function $P_{n}$ such that
\begin{equation}\label{mhd2}
	\left\{\aligned
	&\p_t u_n+\lambda_{n}^{-2\alpha_1}(-\Delta)^{\alpha_1} u_n+ \div(u_n\otimes u_n-B_n\otimes B_n)+\nabla P_n=\div  \mathring{R}_n^u ,  \\
	&\p_t B_n+\lambda_{n}^{-2\alpha_2}(-\Delta)^{\alpha_2} B_n+ \div(B_n\otimes u_n-u_n\otimes B_n)=\div \mathring{R}_n^B  , \\
	\endaligned
	\right.
\end{equation}
where the stresses $\mathring{R}_n^u$ and $\mathring{R}_n^B$ are given by
\begin{align}
   \mathring{R}_n^u
     :=& \left(u_{n} \mathring\otimes u_{n}-B_{n} \mathring\otimes B_{n}\right)
         -\left((u \mathring\otimes u-B \mathring\otimes B) *_{x} \phi_{\lambda_{n}^{-1}}\right) *_{t} \varphi_{\lambda_{n}^{-1}} \notag \\
		& +\lambda_{n}^{-2\alpha_1}\mathcal{R}^u (-\Delta)^{\alpha_1} u_{n}, \label{rnu} \\
	\mathring{R}_n^B
     :=& \left(B_{n} \otimes u_{n}-u_{n} \otimes B_{n}\right)-\left((B \otimes u-u\otimes B) *_{x} \phi_{\lambda_{n}^{-1}}\right) *_{t} \varphi_{\lambda_{n}^{-1}} \notag \\
		& +\lambda_{n}^{-2\alpha_2}\mathcal{R}^B (-\Delta)^{\alpha_2} B_{n},  \label{rnb}
\end{align}
and
\begin{align*}
   P_n:= P*_x \phi_{\lbb_n^{-1}} *_t \vf_{\lbb_n^{-1}}
         -|u_n|^2 + |B_n|^2
         + (|u|^2 - |B|^2) *_x \phi_{\lbb_n^{-1}} *_t \vf_{\lbb_n^{-1}}.
\end{align*}

Let $$ \nu_1:=\nu_{1,n}:=\lambda_n^{-2\alpha_1},\quad \nu_2:=\nu_{2,n}=\lambda_n^{-2\alpha_2}, \quad \nu_n:=(\nu_{1,n},\nu_{2,n})$$
and $$\widetilde{M}:=\max\{ \|u \|_{H^{\widetilde{\beta}}_{t,x}},\|B \|_{H^{\widetilde{\beta}}_{t,x}}\}.$$

{\bf Claim:} For $a$ sufficiently large,
$(u_n,B_n,\mathring{R}_n^u,\mathring{R}_n^B)$ satisfy the iterative estimates \eqref{ubc}-\eqref{rub} at level $q=n(\geq 1)$.

To this end,
let us start with the most delicate estimates in \eqref{rub}.
Applying the Minkowski inequality we have
\begin{align*}
	&\|u -u_{n} \|_{L^2_{t,x}} \notag \\
	\lesssim&\, \left\|\int_{\T} \int_{\T^3}|u(t,x)-u(t-s,x-y)|\phi_{\lambda_{n}^{-1}}(y) \varphi_{\lambda_{n}^{-1}}(s){\rm d}y {\rm d}s\right\|_{L^2_{t,x}}  \notag \\
    \lesssim&  \int_{\T} \int_{\T^3} \left\|\frac{u(t,x)-u(t-s,x-y)}{(|s|+|y|)^{2+\wt \beta}} \right\|_{L^2_{t,x}}
               (|s|+|y|)^{2+\wt \beta} \phi_{\lambda_{n}^{-1}}(y) \varphi_{\lambda_{n}^{-1}}(s){\rm d}y {\rm d}s \notag \\
    \lesssim&   \left\|\frac{u(t,x)-u(t-s,x-y)}{(|s|+|y|)^{2+\wt \beta}} \right\|_{L^2_{t,x}L^2_{s,y}}
               \left\|(|s|+|y|)^{2+\wt \beta} \phi_{\lambda_{n}^{-1}} \varphi_{\lambda_{n}^{-1}} \right\|_{L^2_{s,y}}.
\end{align*}
Since the Slobodetskii-type norm can be bounded by (see, e.g., \cite[Proposition 1.4]{BO13})
\begin{align}\label{e6.38}
   \left\|\frac{u(t,x)-u(t-s,x-y)}{(|s|+|y|)^{2+\wt \beta}} \right\|_{L^2_{t,x}L^2_{s,y}}
   \lesssim \|u\|_{H^{\wt \beta}_{t,x}}
\end{align}
and
\begin{align}\label{e6.39}
    \|(|s|+|y|)^{2+\wt \beta} \phi_{\lambda_{n}^{-1}} \varphi_{\lambda_{n}^{-1}} \|_{L^2_{s,y}}
    \lesssim \lambda_n^{- \wt \beta} \|(|s|+|y|)^{2+\wt \beta} \phi\varphi \|_{L^2_{s,y}}
    \lesssim \lambda_n^{- \wt \beta},
\end{align}
we obtain that
\begin{align}  \label{u-un-lbbn}
	\|u-u_{n}\|_{L^2_{t,x}}
    \lesssim&\, \lambda_{n}^{- \widetilde{\beta}} \|u\|_{H^{\widetilde{\beta}}_{t,x}}\lesssim \lambda_{n}^{- \widetilde{\beta}}\widetilde{M}.
\end{align}

Moreover,
we note that,
if $\delta_{s,y} u(t,x) :=u(t,x)-u(t-s, x-y)$,
\begin{align} \label{uu-uun-wtM}
  & \left\|u_{n} \otimes u_{n}-\left((u \otimes u) *_{x} \phi_{\lambda_{n}^{-1}}\right) *_{t} \varphi_{\lambda_{n}^{-1}}  \right\|_{L^{1}_{t,x}}  \notag \\
  \lesssim&  \| (u-u_{n}) \otimes (u-u_{n}) \|_{L^1_{t,x}}
            + \bigg\|\int_{\T\times \T^3}
              \delta_{s,y} u(t,x) \otimes \delta_{s,y} u(t,x) \phi_{\lambda_{n}^{-1}} \varphi_{\lambda_{n}^{-1}} ds dy \bigg\|_{L^1_{t,x}}  \notag \\
  \lesssim& \|u-u_n\|_{L^2_{t,x}}^2
            + \|(|s|+|y|)^{4+2\wt \beta} \phi_{\lambda_{n}^{-1}} \varphi_{\lambda_{n}^{-1}} \|_{L^\9_{s,y}}
              \left\|\frac{u(t,x)-u(t-s,x-y)}{(|s|+|y|)^{2+\wt \beta}} \right\|_{L^2_{t,x}L^2_{s,y}}^2 \notag \\
  \lesssim&   \lbb_n^{-2\wt \beta} \|u\|_{H^{\wt \beta}_{t,x}}^2
  \lesssim    \lbb_n^{-2\wt \beta}  \wt M^2.
\end{align}
Similarly,
\begin{align} \label{BB-BBn-wtM}
   \left\|B_{n} \otimes B_{n}-\left((B \otimes B) *_{x} \phi_{\lambda_{n}^{-1}}\right) *_{t} \varphi_{\lambda_{n}^{-1}}  \right\|_{L^{1}_{t,x}}
    \lesssim    \lbb_n^{-2\wt \beta}  \wt M^2.
\end{align}
Estimating as in \eqref{Delta-d-1/2} and \eqref{e5.18}
we have
\begin{align} \label{DeltaRu-L1-wtM}
   \|\lbb_n^{-2\a_1}\mathcal{R}^u (-\Delta)^{\a_1} u_n\|_{L^1_{t,x}}
   \lesssim& \lbb_n^{-2\a_1} \|\mathcal{R}^u (-\Delta)^{\a_1} u_n\|_{L^1_{t}L^2_x} \notag \\
   \lesssim& \lbb_n^{-2\a_1} \(\|u_n\|_{L^1_{t}L^2_x}
              + \|u_n\|_{L^1_{t}L^2_x}^{\frac{3-2\a_1}{2}} \|u_n\|_{L^1_{t}H^2_x}^{\frac{2\a_1-1}{2}} \)    \notag \\
   \lesssim& \lbb_n^{-2\a_1} \|u_n\|_{L^2_{t,x}}+ \lbb_n^{-1} \|u_n\|_{L^2_{t,x}} \notag \\
   \lesssim& (\lbb_n^{-2\a_1}  + \lbb_n^{-1}) \wt M.
\end{align}

Hence, we conclude from \eqref{rnu}, \eqref{uu-uun-wtM}, \eqref{BB-BBn-wtM} and \eqref{DeltaRu-L1-wtM}  that
\begin{align} \label{Ru-wtM-L1}
		\|\mathring{R}_n^u\|_{L^{1}_{t,x}}
          & \lesssim \left\|u_{n} \otimes u_{n}-\left((u \otimes u) *_{x} \phi_{\lambda_{n}^{-1}}\right) *_{t} \varphi_{\lambda_{n}^{-1}}  \right\|_{L^{1}_{t,x}} \notag \\
		&\quad + \left\|B_{n} \otimes B_n -\left((B \otimes B) *_{x} \phi_{\lambda_{n}^{-1}}\right) *_{t} \varphi_{\lambda_{n}^{-1}}  \right\|_{L^{1}_{t,x}}  \notag  \\
		&\quad + \|\lambda_{n}^{-2\alpha_1}\mathcal{R}^u (-\Delta)^{\alpha_1} u_{n}\|_{L^1_{t,x}}  \notag  \\
		&  \lesssim  (  \lambda_n^{-2\a_1}  +\lambda_{n}^{-1}) \widetilde{M}
            + \lambda_{n}^{-2\widetilde{\beta}} \widetilde{M}^{2}.
\end{align}
Analogous arguments also yield that
\begin{align}   \label{RB-wtM-L1}
  \|\mathring{R}_n^B \|_{L^{1}_{t,x}}
  \lesssim ( \lbb_n^{-2\a_2}  + \lbb_n^{-1}) \widetilde{M} + \lambda_{n}^{-2\widetilde{\beta}} \widetilde{M}^{2}.
\end{align}

Thus, estimate \eqref{rub} is verified at level $q=n$ for $\beta$ small enough,
where $\beta>0$ is as in the proof of Theorem \ref{Thm-Non-gMHD}.

Moreover, using the Sobolev embedding $H^3_{t,x} \hookrightarrow L^\9_{t,x}$
we have
\begin{align} \label{un-Bn-C1}
	\left\|u_{n}\right\|_{C_{t, x}^{1}}  + \left\|B_{n}\right\|_{C_{t, x}^{1}}
      \lesssim& \|u_n\|_{H^4_{t,x}}  +   \|B_n\|_{H^4_{t,x}}  \notag \\
      \lesssim& \lbb_n^4 (\|u\|_{L^2_{t,x}} + \|B\|_{L^2_{t,x}})
      \lesssim \lbb_n^4 \wt M,
\end{align}
which yields \eqref{rubc1} at level $q=n$.

Regarding estimate \eqref{ubc},
by \eqref{rnu} and the Sobolev embedding $W^{5,1}_{t,x} \hookrightarrow L^{\9}_{t,x}$,
\begin{align}
   \|\mathring{R}^u_n\|_{C^1_{t,x}}
   \leq&  \|u_{n} \otimes u_{n}   - (u \otimes u) *_{x} \phi_{\lambda_{n}^{-1}} *_{t} \varphi_{\lambda_{n}^{-1}} \|_{W^{6,1}_{t,x}} \notag \\
       & + \|B_{n} \otimes B_{n}   - (B \otimes B) *_{x} \phi_{\lambda_{n}^{-1}} *_{t} \varphi_{\lambda_{n}^{-1}} \|_{W^{6,1}_{t,x}}
	    +\lambda_{n}^{-2\alpha_1} \|\mathcal{R}^u (-\Delta)^{\alpha_1} u_{n}\|_{W^{6,1}_{t,x}} \notag \\
    =:& J_1 + J_2 + J_3.
\end{align}
Note that
\begin{align} \label{J1-wtM}
   J_1
   \leq& \sum\limits_{0\leq N_0+ N\leq 6} \|\partial_t^{N_0}\na^N(u_{n} \otimes u_{n})\|_{L^1_{t,x}}
        + \|\partial_t^{N_0}\na^N \(  (u \otimes u) *_{x} \phi_{\lambda_{n}^{-1}} *_{t} \varphi_{\lambda_{n}^{-1}}\)  \|_{L^1_{t,x}}  \notag \\
   \lesssim&\sum\limits_{0\leq N_{02}+N_{01}+N_1+N_2\leq 6}
              \|\partial_t^{N_{01}}\na^{N_1}u_n\|_{L^2_{t,x}}  \|\partial_t^{N_{02}}\na^{N_2}u_n\|_{L^2_{t,x}} \notag \\
       & \quad  + \sum\limits_{0\leq N_0+ N\leq 6}   \|u\|_{L^2_{t,x}}^2
               \|\na^{N} \phi_{\lambda_{n}^{-1}} \|_{L^1_t}   \|\partial_t^{N_0} \varphi_{\lambda_{n}^{-1}} \|_{L^1_x}  \notag \\
   \lesssim& \lbb_n^6 \|u\|_{L^2_{t,x}}^2
   \lesssim  \lbb_n^6 \wt M^2.
\end{align}
Similarly,
\begin{align}  \label{J2-wtM}
  J_2 \lesssim \lbb_n^6 \|B\|_{L^2_{t,x}}^2
      \lesssim \lbb_n^6 \wt M^2.
\end{align}
Moreover,
if $\a_1\in (0,  1/2]$,
we have
\begin{align}  \label{J3-wtM.1}
   J_3 \lesssim \lbb_n^{-2\a_1} \|u_n\|_{H^6_{t,x}}
       \lesssim \lbb_n^{6-2\a_1} \|u\|_{L^2_{t,x}}
        \lesssim \lbb_n^{6-2\a_1} \wt M.
\end{align}
If $\a_1\in (1/2, 5/4)$,
by the interpolation inequality,
\begin{align}  \label{J3-wtM.2}
   J_3 \lesssim& \lbb_n^{-2\a_1} \||\na|^{2\a_1-1} u_n\|_{H^6_{t,x}}
   \lesssim  \lbb_n^{-2\a_1} \|u_n\|_{L^2_{t,x}}^{1-\frac{5+2\a_1}{8}}
              \|u_n\|_{H^8_{t,x}}^{\frac{5+2\a_1}{8}}
   \lesssim \lbb_n^5 \wt M.
\end{align}
Thus, we conclude from \eqref{J1-wtM}-\eqref{J3-wtM.2} that
\begin{align}
    \|\mathring{R}^u_n\|_{C^1_{t,x}}
    \lesssim \lbb_n^6 (\wt M+ \wt M^2 ).
\end{align}

Similarly, we have
\begin{align} \label{Ru-RB-C1}
	&  \|\rr^B_{n} \|_{C_{t, x}^{1}}
   \lesssim  \lbb_n^6 (\wt M+ \wt M^2).
\end{align}

Thus, taking $a$ sufficiently large
we prove estimate \eqref{ubc}  at level $q=n$, as claimed.

Now,
we apply Theorem~\ref{Prop-Iterat}
to obtain a sequence of solutions $(u_{n,q}, B_{n,q}, \rr^u_{n,q}, \rr^B_{n,q})_{q\geq n}$ satisfying \eqref{ubc}-\eqref{rub}.
Letting $q\rightarrow+\infty$
we obtain a weak solution $(u^{(\nu_{n})},B^{(\nu_{n})}) \in H^{\beta'}_{t,x} \times H^{\beta'}_{t,x}$ to \eqref{equa-gMHD}
for some $\beta'\in (0,\beta/(7+\beta))$.

Moreover,
as in \eqref{interpo}, it holds that for  $n\geq 1$,
\begin{align} \label{uv1n-un-Bv1n-Bn-Hb}
     \left\|u^{(\nu_{n})}-u_{n}\right\|_{H^{\beta^{\prime}}_{t,x}}+ \left\|B^{(\nu_{n})}-B_{n}\right\|_{H^{\beta^{\prime}}_{t,x}}
    \leq C 	\sum_{q=n}^{+\infty}  \lambda_{q+1}^{-\beta(1-\beta^{\prime})} \lambda_{q+1}^{7 \beta^{\prime}} \leq \frac{1}{2 n}.
\end{align}

Taking $\beta'$ smaller such that
$0<\beta'<\min\{\wt \beta, \beta/(7+\beta)\} $
and using \eqref{u-un-lbbn} we obtain
\begin{align}  \label{u-un-B-Bn-Hb}
  \|u-u_n\|_{H^{\beta'}_{t,x}}  +  \|B-B_n\|_{H^{\beta'}_{t,x}}
  \lesssim& \|u-u_n\|_{L^2_{t,x}}^{1-\frac{\beta'}{\wt \beta}}  \|u-u_n\|_{H^{\wt \beta}_{t,x}}^{\frac{\beta'}{\wt \beta}}
            + \|B-B_n\|_{L^2_{t,x}}^{1-\frac{\beta'}{\wt \beta}}  \|B-B_n\|_{H^{\wt \beta}_{t,x}}^{\frac{\beta'}{\wt \beta}}  \notag \\
  \lesssim& \lbb_n^{-(\wt \beta- \beta')} \wt M
  \leq \frac{1}{2n}.
\end{align}

Thus,
we conclude from  \eqref{uv1n-un-Bv1n-Bn-Hb} and \eqref{u-un-B-Bn-Hb}
that for $n\geq 1$,
\begin{align*}
	& \left\|u^{(\nu_{n})}-u\right\|_{H^{\beta^{\prime}}_{t,x}}
	\leq\left\|u^{(\nu_{n})}-u_{n}\right\|_{H^{\beta^{\prime}}_{t,x}}+\left\|u-u_{n}\right\|_{H^{\beta^{\prime}}_{t,x}} \leq \frac{1}{n}, \\
	& \left\|B^{(\nu_{n})}-B\right\|_{H^{\beta^{\prime}}_{t,x}} \leq\left\|B^{(\nu_{n})}-B_{n}\right\|_{H^{\beta^{\prime}}_{t,x}}+\left\|B-B_{n}\right\|_{H^{\beta^{\prime}}_{t,x}} \leq \frac{1}{n},
\end{align*}
which converge to zero,
thereby yielding \eqref{convergence}.

Therefore, the proof of Theorem \ref{Thm-iMHD-gMHD-limit} is complete. \hfill $\square$

\begin{remark}   \label{Rem-IMHD-Hbeta}
We close this section with the remark that,
the weak solutions constructed in \cite{bbv20} to the ideal MHD equations \eqref{equa-IMHD}
is in the space $H^{\wt \beta}_{t,x}$ for some $\wt \beta>0$.

Though the solutions constructed in \cite{bbv20} are on the time interval $[0,T]$,
with slight modifications the proof in \cite{bbv20} also applies to the time interval $\T$.
Hence we mainly consider the time interval $\T$ below.

The constructed solution $(u,B)\in C_tH^{\wt \beta}_x \times C_tH^{\wt \beta}_x$ in \cite{bbv20},
where $\wt \beta>0$,  can be approximated by
the relaxation solutions $(u_q,B_q,\rr_{q}^u,\rr_{q}^B)$ to \eqref{mhd1}
with $\nu_1=\nu_2=0$,
which satisfy the inductive estimates that for any $q\geq 0$
(see \cite[(2.3),(2.4),(2.5)]{bbv20})
\begin{align*}
	& \|\u\|_{C_tL_{x}^{2}} \leq  1-\delta^{\frac12}_q,\quad \|\h\|_{C_tL_{x}^{2}} \leq  1-\delta^{\frac12}_q,  \\
	& \|\u\|_{C_{t,x}^{1}} \leq   \lambda_{q}^{2},\quad \|\h\|_{C_{t,x}^{1}} \leq  \lambda_{q}^{2}, \\
    & \|\ru\|_{C_tL^{1}_{x}} \leq  c_u \delta_{q+1},\quad \|\rb\|_{C_tL^{1}_{x}} \leq c_B \delta_{q+1},  \\
& \left\|u_{q+1}-u_{q}\right\|_{C_tL^{2}_{x}} \leq \delta_{q+1}^{\frac{1}{2}}, \quad \left\|B_{q+1}-B_{q}\right\|_{C_tL^{2}_{x}} \leq \delta_{q+1}^{\frac{1}{2}},
\end{align*}
where $\delta_{q+1}=\lambda_{q+1}^{-2\beta}$, $\wt \beta \in (0,\beta/(2+\beta))$,
$\beta$ is a small positive constant, and $c_u, c_B>0$.
Then, the interpolation inequality yields that
\begin{align}
	&\,\sum_{q \geq 0} \norm{ u_{q+1} - u_q }_{H^{\wt \beta}_{t,x}} +\sum_{q \geq 0} \norm{ B_{q+1} - B_q }_{H^{\wt \beta}_{t,x}}\notag \\
	\leq  & \, \sum_{q \geq 0} \norm{ u_{q+1} - u_q }_{L^2_{t,x}}^{1- \wt \beta}\norm{ u_{q+1} - u_q }_{H^1_{t,x}}^{\wt \beta} +
	\sum_{q \geq 0} \norm{ B_{q+1} - B_q }_{L^2_{t,x}}^{1- \wt \beta}\norm{ B_{q+1} - B_q }_{H^1_{t,x}}^{\wt \beta}\notag\\
\leq  & \, \sum_{q \geq 0} \norm{ u_{q+1} - u_q }_{C_tL^2_{x}}^{1- \wt \beta}\norm{ u_{q+1} - u_q }_{C^1_{t,x}}^{\wt \beta} +
	\sum_{q \geq 0} \norm{ B_{q+1} - B_q }_{C_tL^2_{x}}^{1- \wt \beta}\norm{ B_{q+1} - B_q }_{C^1_{t,x}}^{\wt \beta}\notag\\
	\lesssim  &\,  \sum_{q \geq 0} \delta_{q+1}^{\frac{1-\wt \beta}{2}}\lambda_{q+1}^{2 \wt \beta } <\9, \label{iinterpo}
\end{align}
which yields that $\{(u_q,B_q)\}_{q\geq 0}$ also converges in the
Sobolev space $H^{\wt \beta}_{t,x} \times H^{\wt \beta}_{t,x}$,
and thus, it follows from  the uniqueness of strong limits
that $(u,B)\in H^{\wt \beta}_{t,x} \times H^{\wt \beta}_{t,x}$.
\end{remark}

\noindent{\bf Acknowledgment.}
We would like to thank Peng Qu for giving the nice lecture on convex integration method
at SJTU in the summer 2021 and for valuable discussions.
We are grateful to Sauli Lindberg for pointing out to us
the reference \cite{fls21.2} concerning the sharpness of the $L^3_{t,x}$ integrability
condition for the Onsager-type conjecture for the magnetic helicity.
We also thank Gui-Qiang G. Chen, Fei Wang, Cheng Yu for valuable comments to improve this work.
Yachun Li thanks the support by NSFC (No. 11831011),
Deng Zhang  thanks the support by NSFC (No. 11871337)
and Shanghai Rising-Star Program 21QA1404500.
This work is also partially supported by
NSFC 12161141004 and by Institute of Modern Analysis--A Shanghai Frontier Research Center.

\end{document}